\documentclass[11pt]{article}%
\usepackage{amsfonts}
\usepackage{amsthm,enumerate}
\usepackage{subfigure}
\usepackage{ifthen,latexsym,syntonly}
\usepackage{natbib}
\usepackage{amsmath}
\usepackage{amssymb}
\usepackage{graphicx}
\usepackage{color}
\usepackage{appendix}
\usepackage[bottom]{footmisc}
\usepackage{hyperref}
\hypersetup{bookmarks=false,colorlinks,linkcolor=black,anchorcolor=black,citecolor=black}
\providecommand{\U}[1]{\protect\rule{.1in}{.1in}}
\newtheorem{theorem}{Theorem}
\newtheorem{corollary}{Corollary}

\newtheorem{lemma}{Lemma}
\newtheorem{proposition}{Proposition}

\theoremstyle{definition}
\newtheorem{definition}{Definition}

\allowdisplaybreaks
\begin{document}

% Keywords command
\providecommand{\keywords}[1]
{
  \small	
  \textbf{\textit{Keywords---}} #1
}

\title{Orthounimodal Distributionally Robust Optimization: Representation, Computation and Multivariate Extreme Event Applications}
\date{}
\author{Henry Lam \qquad Zhenyuan Liu \qquad  \\Department of Industrial
Engineering and Operations Research, Columbia University}

\author{Henry Lam\thanks{Department of Industrial
Engineering and Operations Research, Columbia University.}
\and Zhenyuan Liu\footnotemark[1]
\and Xinyu Zhang\footnotemark[1]
}

\maketitle
%\tableofcontents

%\pagebreak

\begin{abstract}
This paper studies a basic notion of distributional shape known as orthounimodality (OU) and its use in shape-constrained distributionally robust optimization (DRO). As a key motivation, we argue how such type of DRO is well-suited to tackle multivariate extreme event estimation by giving statistically valid confidence bounds on target extremal probabilities. In particular, we explain how DRO can be used as a nonparametric alternative to conventional extreme value theory that extrapolates tails based on theoretical limiting distributions, which could face challenges in bias-variance control and other technical complications. We also explain how OU resolves the challenges in interpretability and robustness faced by existing distributional shape notions used in the DRO literature. Methodologically, we characterize the extreme points of the OU distribution class in terms of what we call OU sets and build a corresponding Choquet representation, which subsequently allows us to reduce OU-DRO into moment problems over infinite-dimensional random variables. We then develop, in the bivariate setting, a geometric approach to reduce such moment problems into finite dimension via a specially constructed variational problem designed to eliminate suboptimal solutions. Numerical results illustrate how our approach gives rise to valid and competitive confidence bounds for extremal probabilities.
\end{abstract}

\keywords{multivariate extreme event analysis, orthounimodality, distributionally robust optimization, nonparametric, shape constraint}

\section{Introduction}

Distributionally robust optimization (DRO) is a methodology to tackle optimization under uncertainty that has gathered substantial attention in recent years. This methodology advocates a robust perspective that, when facing uncertain or unknown parameters in decision-making, the modeler looks for a decision that optimizes over the worst-case scenario (\cite{delage2010distributionally,goh2010distributionally,kuhn2019wasserstein}). More precisely, DRO can be considered as a special case of classical robust optimization (RO) (\cite{Bertsimas2011theory,ben2009robust}), in which the uncertain parameter is the underlying probability distribution in a stochastic problem, and the worst case is over constraints constituting a so-called uncertainty set or ambiguity set that, roughly speaking, contains the true distribution with high confidence. In this paper, we will use the term DRO broadly to refer to worst-case optimization over a class of distributions defined via the uncertainty set.

Our focus of study is a particular uncertainty set represented by a shape constraint on a multivariate probability distribution known as \emph{orthounimodality (OU)}. On a high level, this geometric property means that each marginal density of the underlying distribution is monotonically non-increasing away from its mode. OU is arguably the most basic multivariate shape constraint, and has appeared in works such as \cite{devroye1997random,biau2003risk,sager1982nonparametric,polonik1998silhouette,gao2007entropy}. Yet, a systematic investigation of its geometric properties, and the associated optimization strategies for the corresponding DRO, appears open. More specifically, we will build the mixture, or more precisely the so-called Choquet representation that allows us to represent an OU distribution as a mixture of more ``elementary" distributions which act as the extreme points in convex analysis. This subsequently allows us to reformulate the associated DRO in terms of decision variables that correspond to the mixing distribution (without the OU constraint). However, as we will demonstrate, the Choquet representation of OU differs from the range of distributional shapes discussed in the literature (both in probability theory, e.g., \cite{dharmadhikari1988unimodality}, and in DRO, e.g., \cite{van2016generalized,li2019ambiguous}), which substantially increases the complexity of the reformulated DRO to contain function-valued random variables and in turn necessitates new variational arguments to reduce the problem to a tractable form.
% Our goal in this paper is to develop a probabilistic representation theory for OU distributions, which then allows us to develop optimization methods for use in the associated DRO.

%  (ZHENYUAN: DOES THE FOLLOWING SENTENCE NEED TO BE CHANGED? I THINK THE PROPERTIES OF OU ARE INDEED NOT WELL STUDIED. ALL THE LITERATURE I LIST IN THIS PARAGRAPH DON'T STUDY ITS PROPERTIES.)

Our interest in OU is motivated from \emph{multivariate extreme event analysis} (\cite{resnick2013extreme}). This discipline studies the estimation of tail probabilities or other risk quantities from data, and is evidently at the core of risk analytics and management. A beginning well-known challenge in this task is that, by its own definition, there are little data that inform the tail of a distribution. To this end, multivariate extreme event analysis has proposed a range of statistical methods to extrapolate data to their tail, under some principled modeling assumptions. However, the challenges of these methods are also well-documented, and exacerbated especially in multivariate settings. \emph{A key motivation of our paper is to advocate DRO as a well-grounded alternative for multivariate extreme event analysis}. We show in particular how OU, in contrast to other more well-studied distributional shapes, constitutes the most natural uncertainty set for this purpose. This is argued in terms of both interpretability on the multivariate tail and robustness against the misspecification of the distributional mode, two main challenges critically faced by other established distributional shapes as we will illustrate. From a broader perspective, the power of DRO in tackling uncertainty under partial distributional information has been actively investigated (e.g., \cite{Wiesemann2014,hanasusanto2015distributionally,doan2015robustness,ghaoui2003worst}) -- Our current work follows this perspective, but also distinguishes from it by taking a specialized step to justify the choice of uncertainty sets, develop both the probability and optimization theories, in an important application that has traditionally been handled using parametric statistical models. While our Choquet theory applies to arbitrary-dimensional problems, we will focus our tractable reduction and optimization method on the bivariate setting which, as we will discuss momentarily, already forms a challenging case for traditional statistical approaches, and also requires intricate geometric arguments to handle its DRO. We hope that these developments would open the door to higher-dimensional extremal estimation problems in the future.

In the following, we first discuss the challenges in multivariate extreme event analysis  that motivates the use of our OU-DRO (Section \ref{sec:motivation}). We then introduce in detail the main geometric properties of OU that pertain to optimization (Section \ref{sec:OU}). After that, we present our main DRO formulation (Section \ref{sec:OU-DRO problem}), and our reformulation approaches and optimization methods (Section \ref{sec:opt theory}). Next we motivate and discuss the generalization of OU-DRO to situations where the considered monotonicity only holds for part of all dimensions (Section \ref{sec:POU-DRO problem}). We present some numerical examples (Section \ref{sec:numerics}). Supplemental details and all proofs are delegated to the Appendix.

% We consider optimization problems that, roughly speaking, take the form
% $$\max P((X_1,\ldots,X_d)\in S)\text{\ \ subject to\ \ density of $X_1,\ldots,X_d$ is orthogonally decreasing

\section{Motivation}\label{sec:motivation}
As discussed in the Introduction, our study is motivated from multivariate extreme event analysis. In this section, we first discuss the challenges in conventional methods in this discipline, starting with the univariate case (Section \ref{sec:challenges EEA}) which sets the stage to transit to multivariate (Section \ref{sec:challenges EEA multivariate}). We then propose DRO as an alternative approach, discuss the literature, and also present the challenges in using existing DRO formulations (Section \ref{sec:challenges DRO}). Motivated from these challenges, we finally propose OU-DRO as our solution approach (Section \ref{sec:OU}).
% conventional extreme eventexisting DRO formulations that in turn motivate us to study OU (Section \ref{sec:challenges DRO}). After that, we present our main formulation and theory (Section \ref{sec: existing approaches

\subsection{Challenges in Conventional Extreme Event Analysis: Univariate Case}\label{sec:challenges EEA}
Extreme event analysis refers to the estimation of tail probabilities, quantiles or other measures, and is a core common task in analyzing and managing risks. For instance, in the maritime industry, estimating the extremes of the metocean climate is essential to design oil rigs and other marine structures (\cite{zachary1998multivariate}). In finance, prediction of tail risk measures such as value-at-risk is used to manage portfolio losses (\cite{longin2000value,mcneil1999extreme}). In insurance, product pricing is stress-tested by modeling large claims and estimating ruin probabilities  (\cite{mcneil1997estimating,beirlant1992modeling}). In transportation, safety analysis is built on the estimation of crashes and other defined conflicts that are rare events by nature (\cite{jonasson2014internal,songchitruksa2006extreme}).

% helps to allocate the resources and improves the efficiency of the safety management system Despite the importance of this problem,
A recurrent challenge in extreme event analysis is that, by its very definition, there are few data available to fit the tail distribution. A dominant approach in the statistics literature is to use extreme value theory, which suggests the use of parametric models to extrapolate data based on principled asymptotics. Below we will discuss these models and their documented challenges, starting from the univariate case and then transiting to the multivariate case, which is more subtle and constitutes our focus.
% Therefore, it's intrinsically challenging for EEA which mainly focuses on estimating the tail distribution and rare-event probabilities in this paper.

Extreme value theory in the univariate setting hinges on two important asymptotic theorems that characterize the parametric distributions one should use in fitting tails, leading to the so-called \emph{annual-maxima} method (\cite{gumbel1958statistics}) and \emph{peak-over-threshold} method (\cite{smith1984threshold}) respectively. To set the stage, let us denote $\{X_1,\ldots,X_n\}$ as i.i.d. data or random variables in $\mathbb R$. The first theorem, known as the Fisher–Tippett–Gnedenko theorem (\cite{fisher1928limiting}, \cite{gnedenko1943distribution}), states that under some technical conditions (see, e.g., \cite{embrechts2013modelling} Section 3.3 and 3.4), the maximum of i.i.d. random variables, namely $\max\{X_1,\ldots,X_n\}$, converges to the generalized extreme value (GEV) distribution%
\begin{equation}
G_{\xi,\mu,\sigma}^{GEV}(x)=\left\{
\begin{array}{lr}%
\exp\left[  -\left(  1+\xi\frac{x-\mu}{\sigma}\right)  ^{-1/\xi}\right] & \text{if }\xi\neq0\\
\exp\left[  -\exp\left(  -\frac{x-\mu}{\sigma}\right)  \right] & \text{if }\xi=0
\end{array}
\right.  ,
\label{GEVdistribution}%
\end{equation}
under suitable normalization, where $\mu$ is a location parameter and $\sigma>0$ is a scale parameter. Depending on the value of $\xi$, this distribution is categorized into three regimes known as Gumbel ($\xi=0$), Fr\'{e}chet ($\xi>0$) and Weibull ($\xi<0$), each of which classifies a random variable according to its so-called maximum domain of attraction. This theorem suggests the fitting of data into $G_{\xi,\mu,\sigma}^{GEV}$, provided that the data are first batched into blocks in which the maximum is taken from each block.
% , and the estimation of tail probability can then be inferred from the fitted distribution.
% computed model data are assumed to be i.i.d. samples from the GEV distribution $G_{\xi,\mu,\sigma}(x)$. Such data are available when they can be viewed as the block-wise maxima, since EVT ensures the convergence of the maxima to the GEV distribution.
The second theorem, known as the Pickands–Balkema–de Haan theorem (\cite{pickands1975statistical}, \cite{balkema1974residual}), states that, under the same technical conditions as Fisher–Tippett–Gnedenko, the excess of a considered random variable $X$ over a high threshold $u$, defined as $X-u$ given $X>u$, converges to the generalized Pareto (GP) distribution%
\[
G_{\xi,\sigma}^{GP}(x)=1-\left(  1+\xi\frac{x}{\sigma}\right)  ^{-\frac{1}{\xi}}.
\]
This theorem, while theoretically equivalent to Fisher–Tippett–Gnedenko as hinted by the same needed technical conditions, suggests an alternate approach to fit the tail. More concretely, we choose a high threshold $u$ and fit the excess of data above $u$ into $G_{\xi,\sigma}^{GP}$.
% The tail probability of interest $P(X>x)$ is then approximated by $\bar{F}(u)G_{\hat{\xi},\hat{\sigma}}^{GP}(x-u)$ (ZHENYUAN: I THINK WE CAN DELETE THIS FORMULA. IF WE ALSO WRITE THE DETAIL FOR THE ANNUAL MAXIMA  METHOD, THEN WE NEED TO DEFINE THE NORMALIZING CONSTANTS AND TALK ABOUT HOW TO ESTIMATE THEM. THIS MAY BE TOO LENGTHY. IF WE WRITE THIS FORMULA, THEN WE SHOULD WRITE AT THE SAME LEVEL OF DETAIL FOR THE ANNUAL MAXIMA METHOD FOR THE SAKE OF EXPOSITIONAL CONSISTENCY).
In both methods, the asymptotically justified distributions $G_{\xi,\mu,\sigma}^{GEV}(x)$ and $G_{\xi,\sigma}^{GP}(x)$ are parametrized by a small number of parameters, which can be estimated by maximum likelihood estimation (MLE) and other parametric methods (\cite{embrechts2013modelling} Chapter 6). These methods have been used for decades by hydrologists, insurers, financial managers and  modelers in various other industries (see, e.g., \cite{smith1986extreme}, \cite{rootzen1997extreme}, \cite{danielsson1997tail} and \cite{solari2012unified}).

% \bar{F}(x)=
%  is chosen and the excess distribution over $u$ is approximately

While powerful, the two methods in extreme value theory are known to face a bias-variance tradeoff that is not always easy to handle. The bias comes from the use of a parametric distribution that is valid only asymptotically, but in finite sample (in the case of annual-maxima) or finite threshold (in the case of peak-over-threshold) it incurs a model misspecification error. The variance refers to the estimation variability of the parameters that arises from a limited data size. More precisely, in the case of annual-maxima, given a total sample size, there is a tradeoff between the number of blocks and the sample size per block and, if we choose a large sample size per block to lower the model bias, we must necessarily use fewer blocks that increases the variance of parameter fitting in the GEV. In the case of peak-over-threshold, if we choose a high excess threshold to lower the model bias, then the amount of data above the threshold must necessarily decrease, leading to again a higher variance for fitting the GP. This bias-variance leads to two issues. First is that the optimal, or even a good choice of block size or threshold value relies on intricate second-order distributional properties of the data (\cite{smith1987estimating,bladt2020threshold}). Second is that, when data is limited, there simply may not exist a good choice of block size or threshold value to control the bias and variance simultaneously. In the literature, there exists a variety of visualization and diagnostic tools that, though ad hoc in nature, demonstrably provide good guidance in tuning these prior parameters for model fitting (\cite{embrechts2013modelling} Chapter 6; \cite{mcneil2015quantitative} Chapter 5).

% theorem is equivalent ese two asymptotc theorems lead to the widely used annual maxima method and peak-over-threshold method.

% For the POT method, data are assumed as i.i.d. samples from a distribution $F$ in the domain of attraction of $G_{\xi,\mu,\sigma}(x)$. To use the method,
\subsection{Challenges in Conventional Extreme Event Analysis: Multivariate Case}\label{sec:challenges EEA multivariate}
We have seen the intricacy in univariate extreme event estimation. In the multivariate case, these challenges evidently continue to hold. More importantly, new difficulties arise.

% However, difficulty, however, arises even more prominently in the multivariate setting. , but both run into new challenges in the multivariate case, but,
To explain in more detail, the two approaches in univariate extreme value theory both have multivariate analogs. For annual-maxima, under technical conditions (\cite{resnick2013extreme} Section 5.4), the component-wise maximum of i.i.d. random vectors converges to the multivariate extreme value distribution given by%
\begin{equation}
G(x)=\exp(-l(-\log G_{1}(x_{1}),-\log G_{2}(x_{2}),\ldots,-\log G_{d}(x_{d}))), \label{multivariate_extreme_value_distribution}%
\end{equation}
under suitable normalization, where $G_{i}$'s are the marginal distributions and belong to the GEV family (\ref{GEVdistribution}), and $l(\cdot)$ is called the stable tail dependence function.
%In the bivariate case, this is equivalent to
%\begin{equation}
%G(x)=\exp\left[  \log(G_{1}(x_{1})G_{2}(x_{2}))A\left(  \frac{\log(G_{2}(x_{2}))}{\log(G_{1}(x_{1})G_{2}(x_{2}))}\right)  \right],
%\label{bivariate_extreme_value_distribution}%
%\end{equation}
%where $A(\cdot)$ is called the Pickands dependence function.  and estimate the tail probability from the fitted distribution
This latter function $l$, which is defined on $[0,\infty]^d$, summarizes the dependence structure among the $d$ components of $G(x)$ and needs to satisfy the following necessary conditions (\cite{beirlant2006statistics} Section 8.2.2):
\begin{align*}
& \text{(C1) } l(s\cdot)  =sl(\cdot)\text{ for }0<s<\infty;\\
& \text{(C2) } l(e_{j})  =1\text{ for }j=1,\ldots,d\text{, where }e_{j}\text{\ is the }j\text{th unit vector in }\mathbb{R}^{d};\\
& \text{(C3) } \max\{v_{1},\ldots,v_{d}\}  \leq l(v)\leq v_{1}+\cdots v_{d}\text{ for }v\in\lbrack0,\infty)^{d};\\
& \text{(C4) } l\text{ is a convex function}.
\end{align*}
The above result suggests the multivariate annual-maxima method (\cite{gumbel1964analysis}) where we divide data into blocks like in the univariate case, then fit the component-wise maxima of each block into $G(x)$. On the other hand, the multivariate peak-over-threshold method (\cite{beirlant2006statistics}) relies on the concept of copula which captures the dependency among marginals (e.g., \cite{nelsen2007introduction}), more precisely the convergence
%  simultaneous convergence of marginal distributions and the, The latter means
\begin{equation}
\lim_{t\rightarrow\infty}C_{F}^{t}(v_{1}^{1/t},\ldots,v_{d}^{1/t})=C_{G}(v),\ v=(v_1,\ldots,v_d)\in\lbrack0,1]^{d}, \label{copula_convergence}
\end{equation}
where $C_{F}$ and $C_{G}$ are the copulas of the sample distribution $F$ and the multivariate extreme value distribution $G$ respectively. This yields the approximation
\begin{equation}
C_{F}(v)\approx C_{G}(v)=\exp(-l(-\log(v_{1}),\ldots,-\log(v_{d}))).\label{copula_approximation}
\end{equation}
for high enough values of $v$, where the equality follows a provable equivalence with the representation (\ref{multivariate_extreme_value_distribution}).
% for $u$ whose components are all sufficiently close to $1$, where the equality follows from the representation (\ref{multivariate_extreme_value_distribution}).
This suggests that we choose a high threshold level $u\in \mathbb{R}^d$ such that the approximation (\ref{copula_approximation}) is reliable and, like in the univariate case, employ the data which exceeds $u$ to fit (\ref{copula_approximation}).
% to estimate the tail probability.

% which together with the convergence of marginal distributions is equivalent to the multivariate counterpart of the Fisher–Tippett–Gnedenko theorem above.

The multivariate methods inherit the bias-variance tradeoff as the univariate case regarding the block size or threshold value (\cite{ledford1996statistics}). Moreover, there are two new difficulties that make multivariate case even more challenging. One is the opaqueness on the technical conditions of the underlying theorems, which are all formulated in terms of the asymptotic behaviors of the distribution in the tail region which is exactly the problem target (\cite{resnick2013extreme}). The second difficulty is the lack of information on the function $l$. Since $l$ does not admit a finite-dimensional parametrization, an approach is to restrict it to a parametric subfamily (e.g., logistic model and its variations in \cite{gumbel1960bivariate,tawn1988bivariate,joe1992bivariate}) so that parametric inference tools like MLE can be used, but this encounters the risk of model misspecification. An alternate approach is to use nonparametric methods (e.g., \cite{pickands1981multivariate,caperaa1997nonparametric,de1989intrinsic,hall2000distribution,drees1998best,caperaa2000estimation}). However, the construction of nonparametric estimators of $l$ that satisfy all the necessary properties (C1)-(C4) remains open (\cite{beirlant2006statistics}). Although in the bivariate case there are some ad hoc modifications to the nonparametric methods listed above, e.g., taking the convex minorant for the non-convex estimator in \cite{hall2000distribution}, these modifications appear difficult to generalize to higher-dimensional cases, and moreover even the performances of the bivariate estimators are unclear (\cite{beirlant2006statistics}).

\subsection{DRO for Multivariate Extreme Event Analysis}\label{sec:challenges DRO}
Motivated by the challenges in conventional statistical methods for tackling multivariate extremes, we consider DRO as a well-grounded alternative. As described before, all existing approaches extrapolate tail by using models justified from asymptotic theory. In finite sample, these methods can face difficult error tradeoff. Moreover, in the multivariate case, these asymptotic models are beyond parametric as they involve the stable tail dependence function that does not admit a finite-dimensional parametrization. Our key idea is to replace these models with geometric tail property that is placed as constraints in a worst-case optimization. More precisely, suppose we focus on the estimation of a tail probability $P((X_1,\ldots,X_d)\in S)$ where $S$ is an extreme set. We consider optimization roughly speaking in the form
\begin{equation}
\begin{array}{ll}
\max_F&P_F((X_1,\ldots,X_d)\in S)\\
\text{subject to}&\text{geometric property of the distribution $F$ for $X\ge u$}\\
&\text{auxiliary constraints on $F$}
\end{array}\label{basic}
\end{equation}
where $X=(X_1,\ldots,X_d)$, $u=(u_1,\ldots,u_d)$, and the inequality is defined component-wise. The decision variable in \eqref{basic} is the unknown distribution $F$. At least some of the thresholds $u_i$ is large, so that the geometric property corresponds to the (positive) tail of the distribution. Problem \eqref{basic} can be viewed as a DRO with an uncertainty set on the unknown distribution $F$ that is characterized by the geometric conditions and auxiliary constraints.

We hold off the discussion of a suitable geometric property for the tail and why we need the auxiliary constraints at the moment, and first explain conceptually how to use \eqref{basic} and why this bypasses the challenges of the conventional methods. This geometric property acts as general, nonparametric replacement of the parametric (or semiparametric) models in extreme value theory, which as we have seen faces several challenges in usage, and in particular relies on a good choice of exceedance threshold $u$. When we use the geometric property, we will not hinge on the asymptotic theory that relies on a high enough $u$, thus bypassing the bias-variance tradeoff faced by the peak-over-threshold method. However, the catch is that there can be many distributions that satisfy such general geometric conditions, and consequently we take the worst-case value of the target performance measure to construct a bound.

For the last point above, we note the following trivial guarantee:
\begin{lemma}
Suppose that
\begin{equation}
P\left(\begin{array}{c}
\text{geometric property of the distribution $F$ holds for $X\ge u$},\\
\text{auxiliary constraints on $F$ holds}
\end{array}\right)\geq1-\alpha\label{confidence condition}
\end{equation}
for some confidence level $1-\alpha$. Then the optimal value of \eqref{basic}, called $Z^*$, satisfies
$$P(Z^*\geq Z)\geq1-\alpha$$
where $Z$ is the true value of $P((X_1,\ldots,X_d)\in S)$.\label{lemma:basic}
\end{lemma}

In Lemma \ref{lemma:basic}, the conditions inside the probability in \eqref{confidence condition} are calibrated from data, and the probability is with respect to the randomness from data. Lemma \ref{lemma:basic} concludes that if the uncertainty set contains the true distribution with high confidence, then the optimal value of the associated DRO would be an upper bound for the true target value with at least the same confidence level. The guarantee in Lemma \ref{lemma:basic} is well-established in data-driven DRO (\cite{delage2010distributionally,ben2013robust,esfahani2018data,bertsimas2018robust}). Moreover, a corresponding lower bound guarantee holds analogously and we have skipped to avoid repetition.

Thus, DRO provides confidence bounds on target tail performance measure as long as the uncertainty set is a valid confidence region on the unknown true distribution. The question then is what constitutes a good choice of uncertainty set, which we discuss next.

\subsection{Choices of Uncertainty Set}
In data-driven DRO, uncertainty sets can be generally categorized into two major types. The first type is a neighborhood ball surrounding a baseline distribution, where the ball size is measured via a statistical distance. Common choices of distance include the class of $\phi$-divergence (\cite{glasserman2014robust,gupta2019near,bayraksan2015data,Iyengar2005,hu2013kullback,duchi2016statistics,gotoh2018robust,Lam2016robust,Lam2018sensitivity,ghosh2019robust}) which also covers in particular the Renyi divergence (\cite{atar2015robust,dey2010entropy}) and total variation distance (\cite{jiang2018risk}), and the Wasserstein metric (\cite{esfahani2018data,blanchet2021sample,gao2016distributionally,xie2019tractable,shafieezadeh2019regularization,chen2018robust}). The ball sizes using these distances are calibrated from either density and entropy estimation (\cite{jiang2016data}), using goodness-of-fit statistics (\cite{ben2013robust,bertsimas2018robust}), or employing or developing nonparametric empirical likelihood theory (\cite{LAM2017301,duchi2016statistics,lam2019recovering,Blanchet2019,blanchet2019confidence}). However, these approaches do not apply naturally to tail estimation. The first approach requires substantial amount of data due to the need of using kernel estimation, while the second approach could be conservative, both imposing challenges in the tail region. The third approach, on the other hand, builds on a statistical theory that ties to the objective function in the DRO, and its validity in tail estimation is not established.

The second major type of uncertainty sets constitutes partial information on the distribution, including moments and support (\cite{delage2010distributionally,bertsimas2005optimal,Wiesemann2014,goh2010distributionally,ghaoui2003worst}), marginal constraints (\cite{doan2015robustness,dhara2021worst}), and shape constraints (\cite{van2016generalized,li2019ambiguous,lam2017tail,mottet2017optimization,chen2021discrete}). The first two subtypes require calibration, i.e., setting bounds on the moments, supports or the marginal distributions. The last subtype does not require calibration, and thus can be used even when no data is available. At the same time, it reflects the geometric belief on the distribution and, when correctly imposed, it advantageously alleviates conservativeness.
% even when other constraints are applied together.

Thanks to the power of reducing estimation conservativeness with few data, shape constraints appear suitable to be the primary choice in constructing uncertainty sets for extremal estimation. This observation is in line with some documented motivation in the literature (e.g., \cite{li2019ambiguous}). Before we argue the specific shape constraint to be used, we also mention several works in specializing DRO in extremal estimation. The most relevant is \cite{lam2017tail} that considers convex tail extrapolation, focusing on the univariate setting. It investigates the light-versus-heavy tail properties in the worst-case distributions and the associated computation procedures. \cite{blanchet2020distributionally} considers robustification of GEV using Renyi divergence ball and studies the domain-of-attraction properties of the worst-case distributions, which aims to alleviate the reliance on the validity of asymptotics in justifying the GEV model. \cite{engelke2017robust} derives robust asymptotic bounds on exceedance probabilities subject to $\chi^2$-ball and first moment, and  \cite{birghila2021distributionally} studies bounds on both probabilities and tail indices based on the Wasserstein distance and $f$-divergence around a heavy-tailed distribution. Moreover, as we have seen, copula or dependence structure of random vectors plays an important role in multivariate extreme event analysis. Motivated by this, there are works on robust bounds for extremal performance measures when the marginal distributions are given but dependence structure is unknown. These measures include tail probabilities for functions of random vectors that can be interpreted as financial risks (\cite{embrechts2006bounds1,embrechts2006bounds2,puccetti2013sharp}), expected values for convex functions of sums (\cite{wang2011complete}) and conditional value-at-risk (\cite{dhara2021worst}).

\subsection{Challenges of Existing Shape Constraints}\label{sec:challenges existing MU}
A commonly used shape condition in the multivariate setting is unimodality (\cite{dharmadhikari1988unimodality}). In the univariate case, a unimodal probability density can be readily intuited as having a unique mode (or connected set of modes) with monotonically decreasing density when moving away from the mode. In the multivariate case, defining unimodality becomes more subtle as the notion of monotonicity is primarily one-dimensional and different definitions can be drawn depending on how one defines monotonicity. The three most widely used multivariate unimodality notions are \emph{star unimodality}, \emph{block unimodality} and  \emph{$\alpha$-unimodality} (\cite{dharmadhikari1988unimodality} Sections  2.2 and 3.2). In the following, we will introduce these notions which would then help understand their limitations and our motivations for proposing our OU notion.

%To facilitate our discussion on the key issue, let us restrict the region in consideration as $\mathcal{D}_{0}=\{(x_{1},x_{2},\ldots,x_{d}):x_{i}\geq x_{i0},i=1,2,\ldots,d\}$ where $x_{0}:=(x_{10},x_{20},\ldots,x_{d0}) \in \mathbb{R}^{d}$ is a fixed point (which will define the mode of a distribution). We consider an extreme set $S$ that is a subset of $\mathcal{D}_{0}$. If needed, we can readily consider bigger regions by pasting individual regions of similar form as above into our analysis. has a clear geometric meaning.
First we discuss star unimodality. Given a mode $x_{0}:=(x_{10},x_{20},\ldots,x_{d0}) \in \mathbb{R}^{d}$, a probability distribution with density on $\mathbb{R}^d$ is called star unimodal if this  density is non-increasing along any ray pointing away from the mode $x_{0}$. That is,
% (ZHENYUAN: I THINK THE TERMINOLOGY USED IN THE LITERATURE IS ``ABOUT THE MODE" INSTEAD OF ``AT THE MODE". I CORRECTED IT IN THE FOLLOWING DEFINITION.)

\begin{definition}[Star unimodal density]
A probability distribution with density (with respect to the Lebesgue measure) is \emph{star unimodal} about mode $x_{0}$ if the density is non-increasing along any ray pointing away from $x_0$ (i.e., $tx+x_0,t>0$ for any nonzero vector $x\in\mathbb R^d$).\label{star def}
\end{definition}

Star unimodal distribution can also be defined using a \emph{mixture representation} which does not require the existence of the density. This representation requires us to define star-shaped sets, detailed as follows.

\begin{definition}[Mixture representation of star unimodal distribution]We have:

\begin{description}
\item[(1)] A set $K$ is said to be star-shaped about $x_{0}$ if for every $x\in K$, the line segment joining $x$ to $x_{0}$ is completely contained in $K$.

\item[(2)] A probability distribution on $\mathbb{R}^d$ is called star unimodal about $x_{0}$ if it belongs to the closed convex hull of the set of all uniform distributions on sets that are star-shaped about $x_{0}$.
\end{description}\label{star mix}
\end{definition}

Definition \ref{star mix} is equivalent to Definition \ref{star def} when the distribution has a density (\cite{dharmadhikari1988unimodality}, the criterion in Section 2.2 or Theorem 3.6).

In parallel to star unimodality, a block unimodal distribution is defined as a mixture of uniform distributions on rectangles instead of star-shaped sets. That is,

\begin{definition}[Mixture representation of block unimodal distribution]
A probability distribution on $\mathbb{R}^d$ is called \emph{block unimodal} about $x_{0}$ if it belongs to the closed convex hull of the set of all uniform distributions on rectangles that contain $x_{0}$ and have edges parallel to the coordinate axes.\label{block mix}
\end{definition}

It is easy to see that block unimodality satisfies that the density is non-increasing along any ray pointing away from the mode. By either this observation or combining the fact that rectangle is star-shaped and the mixture representation, we see that block unimodal distributions form a subclass of star unimodal distributions.

On the other hand, $\alpha$-unimodality can be viewed as a generalization of the star unimodality notion, by allowing the density to increase on a ray pointing away from the mode but at a controlled rate. More concretely,
\begin{definition}[$\alpha$-unimodal density]
A probability distribution with density on $\mathbb{R}^d$ is \emph{$\alpha$-unimodal} about $x_0$ if the density $f$ is such that $t^{d-\alpha}f(tx+x_0)$ is non-increasing in $t\in(0,\infty)$ for any nonzero vector $x\in \mathbb{R}^d$. \label{alpha def}
\end{definition}

% Like star unimodality, $\alpha$-unimodality can also be defined without requiring the existence of density, as below:

% \begin{definition}
% A probability distribution $F$ on $\mathbb{R}^d$ is called $\alpha$-unimodal about $x_0$ if for every bounded, non-negative, Borel measurable function $g$ on $\mathbb{R}^d$, $t^{\alpha}E_F[g(t(X-x_0))]$ is non-decreasing in $t\in(0,\infty)$.\label{alpha alt}
% \end{definition}
% % (THIS IS NOT A MIXTURE REPRESENTATION? WE MAY REMOVE IF THERE'S NO ROLE IN OUR DISCUSSION. NOTE THAT FOR STAR UNIMODAL WE NEED THE MIXTURE REPRESENTATION AS A NATURAL TRANSITION TO BLOCK UNIMODAL)(ZHENYUAN: THIS IS NOT A MIXTURE REPRESENTATION SO WE CAN REMOVE IT. LATER IN THE END OF SECTION 4, WE WILL USE THE CHOQUET-TYPE REPRESENTATION OF $\alpha$-UNIMODALITY, I.E., (\ref{Choquet_alpha_unimodality}). I ALSO THINK THERE IS NO NEED TO INTRODUCE (\ref{Choquet_alpha_unimodality}) HERE.)
% Definitions \ref{alpha def} and \ref{alpha alt} are equivalent when the distribution has a density (\cite{dharmadhikari1988unimodality} Theorem 3.6). Nonetheless, unlike Definitions \ref{star mix} and \ref{block mix}, Definition \ref{alpha alt} is not a mixture representation, but it would be useful for our analysis later.(ZHENYUAN:DEFINITION \ref{alpha alt} IS NOT USED LATER. WE MAY NEED TO REMOVE IT.)

We argue that all of the star unimodality, block unimodality and $\alpha$-unimodality notions encounter issues in extreme event estimation using DRO, when we place them as the ``geometric property" in problem \eqref{basic}.
% as one of these unimodality notions, with a certain modal point that needs to be specified.
To facilitate discussion, let us focus on $u>x_0$ in \eqref{basic}, i.e., in the positive tail region, or in other words it suffices to define unimodality about $x_0$ on $\mathcal{D}_{0}=\{x\in\mathbb{R}^d:x\ge x_0\}$ by requiring the rays and sets in the definitions above to be contained in $\mathcal{D}_{0}$. Below, we describe the issues of the existing unimodality under this regime.

% we only care about the geometric property for $X>u$ rather than the entire $\mathbb{R}^d$
% DEFINE STAR AND BLOCK UNIMODALITY. (ZHENYUAN: I ADDED THE NAMES (a)(b) FOR THE SUBFIGURES)

Star unimodality is highly sensitive to its mode. That is, when the mode is misspecified, the intended geometric property can become incorrect, even when star unimodality itself holds. To put it in another way, the two conditions ``$F$ is star unimodal about mode $x_0$ for $X\ge x_0$" and ``$F$ is star unimodal about mode $x_0'$ for $X\ge x_0'$", for two different $x_0$, $x_0'$ (where each of them is component-wise smaller than $u$), can result in very different uncertainty sets. In fact, even when $x_0$ and $x_0'$ differ only slightly, the difference in the uncertainty sets could be huge. For instance, in Figure \ref{fig:star}(a) the shaded area represents the region $\{X\ge u\}$ in which star unimodality is used to capture the tail geometry. With different modes $x_0$ and $x_0'$, the sets of ray directions on which the density is monotonically non-increasing are different. Moreover, when the shaded area is far away from the mode, a small misspecification of its location can cause a huge difference in the set of directions. Since not all problems have clearly defined modes to begin with, and estimation of the mode, even though statistically possible, causes sensitive impacts on the tail geometry, star unimodality can be difficult to apply in practice.
\begin{figure}[ptbh]
% \vskip 0.2in
\begin{center}
\subfigure[Star Unimodality]
{\includegraphics[width=.45\textwidth]{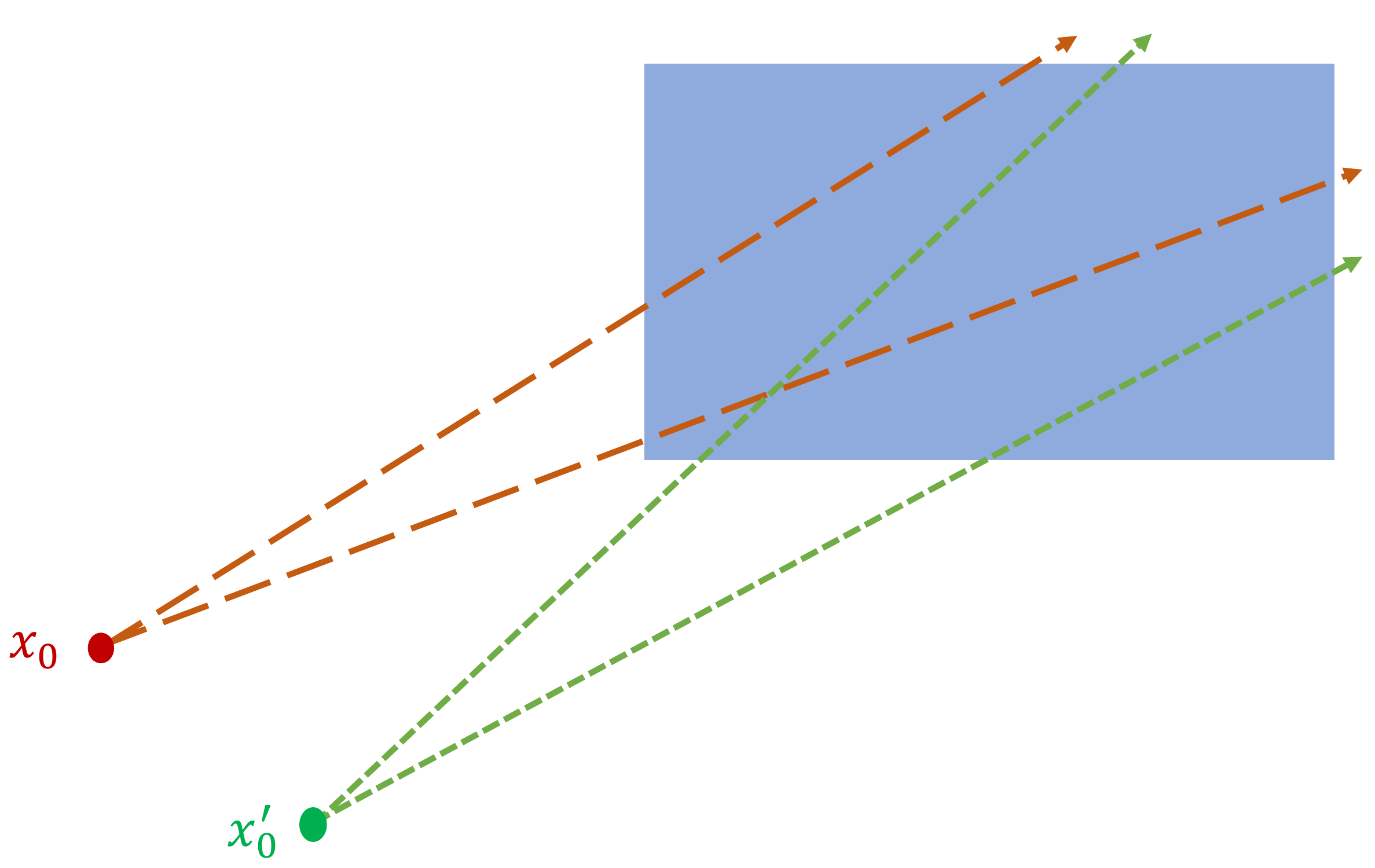}}
\subfigure[Orthounimodality]
{\includegraphics[width=.45\textwidth]{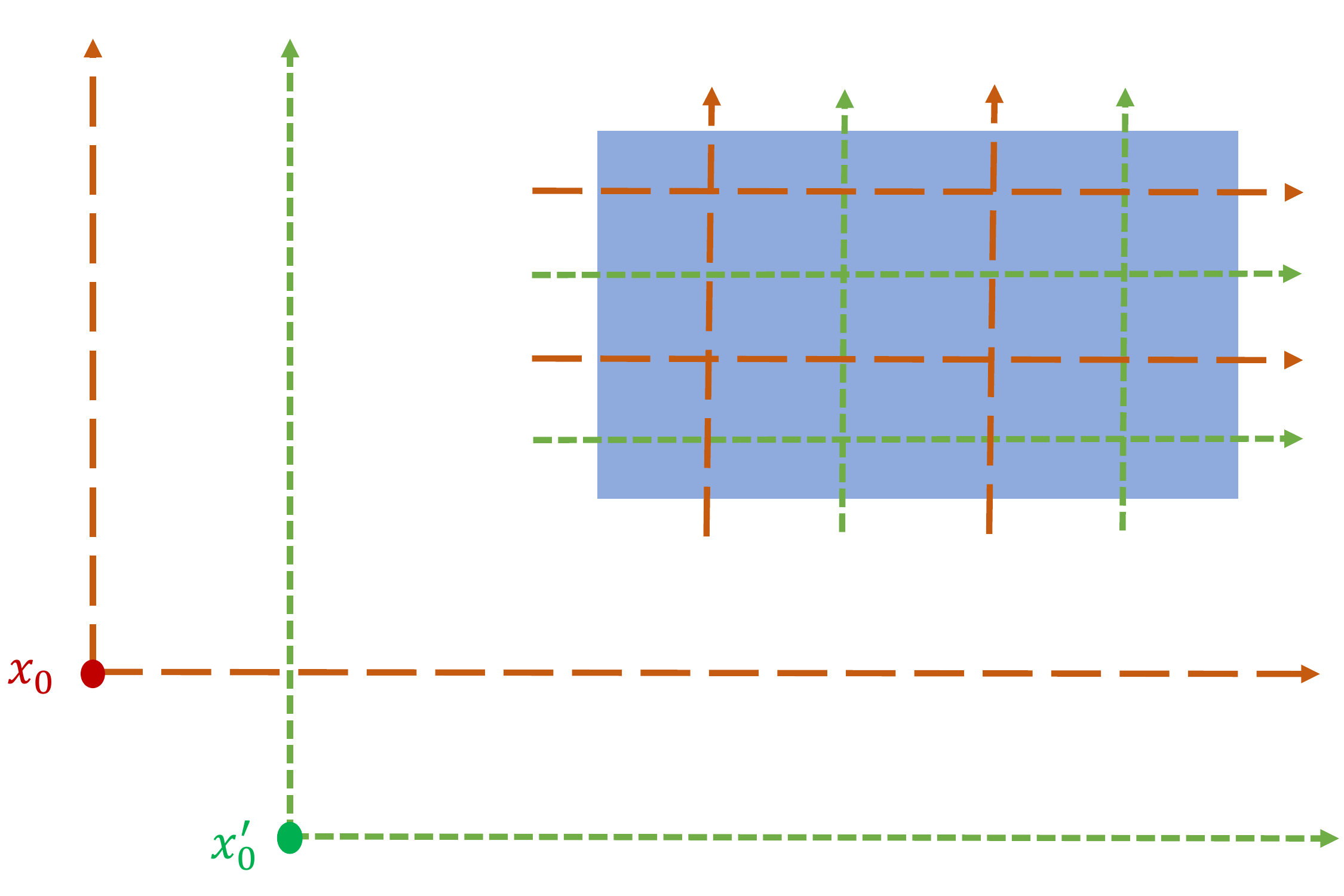}}
\caption{Star unimodality versus orthounimodality. Red (green) dashed lines indicate the rays along which the density of the distribution about mode $x_0$ ($x_0^{\prime}$) is non-increasing.}
\label{fig:star}
\end{center}
% \vskip -0.2in
\end{figure}

Next we discuss the challenges of block unimodality. While having a clear mixture representation, block unimodal distribution owns a ``differencing" property that resembles the requirement of a multivariate cumulative distribution function. In the bivariate case for instance, a distribution $F(x,y)$ with a continuous density $f(x,y)$ that is block unimodal about the mode $(x_{0},y_{0})$ must satisfy that%
\[
f(x_{1},y_{1})-f(x_{1},y_{2})-f(x_{2},y_{1})+f(x_{2},y_{2})
\]
is nonnegative for any $x_{0}<x_{1}<x_{2}$ and $y_{0}<y_{1}<y_{2}$.
This requirement is unintuitive and difficult to check in general. In fact, it is difficult to reason why a distribution should behave this way. Because of this, block unimodality is also difficult to apply, and its differencing requirement flags that it could be too stringent as a geometric property to be used.

Lastly, $\alpha$-unimodality runs into several challenges similar to star unimodality and block unimodality. Similar to star unimodality, misspecification of the mode can cause sensitive impact to the implied geometric property in the tail region. Similar to block unimodality, it is unclear why a distribution should behave in the way that the notion is specified, namely that the density changes in the precisely controlled way in Definition \ref{alpha def}. Moreover, even if such a property is true, the specification of $\alpha$ can be a challenge. As an example, we consider a distribution $F$ with density
\begin{equation}
f(x_1,x_2)=C\exp\left(  -  \max\left(  \arctan\left(  \frac{x_2}{x_1}\right),\arctan\left(  \frac{x_1}{x_2}\right)  \right)    (x_1+x_2)\right)  ,x_1\geq0,x_2\geq0,\label{star_density}
\end{equation}
where $C$ is a normalizing constant to make $f$ a probability density. For this density, it is not easy to tell if it could be $\alpha$-unimodal for some mode $x_0$ at a first glance. In fact, when $\alpha<2$, $f$ cannot be $\alpha$-unimodal for any $x_0$ since $\lim_{t\rightarrow0}t^{2-\alpha}f(tx+x_{0})=0$ would contradict the condition that $t^{2-\alpha}f(tx+x_{0})$ is non-increasing in $t\in(0,\infty)$. When $\alpha\ge2$, $f$ can be $\alpha$-unimodal ($\alpha$-unimodality reduces to star unimodality when $\alpha=2$) but the choice of mode is very subtle. We can show by routine calculus that $f$ is $\alpha$-unimodal about $x_0\in \mathbb{R}_{+}^{2}$ if $x_0$ belongs to the diagonal while $f$ is not $\alpha$-unimodal about $x_0\in \mathbb{R}_{+}^{2}$ if $|x_{20}-x_{10}|>4(\alpha-2)/(4-\pi)$. This example illustrates the main drawbacks discussed above, that it is hard to judge whether, or reason why, a distribution is $\alpha$-unimodal, and the specification of $\alpha$ and the mode is not an easy task.

\section{Resolution via Orthounimodality}\label{sec:OU}

Due the limitations of the existing multivariate unimodality notions discussed in Section \ref{sec:challenges existing MU}, we propose another multivariate unimodality notion called orthounimodality (OU) which, as we will argue, is natural for extreme event analysis and resolves the statistical challenges faced by the existing unimodality notions.
% We then discuss the methodological difficulties in solving the OU-DRO problem, whose detailed reduction will be presented in the next section.
% In this section, we will introduce the definition of OU, understand its geometric properties, provide a so-called Choquet representation theorem and

% explain its advantages for extreme event analysis,
%  gets less attention but turns out to be natural and intuitive for extreme event analysis and avoids the sensitivity of the mode.
%(ZHENYUAN: FOR NOW I ONLY DEFINE OU ON $\mathcal{D}_0$ BECAUSE THIS IS WHAT WE REALLY USE IN OUR PROBLEM. BESIDES, I THINK THE DEFINITION OF OU ON $\mathbb{R}^d$ SEEMS MESSY. IN GENERAL, OU REQUIRES $f(x)\ge f(x^{\prime})$ IF $x^{\prime}$ IS MORE EXTREME THAN $x$. BUT THE DEFINITION OF ``MORE EXTREME" DEPENDS ON THE RELATIVE LOCATION WITH RESPECT TO THE MODE. SO MAYBE COMPLETELY EXPLAINING IT IS MESSY IN THE MAIN BODY. BUT TO GIVE A COMPLETE THEORY ON OU, I CAN ALSO GIVE THE DEFINITION AND CHOQUET-TYPE REPRESENTATION IN THE APPENDIX. HOW DO YOU THINK?)
% The definition is motivated from two aspects to tackle the challenges. One is we want to modify the star unimodality in a way that is not sensitive to the mode. The other is we want to describe a natural distribution behavior in the tail part: any point that is more ``extreme" than a point in the tail part should appear more rarely.

First, we define OU for probability densities. Here for simplicity, we only define OU in the positive region of the mode since we only focus on the geometric property in the positive tail part of $u$ in Problem \eqref{basic} (we leave the discussion on more general OU distributions in Appendix \ref{sec:general OU}). Given a mode $x_0\in\mathbb{R}^d$ and its positive region $\mathcal{D}_{0}=\{x\in\mathbb{R}^d:x\ge x_0\}$, a probability $f$ on $\mathcal{D}_{0}$ is called OU if $f(x)$ is non-increasing in any component of $x$ on $\mathcal{D}_{0}$, i.e.,
\begin{definition}[Orthounimodal density]
A probability distribution on $\mathcal{D}_{0}$ with density $f$ (with respect to the Lebesgue measure) is OU about mode $x_0$ if $f(x^{\prime})\ge f(x)$ for $x\ge x^{\prime}\ge x_0$.
\label{OU def}
\end{definition}

% Besides, the requirement $f(x^{\prime})\ge f(x)$ for $x\ge x^{\prime}\ge x_0$ precisely reflects the idea that any point $x$ that is more ``extreme" than a point $x^{\prime}$ in the tail part $\mathcal{D}_0$ should appear more rarely, which ,  is more natural and intuitive than block unimodality. densities possess the Now we discuss how the definition of OU manifests its motivations and the advantages of OU over the three multivariate unimodality notions defined in Section \ref{sec:challenges existing MU}
The definition of OU above is very intuitive in that any point that is more ``extreme" than a point in the tail should appear even more rarely, where the extremeness is measured simply by the marginal positions of the points. This avoids the additional ``differencing" property possessed by block unimodality, and the intricate density change property possessed by $\alpha$-unimodality, both of which are hard to interpret. Moreover, compared to star unimodality and $\alpha$-unimodality, OU is insensitive to the misspecification of the mode, in the sense that the OU property imposed on a tail region remains correct regardless of the exact position of the mode. To see this, in Figure \ref{fig:star}(b), the shaded area represents the tail region $\{X\ge u\}$ where OU is imposed. Although the modes $x_0$ and $x_0^{\prime}$ are different, their requirements on the tail region are the same: the density is non-increasing along any ray parallel to the axes. Therefore, the geometric requirement on the tail region is independent of the choice of the mode as long as the mode is less than or equal to $u$ component-wise. Because of the interpretability and the robustness to mode misspecification presented above, OU appears suitable as the multivariate unimodality for use in extreme value analysis.

%Based on the observations, we modify the definition of star unimodality and propose another multivariate unimodality called orthogonal monotonicity which is not sensitive to the mode. In terms of the density $f$, it requires $f$ is non-increasing in any component not just along the rays coming out from the mode.

Next, like star unimodal and block unimodal distributions, an OU distribution can also be defined as a mixture of uniform distributions on what we call OU sets.

\begin{definition}[Mixture representation of orthounimodal distribution]
We have:
\begin{description}
\item[(1)] A set $K\subset\mathcal{D}_{0}$ is said to be OU about $x_{0}$ if for every $x\in K$, we have $x^{\prime}\in K$ if $x\ge x^{\prime}\ge x_0$.

\item[(2)] A distribution on $\mathcal{D}_{0}$ is called OU about $x_{0}$ if it belongs to the closed convex hull of the set of all uniform distributions on subsets of $\mathcal{D}_{0}$ that are OU about $x_{0}$.
\end{description}
\label{OU mix}
\end{definition}

We establish the equivalence of Definitions \ref{OU def} and \ref{OU mix} as well as a \emph{Choquet representation theorem} for OU. In convex analysis, Choquet theory establishes that any point in a compact convex set $C$ can be written as the mixture of the extreme points of $C$ (\cite{phelps2001lectures} Section 3). In the field of unimodal distributions, Choquet representation means any distribution in a certain class of unimodal distributions can be written as the mixture of the extreme points of this class of distributions. For example, Choquet representation has been established for the three unimodal distributions presented in Section \ref{sec:challenges existing MU} (\cite{dharmadhikari1988unimodality} Theorem 2.2 and Theorem 3.5). For $\alpha$-unimodal distribution $P$ about the origin (including star unimodality when $\alpha=d$), it can be written as
\begin{equation}
P=\int_{\mathbb{R}^{d}}W_{\alpha\text{-} uni}(z)dQ(z),\label{Choquet_alpha_unimodality}%
\end{equation}
where $W_{\alpha\text{-}uni}(z)$ is the distribution of $U^{1/\alpha}z$, $U$ is the uniform distribution on $(0,1)$, and $Q$ is a probability measure on $\mathbb{R}^{d}$ uniquely determined by $P$. For block unimodal distribution $P$ about the origin, it can be be written as
\begin{equation}
P=\int_{\mathbb{R}^{d}}W_{rect}(z)dQ(z),\label{Choquet_block_unimodality}
\end{equation}
where $W_{rect}(z)$ is the uniform distribution on the rectangle with edges parallel to the axes and opposite vertices $0$ and $z$, and $Q$ is again a probability measure on $\mathbb{R}^{d}$ uniquely determined by $P$. Here, $W_{\alpha\text{-}uni}(z),z\in\mathbb{R}^{d}$ and $W_{rect}(z),z\in\mathbb{R}^{d}$ are exactly the extreme points in the class of $\alpha$-unimodal distributions and block unimodal distributions respectively. The Choquet representation that we will prove in the following Theorem \ref{Choquet_OU} is a natural analog of (\ref{Choquet_alpha_unimodality}) and (\ref{Choquet_block_unimodality}) for OU distributions. We note that Choquet representation is not the same as the mixture representations in Definitions \ref{star mix},  \ref{block mix} and \ref{OU mix} since these definitions do not tell us if the unimodal distributions can be generated by the mixture of only the extreme points. To derive our results for OU distributions, we begin with some notations. For a set $K$, we define $W_{K}$ as the uniform distribution on $K$ and define $\lambda(K)$ as its Lebesgue measure. Besides, we denote the interior, closure and boundary of $K$ by $K^{\circ},\bar{K}$ and $\partial K$ respectively. We show that the OU set has the following properties.

\begin{lemma}%[Properties of OU sets]
\label{property_of_OU_set}Suppose that $K\subset\mathcal{D}_0$ is an OU set about $x_{0}$. Then $K$ is Lebesgue measurable. Besides, $\bar{K}$ and $\bar{K^{\circ}}$ (closure of $K^{\circ}$) are also OU sets about $x_{0}$ and $K^{\circ},K,\bar{K}$ have the same Lebesgure measure, i.e., $\lambda(K^{\circ})=\lambda(K)=\lambda(\bar{K})$.
\end{lemma}

The following theorem justifies the equivalence of Definitions \ref{OU mix} and \ref{OU def} and also establishes the Choquet representation for OU distributions in the presence of density.

\begin{theorem}
\label{Choquet_OU}Suppose a distribution $P$ on $\mathcal{D}_{0}$ is absolutely continuous with respect to Lebesgue measure. Then $P$ is OU about $x_{0}$ if and only if there is a density $f(x)$ of $P$ such that for every $s>0$, the set%
\[
C_{s}=\{x\in\mathcal{D}_{0}:f(x)\geq s\}
\]
is OU about $x_{0}$, or equivalently, if and only if
\[
f(x^{\prime})\geq f(x)\text{ for }x\ge x^{\prime}\ge x_0.
\]
Besides, $P$ has the following Choquet representation:%
\begin{equation}
P(B)=\int_{0}^{\infty}W_{\bar{C}_{s}}(B)g(s)ds \label{equ_Choquet_OU}%
\end{equation}
for any Lebesgue measurable set $B$, where $\bar{C}_{s}$ is the closure of $C_{s}$ and $g(s)=\lambda(\bar{C}_{s})$ is a probability density on $(0,\infty)$.
\end{theorem}

In the proof of Theorem \ref{Choquet_OU}, we need the following lemma about the extreme point in the class of OU distributions on $\mathcal{D}_0$ to ensure (\ref{equ_Choquet_OU}) is ineed a Choquet representation.

\begin{lemma}
If $K\subset\mathcal{D}_0$ is an OU set about $x_0$ with $\lambda(K)>0$, then $W_{K}$ is an extreme point in the class of OU distributions on $\mathcal{D}_0$.
\label{extreme_point_OU}
\end{lemma}

%Theorem \ref{representation_of_continuous_OM_distribution} tells us a criterion and a representation of the OM distribution with a density: if for any $i\in\{1,2,\ldots,d\}$, a density $f(x_{1},x_{2},\ldots,x_{d})$ is non-increasing in $x_{i}\geq x_{i0}$ for any fixed $x_{j}\geq x_{j0},j\neq i$, then the probability distribution is OM about $x_{0}$. This condition is very intuitive and reasonable especially in the tail part of the distribution. Besides, unlike star unimodality, the OM distribution is not sensitive to the choice of $x_{0}$. Actually, if a distribution is OM about $x_{0}$, then it's also OM about any point $x_{0}^{\prime}\in\mathcal{D}_{0}$. Even if we mis-specify $x_{0}$ as $x_{0}^{\prime}$ in the shape constraint, the true probability is still a feasible solution. As an example, Figure \ref{contour_OM} shows the contours of the density of the standard bivariate normal distribution which is OM about $(0,0)$. As claimed in Theorem \ref{representation_of_continuous_OM_distribution}, the graphs that are formed by the contours and two axes are OM about $(0,0)$. This distribution is also OM about any mode $(x_0,y_0)\in \mathbb{R}^2_{+}$ according to Theorem \ref{representation_of_continuous_OM_distribution}. Because of these advantages, we choose OM distributions as our shape constraint.
Note that the Choquet representation for OU obviously implies its mixture representation. Comparing the mixture representation of OU with those of star and block unimodality, we also see that the class of OU distributions in fact lies in between star unimodality and block unimodality as stated in the following proposition. In other words, the additional ``differencing" property of block unimodality that is hard to interpret can be viewed as overly stringent and unnecessary if we use the OU notion.

%\begin{figure}[ptbh]
%\centering
%\includegraphics[width=0.5\textwidth]{contour_OM.png}
%\caption{Contours of $f(x,y)=\exp(-(x^2+y^2)/2)/(2\pi)$}
%\label{contour_OM}
%\end{figure}

\begin{proposition}\label{inclusion_unimodalities}
For any $x_{0}\in \mathbb{R}^{d}$, the following is true:%
\begin{align*}
&  \{\text{distributions on }\mathcal{D}_0\text{ that is block unimodal about }x_{0}\}\\
\subsetneq & \{\text{distributions on }\mathcal{D}_0\text{ that is OU about }x_{0}\}\\
\subsetneq & \{\text{distributions on }\mathcal{D}_0\text{ that is star unimodal about }x_{0}\}.
\end{align*}

\end{proposition}

As a notion of multivariate unimodality, Definition \ref{OU def} of OU has been studied in a range of fields with different names such as ``orthounimodal" and ``block decreasing". For example, \cite{devroye1997random} proposes several general algorithms for random vector generation based on the accept-reject algorithm when the sample density is OU. \cite{sager1982nonparametric} proves the existence and consistency of the nonparametric MLE for the probability density under the OU constraint. \cite{polonik1998silhouette} derives a new graphical representation of the nonparametric MLE for the probability density under the OU constraint and proves the equivalence of MLE and a density estimator called ``silhouette". \cite{biau2003risk} studies the minimax lower bound with the $L_1$ distance for estimating an OU density and proposes Birg\'{e}'s multivariate histogram estimate which is minimax optimal. \cite{gao2007entropy} derives upper and lower bounds for the metric entropy and bracketing entropy of the set of bounded OU functions on $[0,1]^d$ under $L_p$ norms. However, none of these works propose the more general Definition \ref{OU mix} that does not require the existence of density. To our best knowledge, we are the first to propose the mixture representation in Definition \ref{OU mix}, prove the equivalence between Definitions \ref{OU def} and \ref{OU mix} in the presence of probability densities (Theorem \ref{Choquet_OU}), and study the relation among star unimodality, block unimodality and OU in general, i.e., without requiring the existence of probability densities (Proposition \ref{inclusion_unimodalities}), as well as the extensions in Appendix \ref{sec:general OU}. Moreover, by the same token, we are also the first to study DRO with the OU constraint.
% (ZHENYUAN: I ADDED SOME WORDS TO MAKE THE FOLLOWING SENTENCE MORE PRECISE. IN PARTICUALR, PEOPLE ALREADY NOTE THE RELATION BETWEEN BLOCK, OU AND STAR WHEN THEY HAVE DENSITIES.)

Finally, before we go to the optimization details in the next section, let us discuss the difficulties and our main contributions on solving DRO with OU constraint on a high level. First of all, there is no off-the-shelf method to solve (an infinite-dimensional) DRO problem with geometric shape constraints. In the literature, e.g., \cite{van2016generalized}, DRO problems with $\alpha$-unimodality are successfully reduced to finite-dimensional semidefinite programs by means of the Choquet representation in (\ref{Choquet_alpha_unimodality}). However, there is a significant difference between the Choquet representation of $\alpha$-unimodal distributions in (\ref{Choquet_alpha_unimodality}) and OU distributions in (\ref{equ_Choquet_OU}), which imposes additional difficulties on the reduction of the OU-DRO problem via the Choquet representation. In (\ref{Choquet_alpha_unimodality}), the uniform distributions $W_{\alpha\text{-}uni}(z)$ are parametrized by $z$ in an obvious way. On the contrary, although the distributions $W_{\bar{C}_{s}}$ in (\ref{equ_Choquet_OU}) is formally parametrized by $s$ (or $\bar{C}_{s}$), we only know $\bar{C}_{s}$ is OU but do not known what $\bar{C}_{s}$ looks like. Moreover, $\bar{C}_{s}$ even depends on the underlying density $f$. So the Choquet representation of OU distributions gives us less information and is not as ready to use as (\ref{Choquet_alpha_unimodality}). This difficulty essentially results from the complication of the extreme points of OU distributions. As Lemma \ref{extreme_point_OU} shows, all the uniform distributions on the OU sets with positive Lebesgue measure are extreme points in the class of OU distributions. However, there is no obvious way to parametrize OU sets, and hence parametrize the extreme points of OU distributions.

In view of the above challenge, our main contribution in terms of optimization methodology is that, in the bivariate case, we succeed in transforming the challenging OU-DRO problem into a finite-dimensional moment problem which can then be solved by methods such as generalized linear programming (GLP). Roughly speaking, we first use Choquet representation (\ref{equ_Choquet_OU}) to rewrite the DRO problem as an optimization problem whose feasible solutions are distributions on OU sets. This optimization problem is still difficult since the space of OU sets is too large (infinite-dimensional). Fortunately, each (closed) OU set in $\mathcal{D}_0$ is uniquely characterized by a left-continuous non-increasing function, which allows us to eliminate many suboptimal OU sets by analyzing a variational problem in the space of these functions. The remaining feasible OU sets then have a finite-dimensional parametrization, which gives us an equivalent finite-dimensional moment problem. Although such reduction can only be done in the bivariate case, note that multivariate extreme event analysis already faces all the drawbacks discussed in Section \ref{sec:challenges EEA multivariate} in the bivariate case.

\section{Orthounimodal Distributionally Robust Optimization}\label{sec:OU-DRO problem}
In this and the next section, we will study the shape-constrained DRO problem (\ref{basic}) where the geometric property in the tail part $\{X\ge u\}$ is characterized by OU. We call this problem orthounimodal distributionally robust optimization (OU-DRO). In the following, we first present the detailed formulation of OU-DRO and the associated rationale (Section \ref{sec:formulation}). Then we analyze its reformulation to a program with decision variables that are distributions on OU sets (Section \ref{sec:general reduction}). Section \ref{sec:opt theory} will continue to present the methodology in further reducing the reformulation to a finite-dimensional moment problem in the bivariate case.

% , including how to calibrate the constraints in the formulation using data (Section \ref{sec:calibration})
\subsection{Formulation}\label{sec:formulation}
% Based on the discussion in Section \ref{sec:OU}, we use OU as the . Moreover
First, as explained in Section \ref{sec:OU}, OU retains the same requirement on $\{X\ge u\}$ regardless of the mode $x_0$ as long as it is less than or equal to $u$. For simplicity, we just choose $x_0=u$ as the mode and the tail part is exactly $\mathcal{D}_0=\{x\ge x_0\}$ (hereafter we will use $x_0$ instead of $u$ when discussing the tail part).
% (ZHENYUAN: THIS TIME I CHOOSE TO FIRST FORMULATE THE PROBLEM AND THEN EXPLAIN IT. WE CAN ALSO CONSIDER TO FIRST EXPLAIN THE CONSTRAINTS AND THEN GIVE THE FORMULATION.)
Our OU-DRO problem is formulated as%
\begin{align}
\max\text{ }  &  P((X_{1},\ldots,X_{d})\in S)\nonumber\\
\text{subject to }  &  l_{\bar{F}}\leq\bar{F}(x_{0})\leq u_{\bar{F}}\nonumber\\
&  l_{X_{i}}\leq f_{X_{i}}(x_{i0})\leq u_{X_{i}},i=1,\ldots,d\label{DRO_problem_formulation1}\\
&  a_{i}\bar{F}(x_{0})\leq P(\underline{x}_{ji}\leq X_{j}\leq\bar{x}_{ji},j=1,\ldots,d)\leq b_{i}\bar{F}(x_{0}),i=1,\ldots,n\nonumber\\
&  f(x^{\prime})\ge f(x) \text{ for } x\ge x^{\prime}\ge x_0\nonumber
\end{align}
The decision variable of the problem is the unknown distribution $F$ of the random vector $X$. In the following, we will explain the notations and logic of the formulation (\ref{DRO_problem_formulation1}) and justify its statistical validity.
% state additional assumptions

For the objective probability, as we focus on the extreme event analysis, we assume $S$ is a subset of the tail part $\mathcal{D}_0$. Further, we assume $S$ has the following representation:
\begin{equation}
S=\{(x_{1},\ldots,x_{d})\in\mathcal{D}_{0}:x_{d}\geq g(x_{1},\ldots,x_{d-1})\},\label{form_S}
\end{equation}
for some known function $g:[x_{10},\infty)\times\cdots\times\lbrack x_{(d-1)0},\infty)\mapsto (-\infty, \infty]$. Without loss of generality, we can assume the range of $g$ is $[x_{d0},\infty]$; otherwise we can replace $g$ with $\max(g,x_{d0})$, which will not change the set $S$.

Next we discuss the constraints. The last one is the OU property. The others are auxiliary constraints on $F$ used to reduce conservativeness. Overall, we have two types of auxiliary constraints: density constraints and moment constraints, both of which are used to control the magnitude of the distribution in the tail. Due to the requirement that OU density $f(x)$ is non-increasing in each component of $x$ on $\mathcal{D}_{0}$, we can control the behavior of each component of $X$ in the tail by restricting the marginal densities at $x_{i0},i=1,\ldots,d$. This leads to the following density constraints in (\ref{DRO_problem_formulation1}):
%Due to the requirement that OU density $f(x)$ is non-increasing in each component of $x$ on $\mathcal{D}_{0}$, we can control the magnitude of $f(x)$ in the tail part by specifying the range of its maximum value $f(x_0)$. However, such a single constraint is so crude that it may not be able to reduce the conservativeness very well. So we enlarge the number of constraints by considering to control all the (truncated) marginal densities at $x_0$. More concretely, we impose the following density constraints:
\begin{equation*}
l_{X_{i}}\leq f_{X_{i}}(x_{i0})\leq u_{X_{i}},i=1,\ldots,d,
\end{equation*}
where $f_{X_{i}}(x_{i})$'s are the truncated marginal densities defined by%
\[
f_{X_{i}}(x_{i})=\int_{x_{10}}^{\infty}\cdots\int_{x_{(i-1)0}}^{\infty}\int_{x_{(i+1)0}}^{\infty}\cdots\int_{x_{d0}}^{\infty}f(x_{1},\ldots,x_{d})dx_{1}\cdots dx_{i-1}dx_{i+1}\cdots dx_{d},
\]
and $0\leq$ $l_{X_{i}}\leq u_{X_{i}}$ are constants. On the other hand, the moment constraints are used to control the magnitudes of some tail probabilities which play an important role in making the DRO problem nontrivial and reducing conservativeness (see Proposition \ref{triviality_without_moment_constraints} below). In (\ref{DRO_problem_formulation1}), we use the following moment constraints:
\begin{align}
&  l_{\bar{F}}\leq\bar{F}(x_{0})\leq u_{\bar{F}},\label{tail_probability_constraint}\\
&  a_{i}\bar{F}(x_{0})\leq P(\underline{x}_{ji}\leq X_{j}\leq\bar{x}_{ji},j=1,\ldots,d)\leq b_{i}\bar{F}(x_{0}),i=1,\ldots,n,\label{moment_constraints}
\end{align}
where $\bar{F}$ is the tail distribution function defined by $\bar{F}(x)=P(X\ge x)$ and $0\leq l_{\bar{F}} \leq u_{\bar{F}},0\leq a_{i}\leq b_{i}\leq1,x_{j0}\leq\underline{x}_{ji}\leq\bar{x}_{ji}$ are constants ($\bar{x}_{ji}$ can be infinity). Constraint (\ref{tail_probability_constraint}) controls the probability of the entire tail part and constraint (\ref{moment_constraints}) provides additional information about how the probability mass is distributed in this part. We note that constraint (\ref{moment_constraints}) in fact regards conditional probabilities since it can be written as
\[
a_{i}\leq P(\underline{x}_{ji}\leq X_{j}\leq\bar{x}_{ji},j=1,\ldots,d|X\ge x_0)\leq b_{i},i=1,\ldots,n.
\]
We can also consider constraints for unconditional probabilities as follows:
\begin{equation}
a_{i}\leq P(\underline{x}_{ji}\leq X_{j}\leq\bar{x}_{ji},j=1,\ldots,d)\leq b_{i},i=1,\ldots,n.\label{uncon_moment_constraints}
\end{equation}
However, for ease of illustration, we only consider the form (\ref{moment_constraints}) in most of our subsequent discussion, and will explain how to deal with the constraint (\ref{uncon_moment_constraints}) at the end of Section \ref{sec:opt theory}.
% for the discussion in Section \ref{sec:general reduction} and Section \ref{sec:opt theory}

Note that formulation (\ref{DRO_problem_formulation1}) only depends on the values in $\mathcal{D}_{0}$, so we can restrict our attention to truncated distributions on $\mathcal{D}_{0}$. By Lemma \ref{lemma:basic}, the OU-DRO problem (\ref{DRO_problem_formulation1}) provides a statistically valid upper bound on the true probability as long as the constraints in (\ref{DRO_problem_formulation1}) are statistically valid (joint) confidence intervals. We summarize this as:
\begin{corollary}
Suppose that
\[
P(\text{constraints in (\ref{DRO_problem_formulation1}) holds for the true distribution})\geq1-\alpha
\]
for some confidence level $1-\alpha$. Then the optimal value of (\ref{DRO_problem_formulation1}), called $Z^{\ast}$, satisfies%
\[
P(Z^{\ast}\geq Z)\geq1-\alpha,
\]
where $Z$ is the true rare event probability $P((X_{1},\ldots,X_{d})\in S)$. If the condition holds in the asymptotic sense, i.e.,
\begin{equation}
\liminf P(\text{constraints in (\ref{DRO_problem_formulation1}) holds for the true distribution})\geq1-\alpha,\label{constraint requirement}
\end{equation}
then we have
\begin{equation}
\liminf P(Z^{\ast}\geq Z)\geq1-\alpha\label{guarantee specific}
\end{equation}
where $\liminf$ refers to the limit as the data size grows to $\infty$.
\label{OU-DRO_validity}
\end{corollary}

All the parameters in formulation \eqref{DRO_problem_formulation1} can be readily calibrated using data so that \eqref{constraint requirement} holds and hence also the guarantee \eqref{guarantee specific}. This requires constructing confidence regions for expectation-type quantities and densities that is quite standard in statistics, and hence we delegate this discussion to Appendix \ref{sec:calibration}.
% discusses how to calibrate.
% Since the discussion is quite standard,
% (ZHENYUAN: DONE. ADD DISCUSSION ABOVE COROLLARY \ref{OU-DRO_validity}. NEED TO AT LEAST HINT THAT THE CONSTRAINTS WE PROPOSE CAN BE READILY CALIBRATED)

% (ZHENYUAN: DONE. LET'S JUST EXPLAIN HERE HOW TO CALIBRATE ALL THE CONSTRAINTS, INCLUDING THE BONFERRONI CORRECTION. THIS WOULD BE CLEANER)

% (ZHENYUAN: DONE. FOR SOMETHING LIKE THIS, WE SHOULD JUST WRITE THE FORMULA MATHEMATICALLY)

\subsection{Reduction of the Problem}\label{sec:general reduction}
The OU-DRO problem (\ref{DRO_problem_formulation1}) is an infinite-dimensional optimization program, and to proceed we need to reduce it to a tractable form. To begin with, we show that the lower bound density constraints $l_{X_{i}}\leq f_{X_{i}}(x_{i0})$ are redundant. In fact, for a density $f$ which only violates these lower bound density constraints, we can increase its values on $\partial \mathcal{D}_{0}$, i.e., $\cup_{i=1}^{d}\{(x_{1},\ldots,x_{d}):x_{i}=x_{i0},x_{j}\geq x_{j0},j\neq i\}$. Note that such modification will only increase $f_{X_{i}}(x_{i0})$'s but not affect the distribution and the feasibility of the OU constraint so it can change $f$ into a feasible solution. Thus we can remove these constraints and get the following equivalent formulation:%
\begin{align}
\max\text{ }  &  P((X_{1},\ldots,X_{d})\in S)\nonumber\\
\text{subject to }  &  l_{\bar{F}}\leq\bar{F}(x_{0})\leq u_{\bar{F}}\nonumber\\
&  f_{X_{i}}(x_{i0})\leq u_{X_{i}},i=1,\ldots,d\label{DRO_problem_formulation2}\\
&  a_{i}\bar{F}(x_{0})\leq P(\underline{x}_{ji}\leq X_{j}\leq\bar{x}_{ji},j=1,\ldots,d)\leq b_{i}\bar{F}(x_{0}),i=1,\ldots,n\nonumber\\
&  f(x^{\prime})\ge f(x) \text{ for } x\ge x^{\prime}\ge x_0\nonumber
\end{align}
As a byproduct, formulation (\ref{DRO_problem_formulation2}) reveals that, for many choices of the target rare-event set $S$, the moment constraint (\ref{moment_constraints}) is essential to make the OU-DRO problem nontrivial. This is stated by the following proposition.

\begin{proposition}
\label{triviality_without_moment_constraints} Consider the problem (\ref{DRO_problem_formulation2}) without the moment constraint (\ref{moment_constraints}). Suppose $0<u_{\bar{F}}\leq1$ and $u_{X_{i}}>0$. Suppose the rare-event set $S$ satisfies
\[
S=\{(x_{1},\ldots,x_{d})\in\mathcal{D}_{0}:x_{d}\geq g(x_{1},\ldots,x_{d-1})\},
\]
where $g:[x_{10},\infty)\times\cdots\times\lbrack x_{(d-1)0},\infty)\mapsto [x_{d0}, \infty)$ is bounded on compact sets.
Then the optimal value of this problem is $u_{\bar{F}}$.
\end{proposition}

Since the choice of the density of a distribution is not unique, we can choose a good one to ease our analysis for further reduction of the problem (\ref{DRO_problem_formulation2}). For our convenience, we add the following constraint to the problem (\ref{DRO_problem_formulation2}):
\begin{equation}
f(x)=\limsup_{y\downarrow x,y\in\mathcal{D}_{0}^{\circ}}f(y)\text{ for }x\in\partial\mathcal{D}_{0}.
\label{regularity_condition_of_density}%
\end{equation}
For a feasible density $f$ of problem (\ref{DRO_problem_formulation2}) that only violates (\ref{regularity_condition_of_density}), we can reset its values on $\partial\mathcal{D}_{0}$ according to (\ref{regularity_condition_of_density}). Such changes will not affect the distribution and the feasibility of the OU constraint so it will not affect the optimal value of the problem (\ref{DRO_problem_formulation2}). Moreover, let us focus on the subproblem with the equality constraint $\bar{F}(x_{0})=c$ instead of $l_{\bar{F}}\leq\bar{F}(x_{0})\leq u_{\bar{F}}$, where $c\in[l_{\bar{F}},u_{\bar{F}}]$ is a fixed positive number (which can be chosen at the end by solving the DRO repeatedly at different $c$ and applying a simple one-dimensional line search). Thus, the OU-DRO problem becomes
\begin{align}
\max\text{ }  &  P((X_{1},\ldots,X_{d})\in S)\nonumber\\
\text{subject to }  &  \bar{F}(x_{0})=c\nonumber\\
&  f_{X_{i}}(x_{i0})\leq u_{X_{i}},i=1,\ldots,d\label{DRO_subproblem_formulation1}\\
&  a_{i}\bar{F}(x_{0})\leq P(\underline{x}_{ji}\leq X_{j}\leq\bar{x}_{ji},j=1,\ldots,d)\leq b_{i}\bar{F}(x_{0}),i=1,\ldots,n\nonumber\\
&  f(x^{\prime})\ge f(x) \text{ for } x\ge x^{\prime}\ge x_0\nonumber\\
&  f(x)=\limsup_{y\downarrow x,y\in\mathcal{D}_{0}^{\circ}}f(y)\text{ for }x\in\partial\mathcal{D}_{0}.\nonumber
\end{align}

Next, Choquet representation (\ref{equ_Choquet_OU}) helps us rewrite problem (\ref{DRO_subproblem_formulation1}) by means of the sets $\bar{C}_{s}$ and the probability density $g(s)$. We introduce some needed notations. For any OU set $K$ about $x_{0}$, we define $K_{i}$ as the slice of $K$ on the plane $x_{i}=x_{i0}$, i.e., $K_{i}=\{(x_{1},\ldots,x_{i-1},x_{i+1},\ldots,x_{d}):(x_{1},\ldots,x_{i-1},x_{i0},x_{i+1},\ldots,x_{d})\in K\}$. For clarity, we write $\lambda_{d}(\cdot)$ and $\lambda_{d-1}(\cdot)$ as the Lebesgue measure on $\mathbb{R}^{d}$ and $\mathbb{R}^{d-1}$ respectively. We define $\lambda_{d-1}(K_{i})/\lambda_{d}(K)=0,i=1,\ldots,d$ if $K=\emptyset$ and define $\lambda_{d}(K^{\prime}\cap K)/\lambda_{d}(K)=0$ for any measurable set $K^{\prime}\subset\mathbb{R}^{d}$ if $\lambda_{d}(K)=0$. By Choquet representation (\ref{equ_Choquet_OU}), we reformulate the problem (\ref{DRO_subproblem_formulation1}) as follows:

\begin{lemma}
\label{conversion_to_OU_set_optimization} Problem (\ref{DRO_subproblem_formulation1}) can be rewritten as%
\begin{align}
\max\text{ }  &  c\int_{0}^{\infty}\frac{\lambda_{d}(S\cap\bar{C}_{s})}{\lambda_{d}(\bar{C}_{s})}g(s)ds\nonumber\\
\text{subject to }  &  \int_{0}^{\infty}\frac{\lambda_{d-1}(\bar{C}_{s,i})}{\lambda_{d}(\bar{C}_{s})}g(s)ds\leq\frac{u_{X_{i}}}{c},i=1,\ldots,d\label{DRO_OU_optimization_MD}\\
&  a_{i}\leq\int_{0}^{\infty}\frac{\lambda_{d}(\{x\in\mathcal{D}_{0}:\underline{x}_{ji}\leq x_{j}\leq\bar{x}_{ji},j=1,\ldots,d\}\cap\bar{C}_{s})}{\lambda_{d}(\bar{C}_{s})}g(s)ds\leq b_{i},i=1,\ldots,n,\nonumber
\end{align}
where $C_{s}=\{x\in\mathcal{D}_{0}:f(x)\geq s\}$, $\bar{C}_{s}$ is the closure of $C_{s}$, $\bar{C}_{s,i}$ is the slice of $\bar{C}_{s}$ on the plane $x_{i}=x_{i0}$, $g(s)=\lambda_{d}(\bar{C}_{s})/c$ is a probability density on $(0,\infty)$ and $f(x)$ is a density on $\mathcal{D}_{0}$ with total mass $c$ and satisfies the last two constraints in the problem (\ref{DRO_subproblem_formulation1}).
\end{lemma}

In the above problem, both $\{\bar{C}_{s},s>0\}$ and $g(s)$ are generated by $f$ and they possess some special structures, e.g., $\{\bar{C}_{s},s>0\}$ is a sequence of non-increasing closed OU sets. These structures are not easy to fully characterize and thus impose difficulties on solving problem (\ref{DRO_subproblem_formulation1}). So the next key step is to disentangle the dependence of $\{\bar{C}_{s},s>0\}$ and $g(s)$ on the density $f$ and generalize the choices of $\{\bar{C}_{s},s>0\}$ and $g(s)$. We have the following lemma:
\begin{lemma}\label{lemma_general_OU_optimization}
Consider the optimization problem
\begin{align}
\max\text{ }  &  c\int_{0}^{\infty}\frac{\lambda_{d}(S\cap R_{s})}{\lambda_{d}(R_{s})}dG(s)\nonumber\\
\text{subject to }  &  \int_{0}^{\infty}\frac{\lambda_{d-1}(R_{s,i})}{\lambda_{d}(R_{s})}dG(s)\leq\frac{u_{X_{i}}}{c},i=1,\ldots,d\label{general_OU_optimization_MD}\\
&  a_{i}\leq\int_{0}^{\infty}\frac{\lambda_{d}(\{x\in\mathcal{D}_{0}:\underline{x}_{ji}\leq x_{j}\leq\bar{x}_{ji},j=1,\ldots,d\}\cap R_{s})}{\lambda_{d}(R_{s})}dG(s)\leq b_{i},i=1,\ldots,n\nonumber\\
&  \text{all the integrands are measurable}\nonumber
\end{align}
with the decision variables $\{R_{s}\subset\mathcal{D}_{0},s>0\}$ and $G(s)$, where $\{R_{s},s>0\}$ is a sequence of closed OU sets about $x_0$ satisfying $\lambda_d(R_{s})\in(0,\infty),\lambda_{d-1}(R_{s,i})\in(0,\infty)$ for any $i=1,\ldots,d$ and $s>0$, and $G(s)$ is a probability distribution on the index set $(0,\infty)$. Then the optimal value of problem (\ref{general_OU_optimization_MD}) is not less than the optimal value of problem (\ref{DRO_OU_optimization_MD}).
\end{lemma}

% (ZHENYUAN: DO WE NEED TO EXPLICITLY WRITE THE CONSTRAINT ``ALL THE INTEGRANDS ARE MEASURABLE"?)
If we can solve problem (\ref{general_OU_optimization_MD}) and show its optimal solution is also feasible to problem (\ref{DRO_OU_optimization_MD}), then we know that the optimal values of both problems are the same and thus problem (\ref{DRO_OU_optimization_MD}) is also solved. Section \ref{sec:opt theory} shows that it is possible to solve problem (\ref{general_OU_optimization_MD}) in the bivarate case, which is our next focus.

\section{Reduction to Tractable Form for the Bivariate Case}\label{sec:opt theory}

In this section, we will show how to reduce problem (\ref{general_OU_optimization_MD}) to a finite-dimensional moment problem for the bivariate case. To simplify notations, we will use $(X,Y)$, $(x,y)$, $(x_{0},y_{0})$ instead of $(X_{1},X_{2})$, $(x_{1},x_{2})$, $(x_{10},x_{20})$. For reference, we explicitly write the DRO problem when $d=2$:%
\begin{align}
\max\text{ }  &  P((X,Y)\in S)\nonumber\\
\text{subject to }  &  \bar{F}(x_{0},y_{0})=c\nonumber\\
&  f_{X}(x_{0})\leq u_{X}\nonumber\\
&  f_{Y}(x_{0})\leq u_{Y}\label{DRO_problem_2D}\\
&  a_{i}\bar{F}(x_{0},y_{0})\leq P(x_{1i}\leq X\leq x_{2i},y_{1i}\leq Y\leq y_{2i})\leq b_{i}\bar{F}(x_{0},y_{0}),i=1,\ldots,n\nonumber\\
&  f(x^{\prime},y^{\prime})\geq f(x,y)\text{ if }x_{0}\leq x^{\prime}\leq x\text{ and }y_{0}\leq y^{\prime}\leq y,\nonumber
\end{align}
where $c\in\lbrack l_{\bar{F}},u_{\bar{F}}]$ is a fixed positive number. For any OU set $K\subset\mathcal{D}_{0}$ about $(x_{0},y_{0})$, we define $K^{X}=\sup\{x:(x,y_{0})\in K\}-x_{0}$ and $K^{Y}=\sup\{y:(x_{0},y)\in K\}-y_{0}$ as the $x$-intercept and $y$-intercept of $K$ within the domain $\mathcal{D}_{0}$ respectively. In the bivariate case, Lemma \ref{lemma_general_OU_optimization} reduces to the following corollary.
\begin{corollary}\label{cor_general_OU_optimization_2D}
The optimal value of problem (\ref{DRO_problem_2D}) is not greater than the optimal value of the following problem:
\begin{align}
\max\text{ }  &  c\int_{0}^{\infty}\frac{\lambda(S\cap R_{s})}{\lambda(R_{s})}dG(s)\nonumber\\
\text{subject to }  &  \int_{0}^{\infty}\frac{R_{s}^{Y}}{\lambda(R_{s})}dG(s)\leq\frac{u_{X}}{c}\nonumber\\
&  \int_{0}^{\infty}\frac{R_{s}^{X}}{\lambda(R_{s})}dG(s)\leq\frac{u_{Y}}{c}\label{general_OU_optimization_2D}\\
&  a_{i}\leq\int_{0}^{\infty}\frac{\lambda(\{(x,y):x_{1i}\leq x\leq x_{2i},y_{1i}\leq y\leq y_{2i}\}\cap R_{s})}{\lambda(R_{s})}dG(s)\leq b_{i},i=1,\ldots,n\nonumber\\
&  \text{all the integrands are measurable}\nonumber
\end{align}
where the decision variables are the sequence of closed OU sets $\{R_{s}\subset\mathcal{D}_{0},s>0\}$ and the distribution $G(s)$ on the index set $(0,\infty)$, and moreover $R_{s}$ satisfies $\lambda(R_{s})\in (0,\infty),R_{s}^{X}\in (0,\infty),R_{s}^{Y}\in (0,\infty)$ for any $s>0$.
\end{corollary}

We aim to solve problem (\ref{general_OU_optimization_2D}) and show that problems (\ref{DRO_problem_2D}) and (\ref{general_OU_optimization_2D}) have the same optimal value. As explained at the end of Section \ref{sec:OU}, our main idea to solve the problem (\ref{general_OU_optimization_2D}) is to eliminate suboptimal OU sets and show the remaining OU sets have a finite-dimensional parametrization. To be more specific, for any closed OU set $R_{0}$ satisfying $\lambda(R_0)\in (0,\infty),R_0^{X}\in (0,\infty),R_0^{Y}\in (0,\infty)$, we will find an alternative closed OU set $\tilde{R}$ (called the \emph{dominating OU set} of $R_{0}$) such that by replacing $R_{0}$ with $\tilde{R}$, the constraints in (\ref{general_OU_optimization_2D}) are still satisfied and moreover the objective value is at least as good. In other words, we want $\tilde{R}$ to satisfy $\lambda(\tilde{R})=\lambda(R_{0})$, $\tilde{R}^{X}\leq R_{0}^{X}$, $\tilde{R}^{Y}\leq R_{0}^{Y}$, $\lambda(\{(x,y):x_{1i}\leq x\leq x_{2i},y_{1i}\leq y\leq y_{2i}\}\cap\tilde{R})=\lambda(\{(x,y):x_{1i}\leq x\leq x_{2i},y_{1i}\leq y\leq y_{2i}\}\cap R_{0})$ for $i=1,\ldots,n$ and  $\lambda(S\cap\tilde{R})\geq\lambda(S\cap R_{0})$. Thus, instead of considering all the closed OU sets in problem (\ref{general_OU_optimization_2D}), it suffices to consider all the dominating OU sets $\tilde{R}$ obtained in this way. Besides, we will show $\tilde{R}$ has a finite-dimensional characterization, which reduces problem (\ref{general_OU_optimization_2D}) to a finite-dimensional moment problem.

Now, let us explain how to obtain $\tilde{R}$. Since $R_{0}$ is a closed OU set, it can be represented by%
\[
R_{0}=\{(x,y):y_{0}\leq y\leq h_{0}(x),x_{0}\leq x\leq x_{0}+R_{0}^{X}\}
\]
for some non-increasing left-continuous function $h_{0}:[x_{0},x_{0}+R_{0}^{X}]\mapsto\lbrack y_{0},y_{0}+R_{0}^{Y}]$ with $h_{0}(x_{0})=y_{0}+R_{0}^{Y}$. We define $h_{0}^{-1}(y)=\sup\{x:h_{0}(x)\geq y\}$ for $y$ such that $\{x:h_{0}(x)\geq y\}\neq\emptyset$. We sort $x_{0},x_{0}+R_{0}^{X},x_{1i},x_{2i},h_{0}^{-1}(y_{1i})$ and $h_{0}^{-1}(y_{2i})$ (only consider $x_{2i}<\infty$ and $h_{0}^{-1}(y_{1i}),h_{0}^{-1}(y_{2i})$ that are well-defined) and only keep one if some numbers are repeated. Suppose the sorted sequence is $x_{0}<x_{1}<\cdots<x_{n_{R_{0}}}<x_{n_{R_{0}}+1}=x_{0}+R_{0}^{X}$, where $n_{R_{0}}$ is the number of distinct values (except $x_0$ and $x_0+R_{0}^{X}$) satisfying $n_{R_{0}}\leq4n$ in general. Since we fix $R_{0}$, these $x_{i}$'s and $h_{0}(x)$ are also fixed. Now we consider the following constraints of a function $h:[x_{0},x_{0}+R_{0}^{X}]\mapsto\lbrack y_{0},y_{0}+R_{0}^{Y}]$:%
\begin{equation}
\left\{
\begin{array}
[c]{l}%
\int_{x_{i}}^{x_{i+1}}(h(x)-y_{0})dx=\lambda(\{(x,y):x_{i}\leq x\leq x_{i+1},y\geq y_{0}\}\cap R_{0}),i=0,1,\ldots,n_{R_{0}}\\
h_{0}(x_{i+1})\leq h(x)\leq h_{0}(x_{i}+),x\in(x_{i},x_{i+1}],i=0,1,\ldots,n_{R_{0}}\\
h\text{ is non-increasing and left-continuous with }h(x_{0})=h(x_{0}+),
\end{array}
\right.  \label{reduction_of_constraints}%
\end{equation}
where $h(x+)$ is the right limit of the function $h$ at $x$. Corresponding to this function $h$, we can define a closed OU set%
\begin{equation}
R:=\{(x,y):y_{0}\leq y\leq h(x),x_{0}\leq x\leq x_{0}+R_{0}^{X}\}.
\label{definition_R}%
\end{equation}
Then we have the following claim.

\begin{lemma}
\label{lemma_reduction_of_constraints} If an OU set $R$ is defined by (\ref{definition_R}) with $h(x)$ satisfying the constraints (\ref{reduction_of_constraints}), then $R$ satisfies $\lambda(R)=\lambda(R_{0}),R^{X}\leq R_{0}^{X}, R^{Y}\leq R_{0}^{Y}$ and $\lambda(\{(x,y):x_{1i}\leq x\leq x_{2i},y_{1i}\leq y\leq y_{2i}\}\cap R)=\lambda(\{(x,y):x_{1i}\leq x\leq x_{2i},y_{1i}\leq y\leq y_{2i}\}\cap R_{0})$ for $i=1,\ldots,n$.
\end{lemma}

Lemma \ref{lemma_reduction_of_constraints} tells us an OU set $R$ defined by (\ref{definition_R}) with $h(x)$ satisfying the constraints (\ref{reduction_of_constraints}) already satisfies all the requirements of the dominating OU set of $R_0$ except $\lambda(S\cap R)\geq\lambda(S\cap R_{0})$. Notice that $R_0$ and its corresponding function $h_0$ also satisfy the representation (\ref{definition_R}) and the constraints (\ref{reduction_of_constraints}). Therefore, if we optimize over all the OU sets satisfying the conditions in Lemma \ref{lemma_reduction_of_constraints} with the objective function $\lambda(S\cap R)$, the optimal solution must be a dominating OU set of $R_0$. By the representation (\ref{definition_R}) and the form of the rare-event set $S=\{(x,y)\in\mathcal{D}_{0}:y\geq g(x)\}$ for some known function $g:[x_{0},\infty)\mapsto\lbrack y_{0},\infty]$, the objective function $\lambda(S\cap R)$ can be equivalently written as
\begin{equation}
\lambda(S\cap R)=\int_{x_{0}}^{x_{0}+R_0^{X}}(h(x)-g(x))_{+}dx.\label{representation_obj}
\end{equation}
Therefore, the construction of the dominating OU set can be formulated as a variational problem as stated in the following lemma.
\begin{lemma}\label{find_dom_OU}
Suppose $h^*$ is the optimal solution of the following variational problem:
\begin{align}
\max\text{ }  &  \int_{x_{0}}^{x_{0}+R_0^{X}}(h(x)-g(x))_{+}dx\nonumber\\
\text{subject to }  &  \int_{x_{i}}^{x_{i+1}}(h(x)-y_{0})dx=\lambda(\{(x,y):x_{i}\leq x\leq x_{i+1},y\geq y_{0}\}\cap R_{0}),i=0,1,\ldots,n_{R_{0}} \label{function_optimization_problem}\\
&  h_{0}(x_{i+1})\leq h(x)\leq h_{0}(x_{i}+),x\in(x_{i},x_{i+1}],i=0,1,\ldots,n_{R_{0}}\nonumber\\
&  h\text{ is non-increasing and left-continuous with }h(x_{0})=h(x_{0}+)\nonumber
\end{align}
Then the closed OU set $\tilde{R}$ defined by the representation (\ref{definition_R}) with $h=h^*$ is a dominating OU set of $R_0$.
\end{lemma}

Next we explain how to characterize the optimal solution $h^*$ to problem (\ref{function_optimization_problem}). Notice that (\ref{function_optimization_problem}) is separable, i.e., it can be divided into $n_{R_{0}}+1$ subproblems on different intervals $(x_{i},x_{i+1}]$ for $i=0,1,\ldots,n_{R_{0}}$:
\begin{align}
\max\text{ }  &  \int_{x_{i}}^{x_{i+1}}(h(x)-g(x))_{+}dx\nonumber\\
\text{subject to }  &  \int_{x_{i}}^{x_{i+1}}h(x)dx=y_{0}(x_{i+1}-x_{i})+\lambda(\{(x,y):x_{i}\leq x\leq x_{i+1},y\geq y_{0}\}\cap R_{0})\label{subproblem_of_function_optimization_problem}\\
&  h_{0}(x_{i+1})\leq h(x)\leq h_{0}(x_{i}+),x\in(x_{i},x_{i+1}]\nonumber\\
&  h\text{ is non-increasing and left-continuous}\nonumber
\end{align}
Suppose $h_{i}^{\ast}(x)$ is the optimal solution to the $i$th subproblem. If we define $h^{\ast}:[x_{0},x_{0}+r_{X}]\mapsto\lbrack y_{0},y_{0}+r_{Y}]$ by%
\begin{equation}
h^{\ast}(x)=h_{i}^{\ast}(x),x\in(x_{i},x_{i+1}] \label{construction_h^*}
\end{equation}
and $h^{\ast}(x_{0})=h^{\ast}(x_{0}+)$, we can see $h^{\ast}$ satisfies all the constraints in (\ref{function_optimization_problem}) and it is optimal in each interval $(x_{i},x_{i+1}]$, which implies it is the optimal solution to (\ref{function_optimization_problem}). Therefore, in order to characterize $h^*$, it suffices to characterize each $h_i^*$.  Notice that all the subproblems (\ref{subproblem_of_function_optimization_problem}) have the following form%
\begin{align}
\max\text{ }  &  \int_{\underline{x}}^{\bar{x}}(h(x)-g(x))_{+}dx\nonumber\\
\text{subject to }  &  \int_{\underline{x}}^{\bar{x}}h(x)dx=C\label{geometric_problem}\\
&  \underline{b}\leq h(x)\leq\bar{b},x\in(\underline{x},\bar{x}]\nonumber\\
&  h\text{ is non-increasing and left-continuous}\nonumber
\end{align}
for $\underline{b}(\bar{x}-\underline{x})\leq C\leq\bar{b}(\bar{x}-\underline{x})$. The following lemma characterizes the optimal solution to problem (\ref{geometric_problem}).

\begin{lemma}\label{opt_geometric_problem}
Given a function $g:(\underline{x},\bar{x}]\mapsto(-\infty,\infty]$, the optimal solution $h^{\star}\left(  x\right)  $ to problem (\ref{geometric_problem}) exists and has the following form%
\[
h^{\star}\left(  x\right)  =\left\{
\begin{array}
[c]{c}%
y_{1}^{\star},\\
y_{2}^{\star},\\
y_{3}^{\star},
\end{array}
\left.
\begin{array}
[c]{l}%
\underline{x}<x\leq x_{1}^{\star}\\
x_{1}^{\star}<x\leq x_{2}^{\star}\\
x_{2}^{\star}<x\leq\bar{x}%
\end{array}
\right.  \right.
\]
for some $\underline{b}\leq y_{3}^{\star}\leq y_{2}^{\star}\leq y_{1}^{\star}\leq\bar{b}$ and $\underline{x}\leq x_{1}^{\star}\leq x_{2}^{\star}\leq \bar{x}$.
\end{lemma}

By (\ref{construction_h^*}), $h^*$ can be constructed by ``combining" $n_{R_{0}}+1$ optimal solution $h_i^*$'s, each of which is a step function with at most three steps by Lemma \ref{opt_geometric_problem}. In general, we have an upper bound on $n_{R_{0}}$: $n_{R_{0}}\le 4n$. This gives us the structure of $h^*$ and also a finite-dimensional parametrization of all the dominating OU sets.
\begin{corollary}\label{characterization_dom_OU}
Given a function $g:[x_0,\infty)\mapsto [y_0,\infty]$, the optimal solution $h^*(x)$ to the problem (\ref{function_optimization_problem}) exists and can be represented by the following step function:
\begin{equation}
h^*(x)\equiv h^*(x;z,w)=\left\{
\begin{array}[c]{l}%
y_{0}+\sum_{i=1}^{12n+3}w_{i},\\
y_{0}+\sum_{i=1}^{12n+2}w_{i},\\
y_{0}+\sum_{i=1}^{12n+1}w_{i},\\
\cdots\\
y_{0}+w_{1},
\end{array}
\left.
\begin{array}[c]{l}%
x_{0}\leq x\leq x_{0}+z_{1}\\
x_{0}+z_{1}<x\leq x_{0}+z_{1}+z_{2}\\
x_{0}+z_{1}+z_{2}<x\leq x_{0}+z_{1}+z_{2}+z_{3}\\
\\
x_{0}+\sum_{i=1}^{12n+2}z_{i}<x\leq x_{0}+\sum_{i=1}^{12n+3}z_{i}%
\end{array}
\right.  \right.  \label{12n+3_step_function}%
\end{equation}
for some $(z,w)\equiv(z_{1},\ldots,z_{12n+3},w_{1},\ldots,w_{12n+3})\in(0,\infty)^{12n+4}\times[0,\infty)^{12n+2}$. Moreover, all the dominating OU sets are contained in the class $\mathcal{R}^{\ast}=\{R_{z,w}:R_{z,w}=\{(x,y)\in\mathcal{D}_{0}:y_{0}\leq y\leq h^*(x;z,w)\}$ with $h^*$ defined in (\ref{12n+3_step_function})$\}$ and thus are fully parametrized by $(z,w)$.
\end{corollary}

%By Lemma \ref{opt_geometric_problem}, we know that each $h_{i}^{\ast}$ is a step function with at most three steps. Since we have $n_{R_{0}}+1\leq4n+1$ subproblems, $h^{\ast}$ is a non-increasing left-continuous step function with at most $3(4n+1)$ steps. All such functions can be represented by%
%\begin{equation}
%h(x)=\left\{
%\begin{array}[c]{l}%
%y_{0}+\sum_{i=1}^{12n+3}w_{i},\\
%y_{0}+\sum_{i=1}^{12n+2}w_{i},\\
%y_{0}+\sum_{i=1}^{12n+1}w_{i},\\
%\cdots\\
%y_{0}+w_{1},
%\end{array}
%\left.
%\begin{array}[c]{l}%
%x_{0}\leq x\leq x_{0}+z_{1}\\
%x_{0}+z_{1}<x\leq x_{0}+z_{1}+z_{2}\\
%x_{0}+z_{1}+z_{2}<x\leq x_{0}+z_{1}+z_{2}+z_{3}\\
%\\
%x_{0}+\sum_{i=1}^{12n+2}z_{i}<x\leq x_{0}+\sum_{i=1}^{12n+3}z_{i}%
%\end{array}
%\right.  \right.  \label{12n+3_step_function}%
%\end{equation}
%for some $z_{i}>0$ for $i\ge 1$, $w_1>0$ and $w_{i}\geq0$ for $i\ge 2$, which is a $(24n+6)$-dimensional parametrization. Thus, the collection of all dominating OU sets is $\mathcal{R}^{\ast}=\{\tilde{R}:\tilde{R}=\{(x,y)\in\mathcal{D}_{0}:y_{0}\leq y\leq h(x)\}$ where $h$ has the representation (\ref{12n+3_step_function})$\}$. Moreover, the domain of $(z,w)\equiv(z_{1},\ldots,z_{12n+3},w_{1},\ldots,w_{12n+3})$ ensures $\tilde{R}\in\mathcal{R}^{\ast}$ is bounded and has positive Lebesgue measure.

From the view of constructing dominating OU sets, it suffices to consider $R_{s}\in\mathcal{R}^{\ast}$ instead of all the closed OU sets in problem (\ref{general_OU_optimization_2D}). Corollary \ref{characterization_dom_OU} gives us a finite-dimensional parametrization of the OU sets in $\mathcal{R}^{\ast}$. This helps us reduce problem (\ref{general_OU_optimization_2D}) to an equivalent finite-dimensional moment problem.
\begin{proposition}\label{equiv_moment_problem}
Problem (\ref{general_OU_optimization_2D}) is equivalent to the following moment problem:
\begin{align}
\max\text{ }  &  cE_{Q}\left[  \frac{\sum_{i=1}^{12n+3}\int\left(  y_{0}+\sum_{j=1}^{12n+4-i}W_{j}-g(x)\right)  _{+}I\left(  x_{0}+\sum_{j=1}^{i-1}Z_{j}<x\leq x_{0}+\sum_{j=1}^{i}Z_{j}\right)  dx}{\sum_{i=1}^{12n+3}\sum_{j=1}^{12n+4-i}Z_{i}W_{j}}\right] \nonumber\\
\text{s.t. }  &  E_{Q}\left[  \frac{\sum_{i=1}^{12n+3}W_{i}}{\sum_{i=1}^{12n+3}\sum_{j=1}^{12n+4-i}Z_{i}W_{j}}\right]  \leq\frac{u_{X}}{c}\nonumber\\
&  E_{Q}\left[  \frac{\sum_{i=1}^{12n+3}Z_{i}}{\sum_{i=1}^{12n+3}\sum_{j=1}^{12n+4-i}Z_{i}W_{j}}\right]  \leq\frac{u_{Y}}{c} \label{finite_dimensional_moment_problem}\\
&  a_{k}\leq E_{Q}\left[  \frac{1}{\sum_{i=1}^{12n+3}\sum_{j=1}^{12n+4-i}Z_{i}W_{j}}\left\{  \sum_{i=1}^{12n+3}\int I\left(  \sum_{j=1}^{i-1}Z_{j}<x-x_0\leq \sum_{j=1}^{i}Z_{j},x_{1k}\leq x\leq x_{2k}\right)\right.  \right. \nonumber\\
&  \left.  \left.  \times\left(  \min\left(  y_{0}+\sum_{j=1}^{12n+4-i}W_{j},y_{2k}\right)  -\min\left(  y_{0}+\sum_{j=1}^{12n+4-i}W_{j},y_{1k}\right)  \right)  dx\right\}  \right]  \leq b_{k},k=1,\ldots,n\nonumber
\end{align}
where the decision variable $Q$ is the probability distribution of $(Z,W)\in(0,\infty)^{12n+4}\times[0,\infty)^{12n+2}$.
\end{proposition}

Note that the upper bound on $n_{R_{0}}$ determines the number of steps of $h^*$ in (\ref{12n+3_step_function}) and further determines the dimension of $(Z,W)$ in the moment problem (\ref{finite_dimensional_moment_problem}). If in some cases we can get a sharper upper bound on $n_{R_{0}}$ instead of $n_{R_{0}}\le 4n$, then we are able to reduce the dimension of the moment problem (\ref{finite_dimensional_moment_problem}). We will discuss this point in more detail at the end of this section.

By Corollary \ref{cor_general_OU_optimization_2D} and Proposition \ref{equiv_moment_problem}, we know the optimal value of the OU-DRO problem (\ref{DRO_problem_2D}) is not greater than the optimal value of the moment problem (\ref{finite_dimensional_moment_problem}). Now let us show the other direction, i.e., the optimal value of the OU-DRO problem (\ref{DRO_problem_2D}) is not less than the optimal value of the moment problem (\ref{finite_dimensional_moment_problem}). Consider any feasible solution $Q$ to the problem (\ref{finite_dimensional_moment_problem}). We define an absolutely continuous OU distribution $P$ with total mass $c$ by
\begin{equation}
P=c\int W_{step}(z,w)dQ(z,w),\label{construct_feasible_solution}
\end{equation}
where $W_{step}(z,w)$ is the uniform distribution on the closed OU set $R_{z,w}=\{(x,y)\in\mathcal{D}_{0}:y_{0}\leq y\leq h^*(x;z,w)\}$ with $h^*$ defined in (\ref{12n+3_step_function}). Then we can see the objective value and the constraints for the density $f_P$ in problem (\ref{DRO_problem_2D}) are just the translation of those for $Q$ in problem (\ref{finite_dimensional_moment_problem}). Since $Q$ is feasible to (\ref{finite_dimensional_moment_problem}), $f_P$ must be feasible to (\ref{DRO_problem_2D}) with the same objective value, which means the optimal value of (\ref{DRO_problem_2D}) is not less than the optimal value of (\ref{finite_dimensional_moment_problem}). Combining the above two directions, we see that the optimal value of problem (\ref{DRO_problem_2D}) is equal to that of problem (\ref{finite_dimensional_moment_problem}). Finally, according to Theorem 3.2 in \cite{winkler1988extreme}, to find an optimal solution to (\ref{finite_dimensional_moment_problem}), it suffices to consider discrete probability measures with at most $n+3$ points in the support. So (\ref{finite_dimensional_moment_problem}) is equivalent to the following non-linear optimization:%
\begin{align}
\max\text{ }  &  c\sum_{l=1}^{n+3}p_{l}\frac{\sum_{i=1}^{12n+3}\int\left(y_{0}+\sum_{j=1}^{12n+4-i}w_{lj}-g(x)\right)  _{+}I\left(  x_{0}+\sum_{j=1}^{i-1}z_{lj}<x\leq x_{0}+\sum_{j=1}^{i}z_{lj}\right)  dx}{\sum_{i=1}^{12n+3}\sum_{j=1}^{12n+4-i}z_{li}w_{lj}}\nonumber\\
\text{s.t. }  &  \sum_{l=1}^{n+3}p_{l}\frac{\sum_{i=1}^{12n+3}w_{li}}{\sum_{i=1}^{12n+3}\sum_{j=1}^{12n+4-i}z_{li}w_{lj}}\leq\frac{u_{X}}{c}\nonumber\\
&  \sum_{l=1}^{n+3}p_{l}\frac{\sum_{i=1}^{12n+3}z_{li}}{\sum_{i=1}^{12n+3}\sum_{j=1}^{12n+4-i}z_{li}w_{lj}}\leq\frac{u_{Y}}{c} \label{nonlinear_problem}\\
&  a_{k}\leq\sum_{l=1}^{n+3}\frac{p_{l}}{\sum_{i=1}^{12n+3}\sum_{j=1}^{12n+4-i}z_{li}w_{lj}}\left\{  \sum_{i=1}^{12n+3}\int I\left(  \sum_{j=1}^{i-1}z_{lj}<x-x_0\leq \sum_{j=1}^{i}z_{lj},x_{1k}\leq x\leq x_{2k}\right)  \right. \nonumber\\
&  \left.  \times\left(  \min\left(  y_{0}+\sum_{j=1}^{12n+4-i}w_{lj},y_{2k}\right)  -\min\left(  y_{0}+\sum_{j=1}^{12n+4-i}w_{lj},y_{1k}\right)\right)  dx\right\}  \leq b_{k},k=1,\ldots,n\nonumber\\
&  \sum_{l=1}^{n+3}p_{l}=1,p_{l}\geq0,l=1,\ldots,n+3\nonumber\\
&  w_{l1}>0,w_{li}\geq0,l=1,\ldots,n+3,i=2,\ldots,12n+3,\nonumber\\
&  z_{li}>0,l=1,\ldots,n+3,i=1,\ldots,12n+3,\nonumber
\end{align}
where $p_{l},w_{li},z_{li}$ are the decision variables.

The results in this section are summarized in the following main theorem.

\begin{theorem}\label{main_theorem}
We have the following:
\begin{description}
\item[(1)]The OU-DRO problem (\ref{DRO_problem_2D}), the moment problem (\ref{finite_dimensional_moment_problem}) and the non-linear optimization problem (\ref{nonlinear_problem}) have the same optimal value. If $Q^*$ is the optimal solution to the moment problem (\ref{finite_dimensional_moment_problem}), then the density of $P^*$ defined in (\ref{construct_feasible_solution}) with $Q$ replaced by $Q^*$ is the optimal solution to the OU-DRO problem (\ref{DRO_problem_2D}).
\item[(2)]Consider the OU-DRO problem (\ref{DRO_problem_2D}) with $\bar{F}(x_{0},y_{0})=c$ replaced by $l_{\bar{F}}\leq\bar{F}(x_{0},y_{0})\leq u_{\bar{F}}$. Its optimal value is equal to the optimal value of the non-linear optimization problem (\ref{nonlinear_problem}) with an additional decision variable $c$ and an additional constraint $l_{\bar{F}}\leq c\leq u_{\bar{F}}$. If $c^*,p_{l}^*,w_{li}^*,z_{li}^*$ is the optimal solution to this non-linear optimization, then the optimal solution to the DRO problem is the density of $P^*$ defined in (\ref{construct_feasible_solution}) where $c=c^*$ and $Q$ is the discrete distribution on $(z_{l1}^*,\ldots,z_{l,12n+3}^*,w_{l1}^*,\ldots,w_{l,12n+3}^*),l=1,\ldots,n+3$ with probability mass $p_{l}^*,l=1,\ldots,n+3$.
\end{description}
\end{theorem}

%\subsection{Extensions of the Formulation (\ref{DRO_problem_2D})}

%In the formulation (\ref{DRO_problem_2D}), we can generalize the constraint%
%\begin{equation}
%a_{i}\bar{F}(x_{0},y_{0})\leq P(x_{1i}\leq X\leq x_{2i},y_{1i}\leq Y\leq y_{2i})\leq b_{i}\bar{F}(x_{0},y_{0}) \label{original_moment_constraint}%
%\end{equation}
%by setting $x_{2i}=\infty$ or $y_{2i}=\infty$, which gives the following three kinds of constraints
%\[
%a_{i}\bar{F}(x_{0},y_{0})\leq P(X\geq x_{1i},y_{1i}\leq Y\leq y_{2i})\leq b_{i}\bar{F}(x_{0},y_{0}),
%\]%
%\[
%a_{i}\bar{F}(x_{0},y_{0})\leq P(x_{1i}\leq X\leq x_{2i},Y\geq y_{1i})\leq b_{i}\bar{F}(x_{0},y_{0}),
%\]
%and%
%\[
%a_{i}\bar{F}(x_{0},y_{0})\leq P(X\geq x_{1i},Y\geq y_{1i})\leq b_{i}\bar{F}(x_{0},y_{0}).
%\]
%The above four kinds of constraints are regarding the conditional probabilities since they can be written as%
%\[
%a_{i}\leq P(x_{1i}\leq X\leq x_{2i},y_{1i}\leq Y\leq y_{2i}|X\geq x_{0},Y\geq y_{0})\leq b_{i},
%\]
%where $x_{2i}$ or $y_{2i}$ could be $\infty$. The second extension is to consider the constraints regarding the unconditional probabilities, i.e.,
%\[
%a_{i}\leq P(x_{1i}\leq X\leq x_{2i},y_{1i}\leq Y\leq y_{2i})\leq b_{i},
%\]
%where $x_{2i}$ or $y_{2i}$ could be $\infty$.

% (ZHENYUAN: THE DISCUSSION ABOVE IS BASED ON SOME DETAILS ON SECTION \ref{sec:opt theory}. SO I WONDER IF IT'S BETTER TO DISCUSS THE DIFFICULTIES FOR HIGHER DIMENSIONS AT THE END OF SECTION \ref{sec:opt theory}.)

Theorem \ref{main_theorem} suggests two approaches to solve the OU-DRO problem (\ref{DRO_problem_2D}). One is to solve the moment problem (\ref{finite_dimensional_moment_problem}) by methods such as GLP described in, e.g., Section 3 of \cite{birge1991bounding}. Another is to solve the non-linear program (\ref{nonlinear_problem}). Similarly, in order to solve the DRO problem (\ref{DRO_problem_2D}) with $\bar{F}(x_{0},y_{0})=c$ replaced by $l_{\bar{F}}\leq\bar{F}(x_{0},y_{0})\leq u_{\bar{F}}$, one way is to solve the non-linear program described in Theorem \ref{main_theorem}. An alternative way is to discretize $c\in[l_{\bar{F}},u_{\bar{F}}]$, solve a collection of moment problems (\ref{finite_dimensional_moment_problem}) each with a discretized value of $c$ via GLP, and search for the best discretized $c$, which would lead to an approximation of the optimal value of the DRO problem. In practice, we observe that solving the moment problem (\ref{finite_dimensional_moment_problem}) by GLP gives a solution with better quality so we will use this method in numerical experiments.

Now we discuss some variants of Theorem \ref{main_theorem}. First, we explain how to handle the unconditional moment constraint (\ref{uncon_moment_constraints}). Suppose in the DRO problem (\ref{DRO_problem_2D}), some of the moment constraints
\begin{equation}
a_{i}\bar{F}(x_{0},y_{0})\leq P(x_{1i}\leq X\leq x_{2i},y_{1i}\leq Y\leq y_{2i})\leq b_{i}\bar{F}(x_{0},y_{0}).\label{moment_constraint_2D}
\end{equation}
are replaced by the unconditional version:
\begin{equation}
a_{i}\leq P(x_{1i}\leq X\leq x_{2i},y_{1i}\leq Y\leq y_{2i})\leq b_{i}.\label{uncon_moment_constraint_2D}
\end{equation}
Since the constraint $\bar{F}(x_{0},y_{0})=c$ is included in the DRO problem (\ref{DRO_problem_2D}), (\ref{uncon_moment_constraint_2D}) is equivalent to
\begin{equation*}
\frac{a_{i}}{c}\bar{F}(x_{0},y_{0})\leq P(x_{1i}\leq X\leq x_{2i},y_{1i}\leq Y\leq y_{2i})\leq \frac{b_{i}}{c}\bar{F}(x_{0},y_{0}),i=1,\ldots,n.
\end{equation*}
Therefore, to handle the constraint (\ref{uncon_moment_constraint_2D}), we simply replace $a_{k}$ and $b_{k}$ with $a_{k}/c$ and $b_{k}/c$ in the third constraint of the problem (\ref{finite_dimensional_moment_problem}) or (\ref{nonlinear_problem}) and Theorem \ref{main_theorem} still holds.

Second, the dimension of the moment problem (\ref{finite_dimensional_moment_problem}) can be reduced in some cases. Suppose we have $n$ moment constraints in total which can be in the form of (\ref{moment_constraint_2D}) or (\ref{uncon_moment_constraint_2D}). We define $n^{\prime}$ as the total number of $x_{1i},x_{2i},y_{1i},y_{2i}$ that are equal to $x_0,\infty,y_0,\infty$ respectively, i.e.,
\[
n^{\prime}=\sum_{i=1}^{n}(I(x_{1i}=x_{0})+I(x_{2i}=\infty)+I(y_{1i}=y_{0})+I(y_{2i}=\infty)).
\]
Then in (\ref{reduction_of_constraints}) we have a sharper bound on $n_{R_0}$ given by  $n_{R_0}\le4n-n^{\prime}$. It follows that the function $h^*$ in (\ref{12n+3_step_function}) can be changed into a step function with at most $3(4n-n^{\prime}+1)$ steps. So the dimension of $Z$ and $W$ can be reduced to $3(4n-n^{\prime}+1)$ instead of $12n+3$.
% (ZHENYUAN: I THINK IT'S IMPORTANT AS IT CAN REDUCE THE DIMENSION OF THE MOMENT PROBLEM. IN FACT, I MADE USE OF THIS WHEN WRITING THE CODES FOR THE EXPERIMENTS. TO HIGHLIGHT IT, I MENTIONED IT AFTER PROPOSITION \ref{equiv_moment_problem}. IS THE SECOND POINT IMPORTANT? IF SO, IT MAY BE WORTH HIGHLIGHTING MORE IN THE WRITING)

Lastly, we explain why we cannot directly apply our above reduction procedure to $d\ge3$. As we have seen, the reduction procedure for $d=2$ mainly relies on finding and parametrizing the dominating OU sets, but both aspects face difficulties when $d\geq3$. For the first aspect, when $d=2$, the quantity $\lambda_{d-1}(R_{s,i})$ in (\ref{general_OU_optimization_MD}) has a simple geometric meaning, i.e., $x$-intercept or $y$-intercept within $\mathcal{D}_0$, which allows us to carefully design a solvable variational problem (\ref{function_optimization_problem}) to find the dominating OU set. However, for $d\ge3$, the geometric meaning of $\lambda_{d-1}(R_{s,i})$ is more opaque and not easy to handle, making it challenging to solve the analog of problem (\ref{function_optimization_problem}) in higher dimensions. For the second aspect, the dominating OU sets for $d=2$ are characterized by univariate non-increasing left-continuous step functions with bounded numbers of steps. Such step functions are finite-dimensionally parametrizable as it suffices to parametize the length and height of each step. However, when $d\ge3$, even if we can show the dominating OU sets can be characterized by multivariate step functions, the multivariate steps might bear ``shapes" that are not encoded in finite dimension. The generalization to $d\geq3$ thus appears to require new techniques and is worth a separate future work.
% Due to the two main difficulties when $d\ge3$, the reduction procedure only works for $d=2$.

% Appendix \ref{sec:variation} provides further discussion on a variant of Theorem \ref{main_theorem} that handles the unconditional moment constraint (\ref{uncon_moment_constraints}).

\section{Generalization to Partially Orthounimodal Distributionally Robust Optimization}\label{sec:POU-DRO problem}

In Sections \ref{sec:OU-DRO problem} and \ref{sec:opt theory}, we have discussed the formulation and reduction of the OU-DRO problem motivated by extremal estimation where all the dimensions of the random vector are in their tails. Nevertheless, for an event to be in the extreme, it is possible that only some but not all of the dimensions are in their tails. For example, the sets $\{(x,y):x\ge 4,y\ge -1\}$ and $\{(x,y):x\ge 4,x+y\ge 5\}$ are seen as rare events for the standard bivariate normal distribution $N(0,I_{2\times2})$. However, the density is only non-increasing in $x$ but not in $y$ in these regions since only $x$ is in its tail. OU cannot capture the feature of the probability density in such regions as it requires monotonicity in all dimensions. To remedy this, we design another new notion of multivariate unimodality called \emph{partial orthounimodality (POU)} which can be seen as the generalization of OU that allows for monotonicity in only part of all dimensions. From this, we formulate a POU-DRO problem in parallel to Section \ref{sec:OU-DRO problem} and, like in Section \ref{sec:opt theory}, we demonstrate how to solve it in a special case where only one dimension is in its tail. Due to paper length, the details of this investigation are delegated to Appendix \ref{sec:POU appendix}.

\section{Numerical Experiments}\label{sec:numerics}

In this section, we present some numerical experiments for OU-DRO problems in the bivariate case. In Section \ref{sec:true num exp}, we consider obtaining the upper bounds of rare-event probabilities when the constraints in the DRO problem (\ref{DRO_problem_2D}) is calibrated by the true distribution. In Section \ref{sec:data driven num exp}, we consider estimating rare-event probabilities in data-driven scenarios. In the experiments, we solve the moment problem (\ref{finite_dimensional_moment_problem}) by GLP.

\subsection{OU-DRO Problems Calibrated by True Distributions}\label{sec:true num exp}

We consider the DRO problem (\ref{DRO_problem_2D}) calibrated by the true distribution. We choose the true distribution of $(X,Y)$ to be a bivariate normal distribution $N(0,16I_{2\times2})$, where $I_{2\times2}$ is the identity matrix in $\mathbb{R}^{2\times2}$. The tail part is chosen as $[8,\infty)^2$. We aim to obtain upper bounds of the rare-event probabilities $P(X\geq8,Y\geq1.5X-2)$ and $P(X\geq8,Y\geq X+5)$. The OU-DRO problem is formulated as follows:
\begin{align*}
\max\text{ }  &  P\left(  (X,Y)\in S\right) \\
\text{s.t. }  &  \bar{F}\left(  8,8\right)  =5.176\times10^{-4}\\
&  f_{X}\left(  8\right)  \leq3.071\times10^{-4}\\
&  f_{Y}\left(  8\right)  \leq3.071\times10^{-4}\\
&  P\left(  8\leq X\leq8+i,Y\geq8\right)  =a_{i},i=1,2,3,4,5\\
&  P\left(  X\geq8,8\leq Y\leq8+\frac{3}{5}i\right)  =b_{i},i=1,2,3,4,5,\\
&  f(x^{\prime},y^{\prime})\geq f(x,y)\text{ if }8\leq x^{\prime}\leq x\text{ and }8\leq y^{\prime}\leq y,
\end{align*}
where $a=10^{-4}\times\lbrack2.395,3.763,4.498,4.869,5.044]$ and $b=10^{-4}\times\lbrack1.586,2.736,3.551,4.115,4.498]$. We choose the set $S$ as $S_{1}=\{(x,y):x\geq8,y\geq1.5x-2\}$ and $S_{2}=\{(x,y):x\geq 8,y\geq x+5\}$ corresponding to two target rare-event probabilities.

We solve each DRO problem 50 times (To solve the DRO problem easily, we perturb the moment constraint $P\left(  8\leq X\leq8+i,Y\geq8\right)=a_{i}$ a little, i.e., we replace them with $0.995a_{i}\leq P\left(  8\leq X\leq8+i,Y\geq8\right)  \leq 1.005a_{i}$; the constraint $P\left(  X\geq8,8\leq Y\leq8+3i/5\right)  =b_{i}$ is similarly perturbed). As a benchmark, the ground truths are $P(X\geq8,Y\geq1.5X-2)=5.028\times10^{-5}$ and $P(X\geq 8,Y\geq X+5)=5.341\times10^{-6}$. Figure \ref{num_exp_true_binormal} shows the optimal values of the DRO problems. For $P(X\geq8,Y\geq 1.5X-2)$, the upper bounds are valid and within twice the ground truth. For $P(X\geq8,Y\geq X+5)$, most upper bounds are valid and they are within one order ($10$ times) of the ground truth.

\begin{figure}[ptbh]
\centering
\subfigure[Estimated upper bounds of $P(X\geq8,Y\geq1.5X-2)$]{
\includegraphics[width=0.45\textwidth]{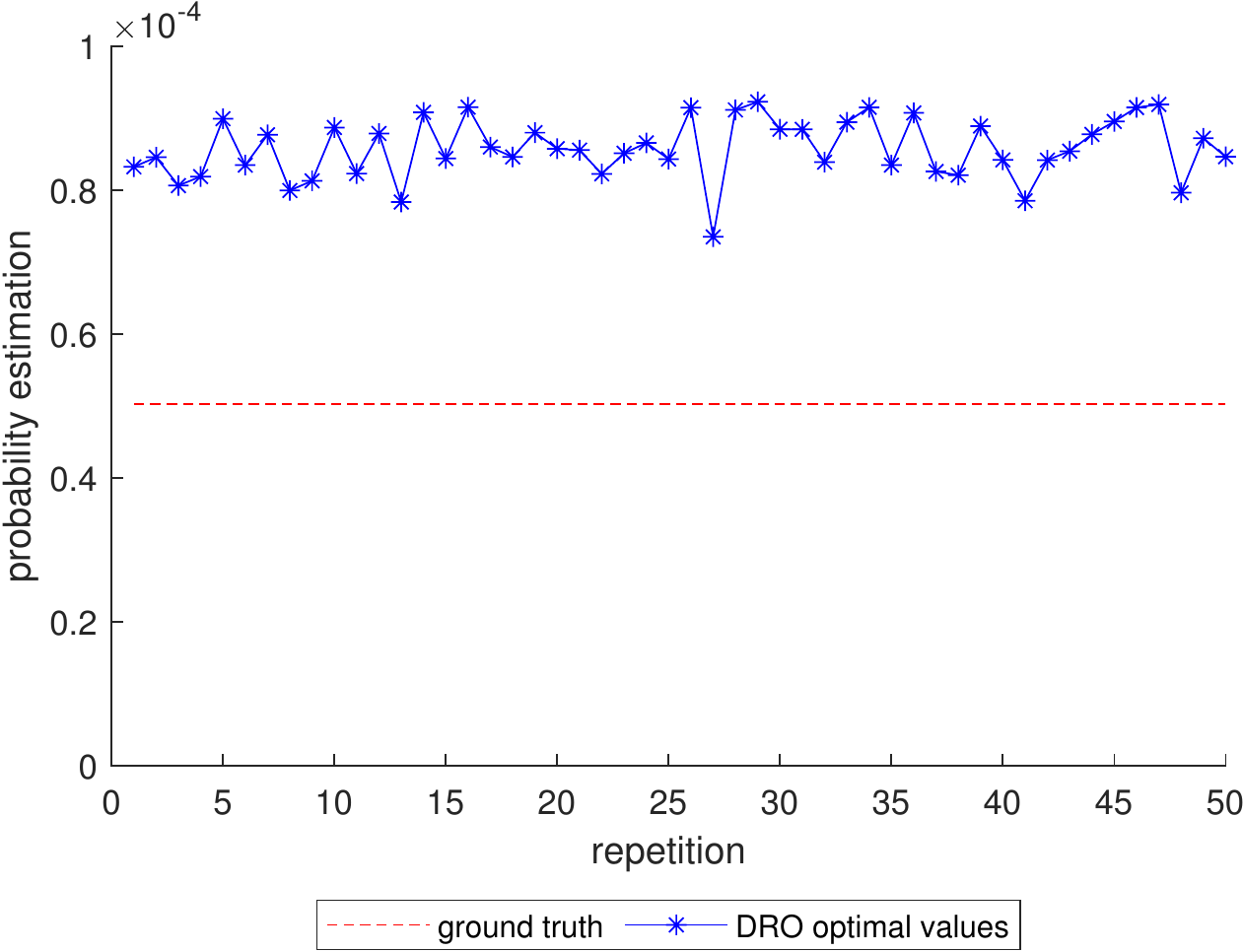}}
\subfigure[Estimated upper bounds of $P(X\geq8,Y\geq X+5)$]{
\includegraphics[width=0.45\textwidth]{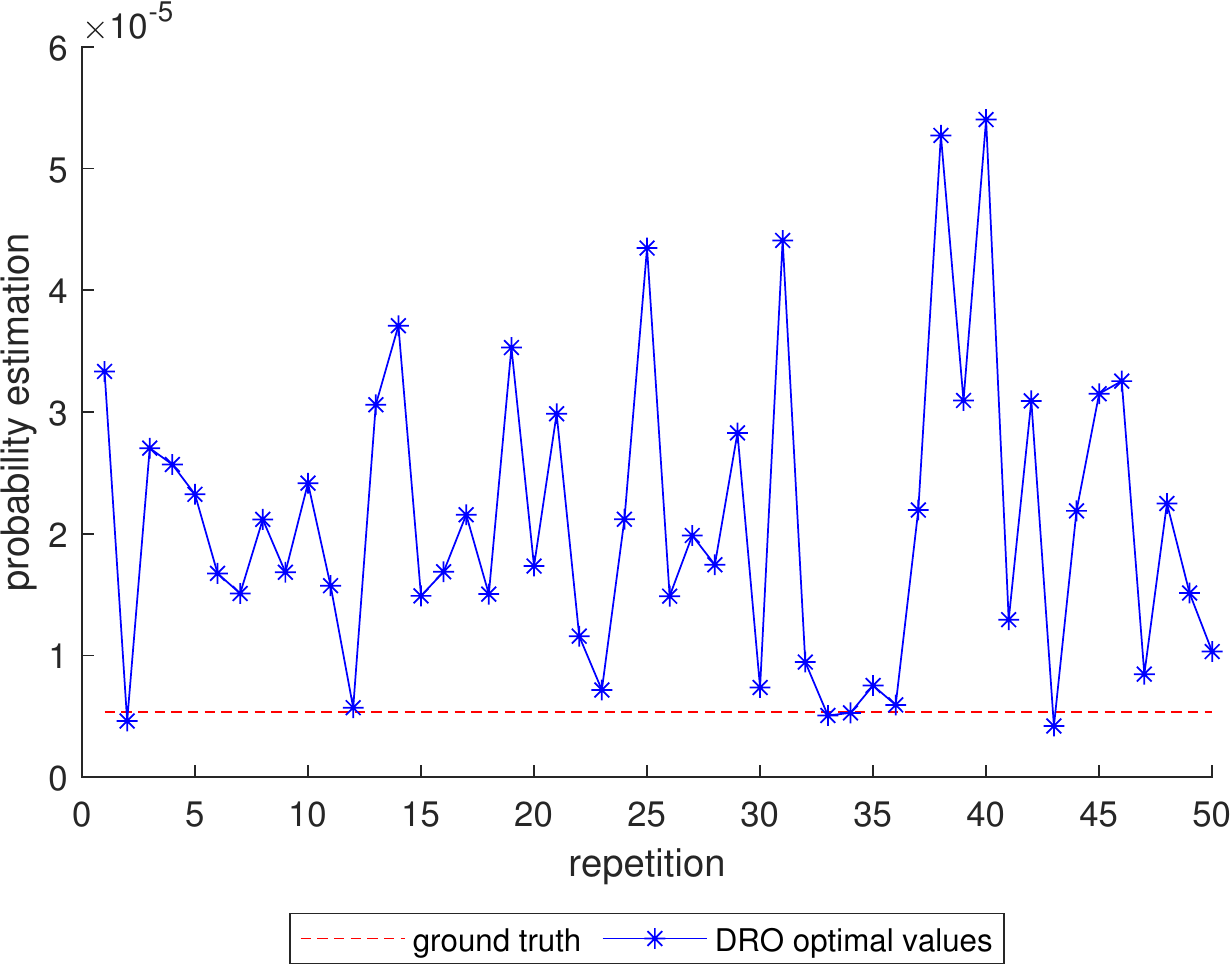}}
\caption{DRO problems with a true bivariate normal distribution. Each problem is solved $50$ times.}%
\label{num_exp_true_binormal}%
\end{figure}

% (ZHENYUAN: DONE. THE LEFT SIDE GRAPH OF FIGURE 2 BETTER HAS Y-AXIS STARTING FROM 0, SO THAT IT'S MORE OBVIOUS THAT OUR APPROACH WORKS OK)

\subsection{Data-Driven OU-DRO Problems}\label{sec:data driven num exp}
% (ZHENYUAN: DONE. THIS CAN BE SUITABLY CONNECTED TO THE CALIBRATION DISCUSSION EARLIER IN THE WRITING)
We consider data-driven OU-DRO problems. In other words, we can only get access to the data $(X_{1},Y_{1}),\ldots,(X_{m},Y_{m})$ generated from the unknown true distribution. Throughout this subsection, we consider the DRO problem (\ref{DRO_problem_formulation2}) where the moment constraints are specified as follows
\begin{align*}
&  a_{X,i}\leq P(x_{0}\leq X\leq x_{i},Y\geq y_{0}|X\geq x_{0},Y\geq y_{0})\leq b_{X,i},i=1,\ldots,n_{X}\\
&  a_{Y,j}\leq P(X\geq x_{0},y_{0}\leq Y\leq y_{j}|X\geq x_{0},Y\geq y_{0})\leq b_{Y,j},j=1,\ldots,n_{Y}
\end{align*}
%  (REVISON 1: REPHRASE THIS SENTENCE) where $\bar{F}(x_0,y_0)=c$ is replaced by the confidence interval of $\bar{F}(x_0,y_0)=c$, i.e., $l_{\bar{F}}\le \bar{F}(x_0,y_0) \le u_{\bar{F}}$ and
All the constraints in the DRO can be calibrated by the methods described in Appendix \ref{sec:calibration}. We still choose the true distribution to be the bivariate normal distribution $N(0,16I_{2\times2})$. We consider three data sizes: $m=10^4,10^5,10^6$. We choose the number of moment constraints as $n_{X}=n_{Y}=5$. For $x_{0},x_{1},x_{2},x_{3},x_{4},x_{5}$, we consider three different choices: the top $20,16.5,13,9.5,6,2.5$ percentiles of $X_{i}$, the top $10,8,6,4,2,1$ percentiles of $X_{i}$, and the top $5,4,3,2,1,0.5$ percentiles of $X_{i}$. For $y_{0},y_{1},y_{2},y_{3},y_{4},y_{5}$, we choose them to be the same percentiles of $Y_{i}$ as in the choices of $x_i$'s. We refer to these three settings as ``DRO truncated at 80\% percentile", ``DRO truncated at 90\% percentile" and ``DRO truncated at 95\% percentile" respectively. We aim at obtaining the upper bounds of three rare-event probabilities $P(X\geq7,Y\geq X+1)$, $P(X\geq8,Y\geq1.5X-2)$ and $P(X\geq8,Y\geq X+5)$ with ground truths $4.35\times10^{-4}$, $5.0275\times10^{-5}$ and $5.3408\times10^{-6}$ respectively. To solve the DRO problem via the moment problem (\ref{finite_dimensional_moment_problem}), we discretize the constraint $l_{\bar{F}}\leq\bar{F}(x_{0},y_{0})\leq u_{\bar{F}}$ as $\bar{F}(x_{0},y_{0})=c$ with $c=l_{\bar{F}},(l_{\bar{F}}+u_{\bar{F}})/2,u_{\bar{F}}$.

Figures \ref{num_exp_dd_10^4}-\ref{num_exp_dd_10^6} show the results. We observe that the results are better when we use higher levels to define $x_i$'s and $y_i$'s, which is reasonable since larger $x_i$'s and $y_i$'s provide more information about the distribution in the rare-event sets. Regarding the quality of the upper bounds, it becomes better with larger data sizes as we expect. Besides, we see that the upper bounds for $P(X\geq7,Y\geq X+1)$ appear accurate, especially for DRO truncated at 90\% percentile and 95\% percentile, most of which are within $2$ or $3$ times of the ground truth. As for the results of $P(X\geq8,Y\geq1.5X-2)$, DRO truncated at 95\% percentile performs reasonably well, whose optimal values are within one order ($10$ times) of the ground truth. But the results for $P(X\geq8,Y\geq X+5)$ seem more conservative, which can be attributed to the very small magnitude of this probability which makes the estimation more challenging. Nonetheless, the upper bounds provided by DRO across all settings correctly bound the ground truths, thus validating the statistical correctness of our approach.
%   fact that accurately estimating or bounding a rare-event probability at magnitude $5\times10^{-6}$ with at most $10^6$ samples is statistically very challenging
\begin{figure}[ptbh]
\centering
\subfigure[Estimated upper bounds of $P(X\geq7,Y\geq X+1)$]
{\includegraphics[width=0.45\textwidth]{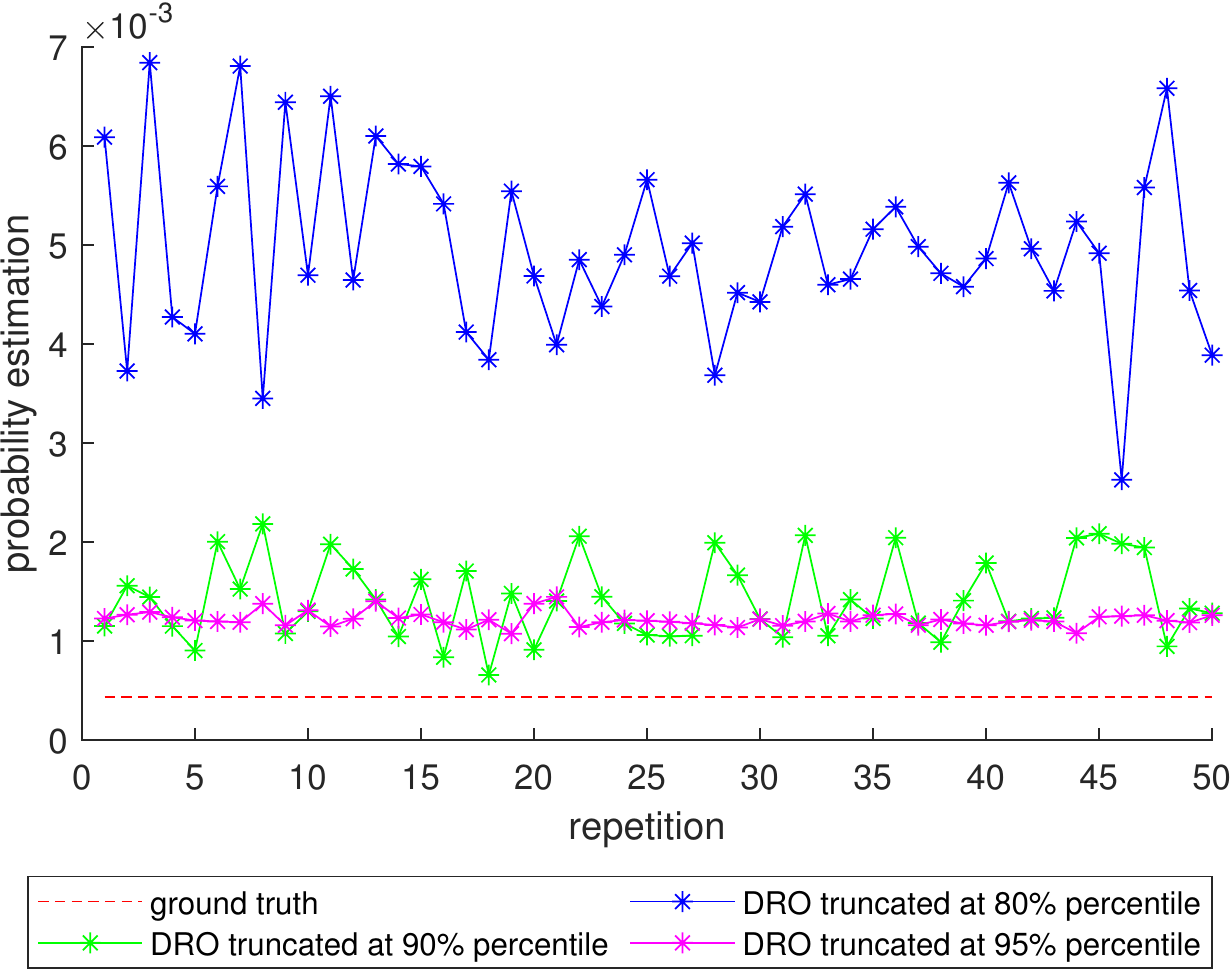}}
\subfigure[Estimated upper bounds of $P(X\geq8,Y\geq1.5X-2)$]
{\includegraphics[width=0.45\textwidth]{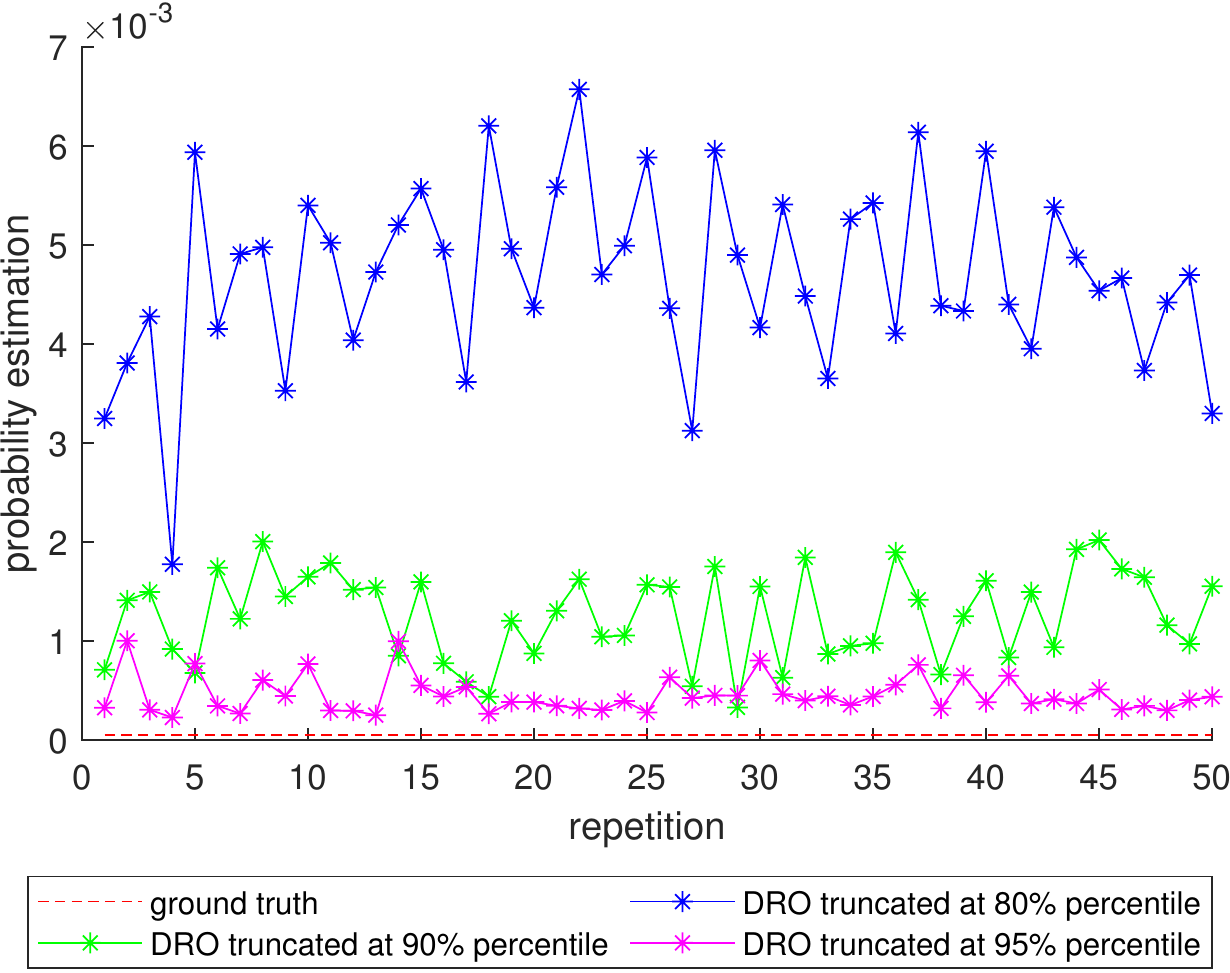}}
\subfigure[Estimated upper bounds of $P(X\geq8,Y\geq X+5)$]
{\includegraphics[width=0.45\textwidth]{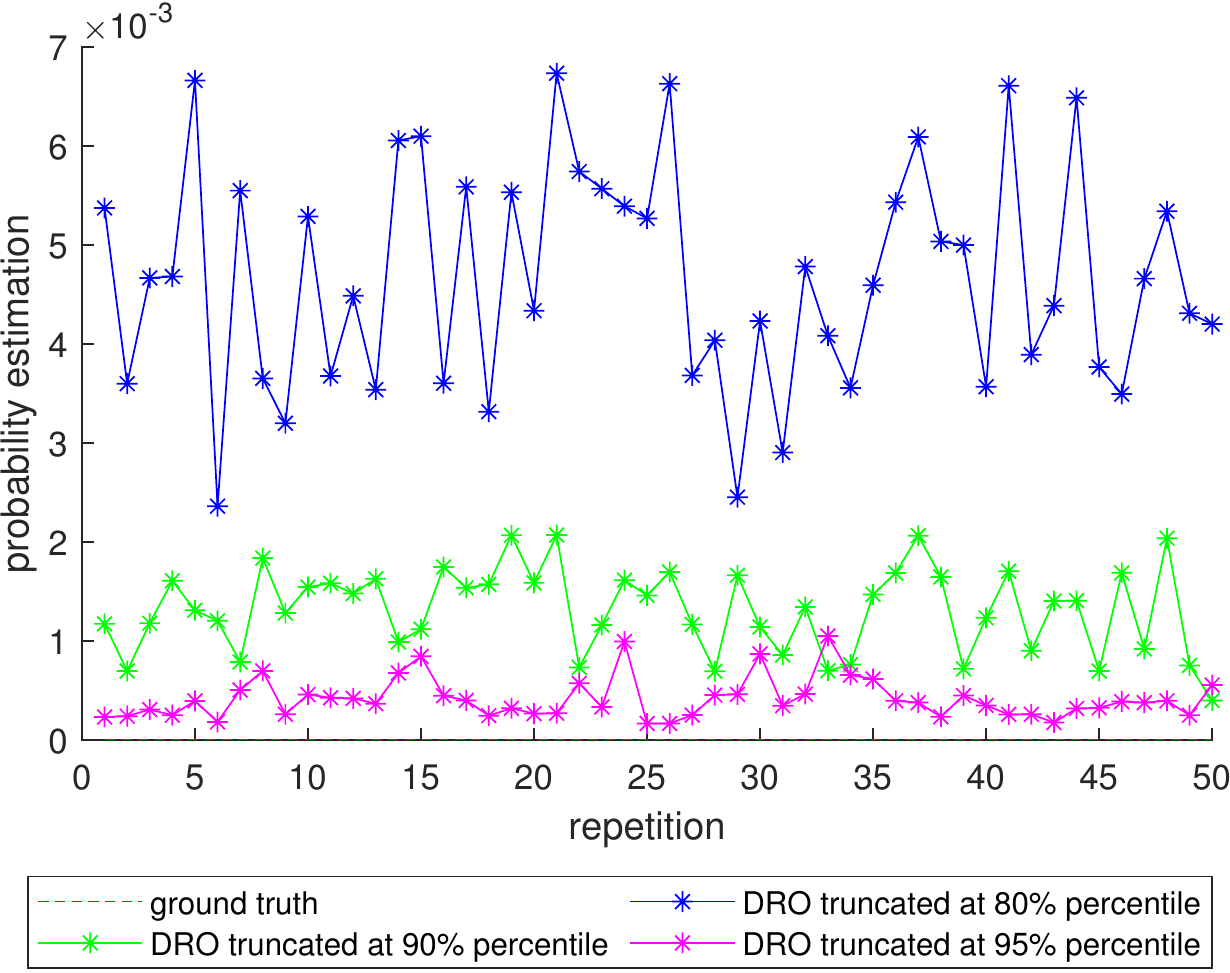}}
\caption{Data-driven DRO problems with $10^4$ samples from a bivariate normal distribution. Each problem is solved $50$ times.}
\label{num_exp_dd_10^4}
\end{figure}

\begin{figure}[ptbh]
\centering
\subfigure[Estimated upper bounds of $P(X\geq7,Y\geq X+1)$]
{\includegraphics[width=0.45\textwidth]{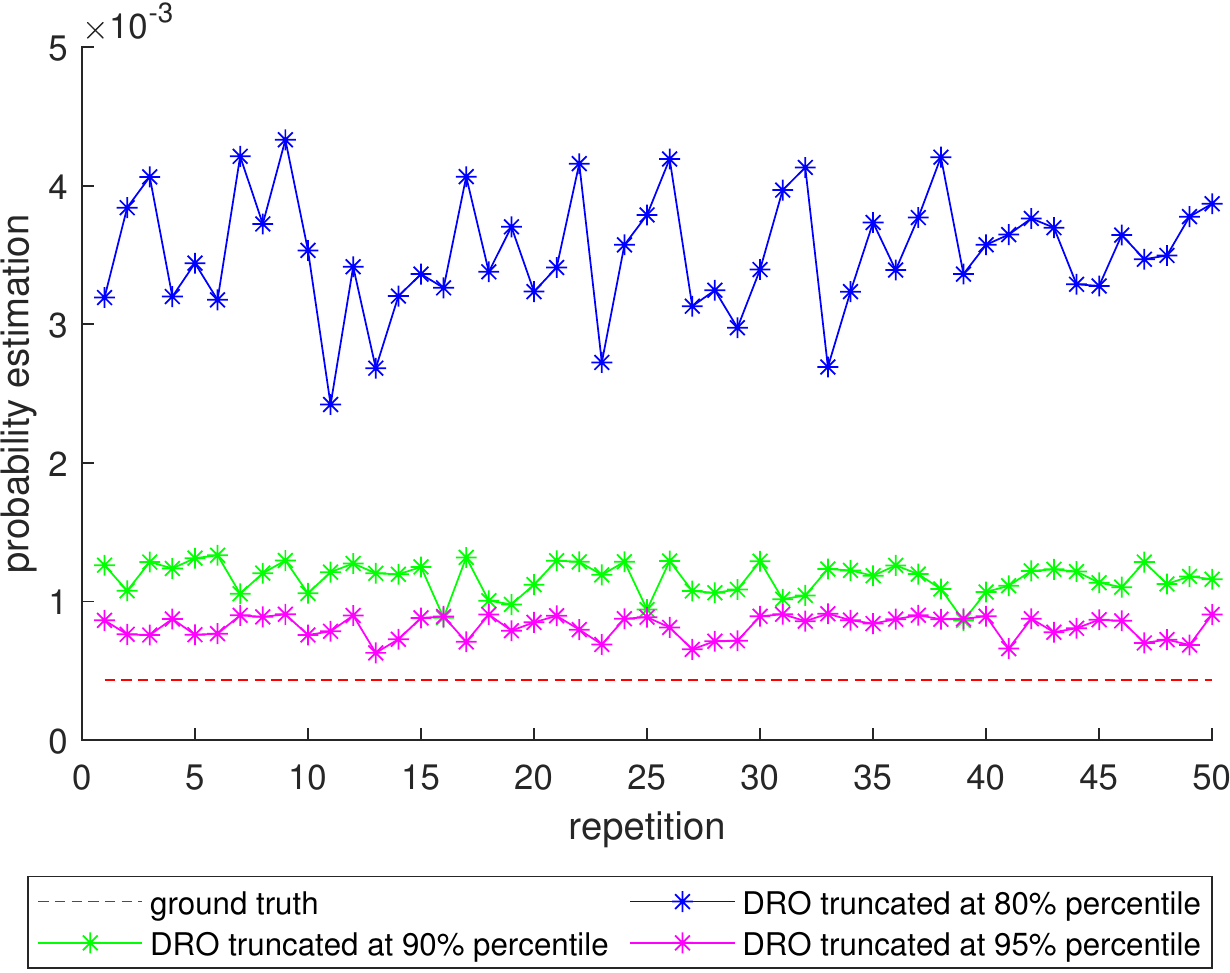}}
\subfigure[Estimated upper bounds of $P(X\geq8,Y\geq1.5X-2)$]
{\includegraphics[width=0.45\textwidth]{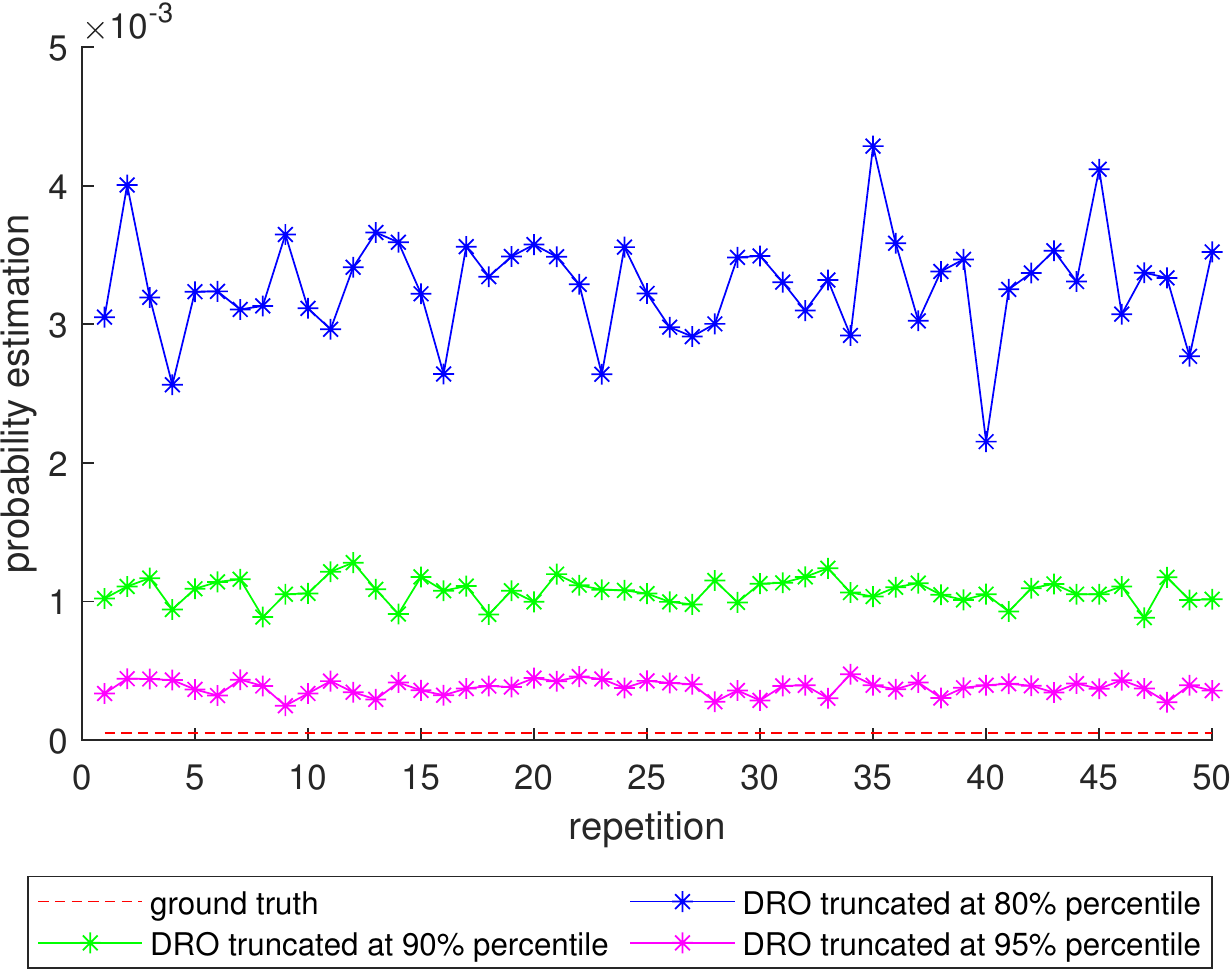}}
\subfigure[Estimated upper bounds of $P(X\geq8,Y\geq X+5)$]
{\includegraphics[width=0.45\textwidth]{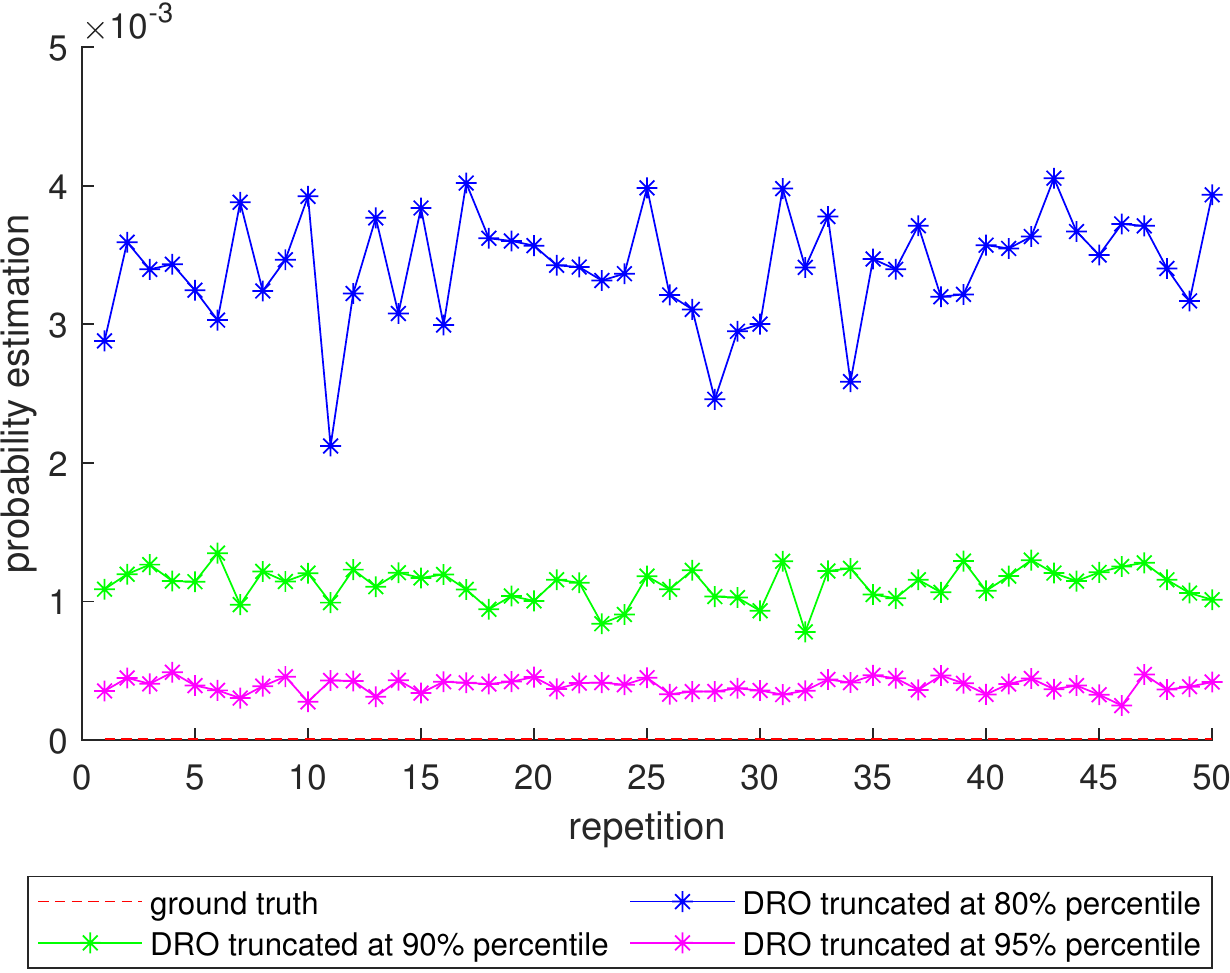}}
\caption{Data-driven DRO problems with $10^5$ samples from a bivariate normal distribution. Each problem is solved $50$ times.}
\label{num_exp_dd_10^5}
\end{figure}

\begin{figure}[ptbh]
\centering
\subfigure[Estimated upper bounds of $P(X\geq7,Y\geq X+1)$]
{\includegraphics[width=0.45\textwidth]{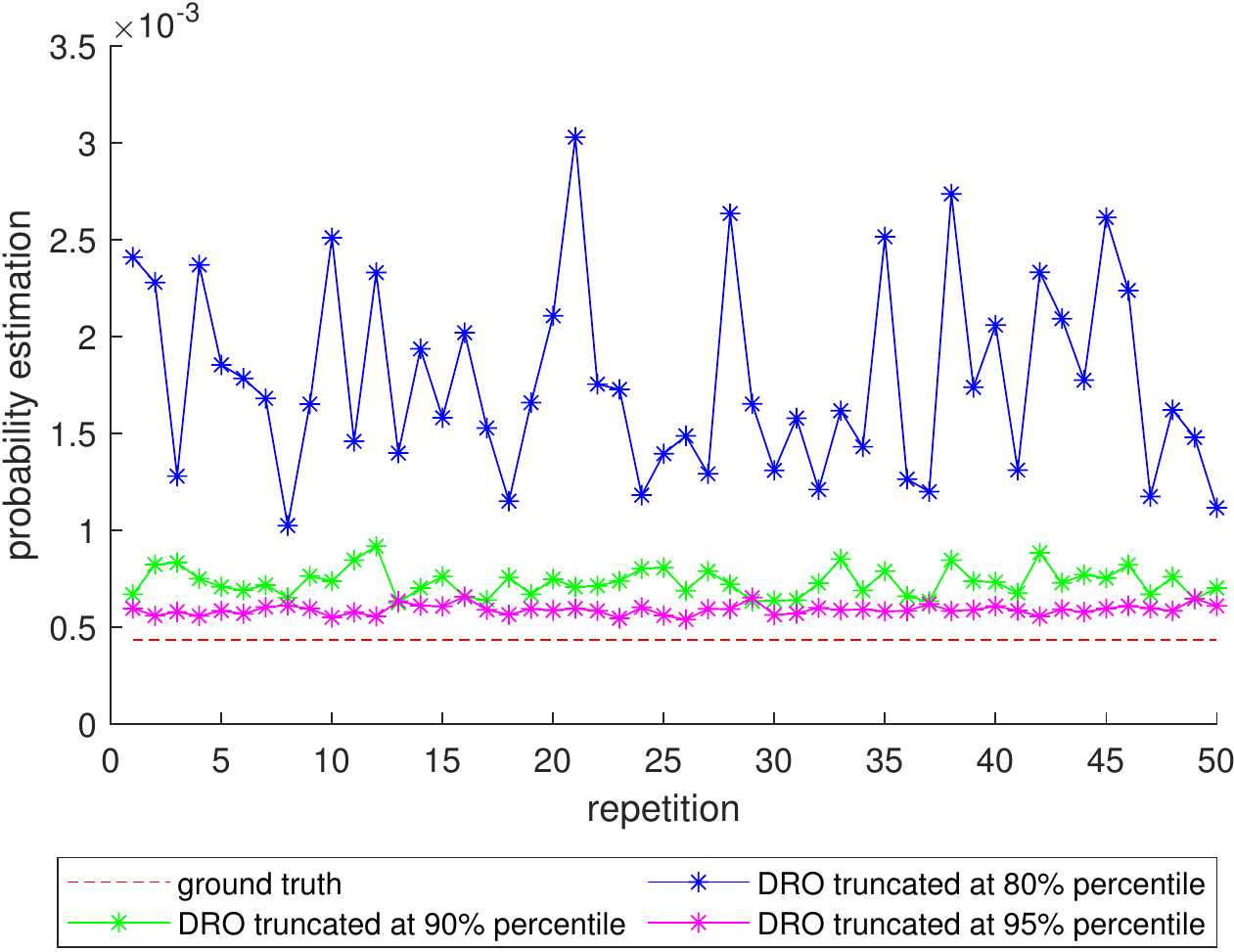}}
\subfigure[Estimated upper bounds of $P(X\geq8,Y\geq1.5X-2)$]
{\includegraphics[width=0.45\textwidth]{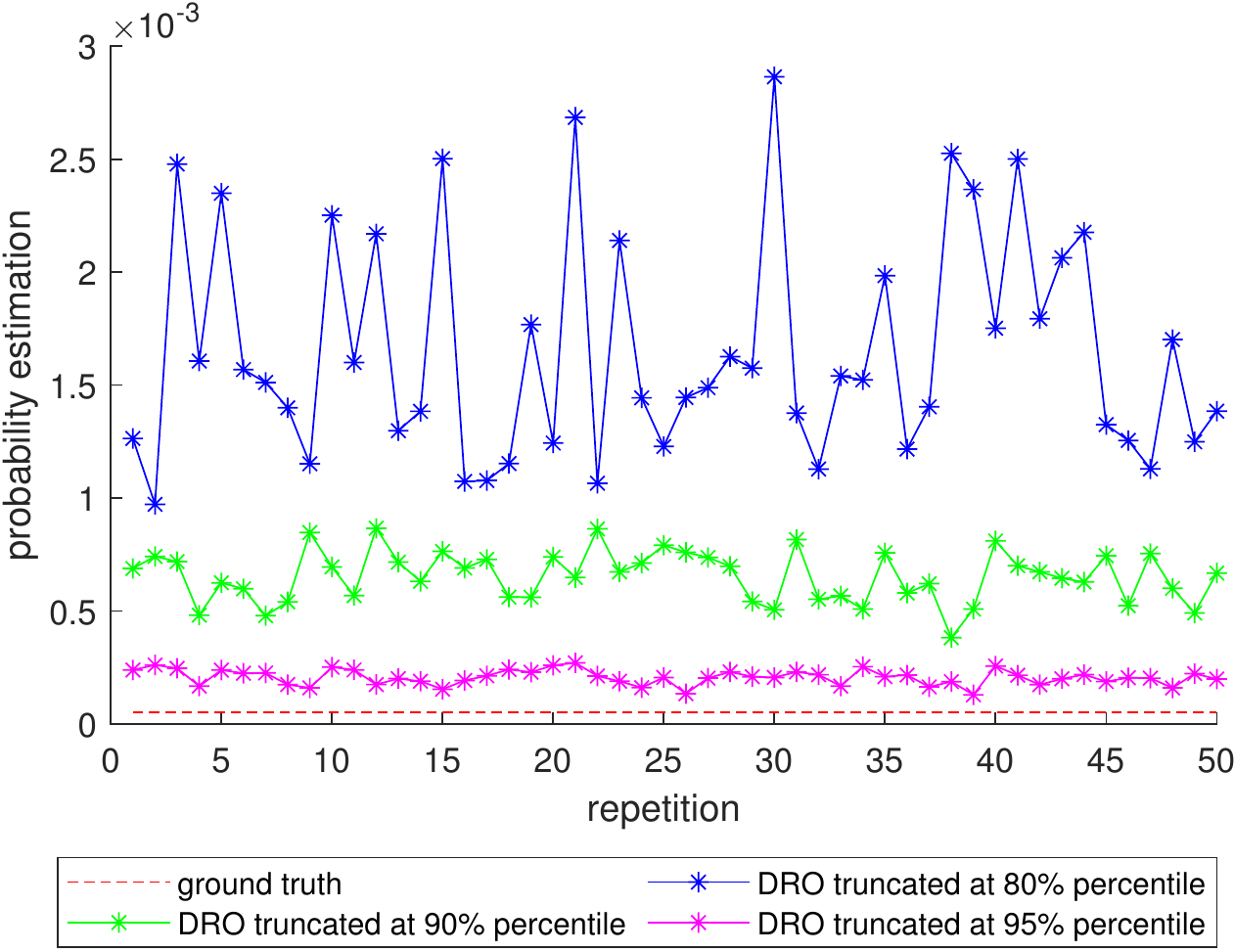}}
\subfigure[Estimated upper bounds of $P(X\geq8,Y\geq X+5)$]
{\includegraphics[width=0.45\textwidth]{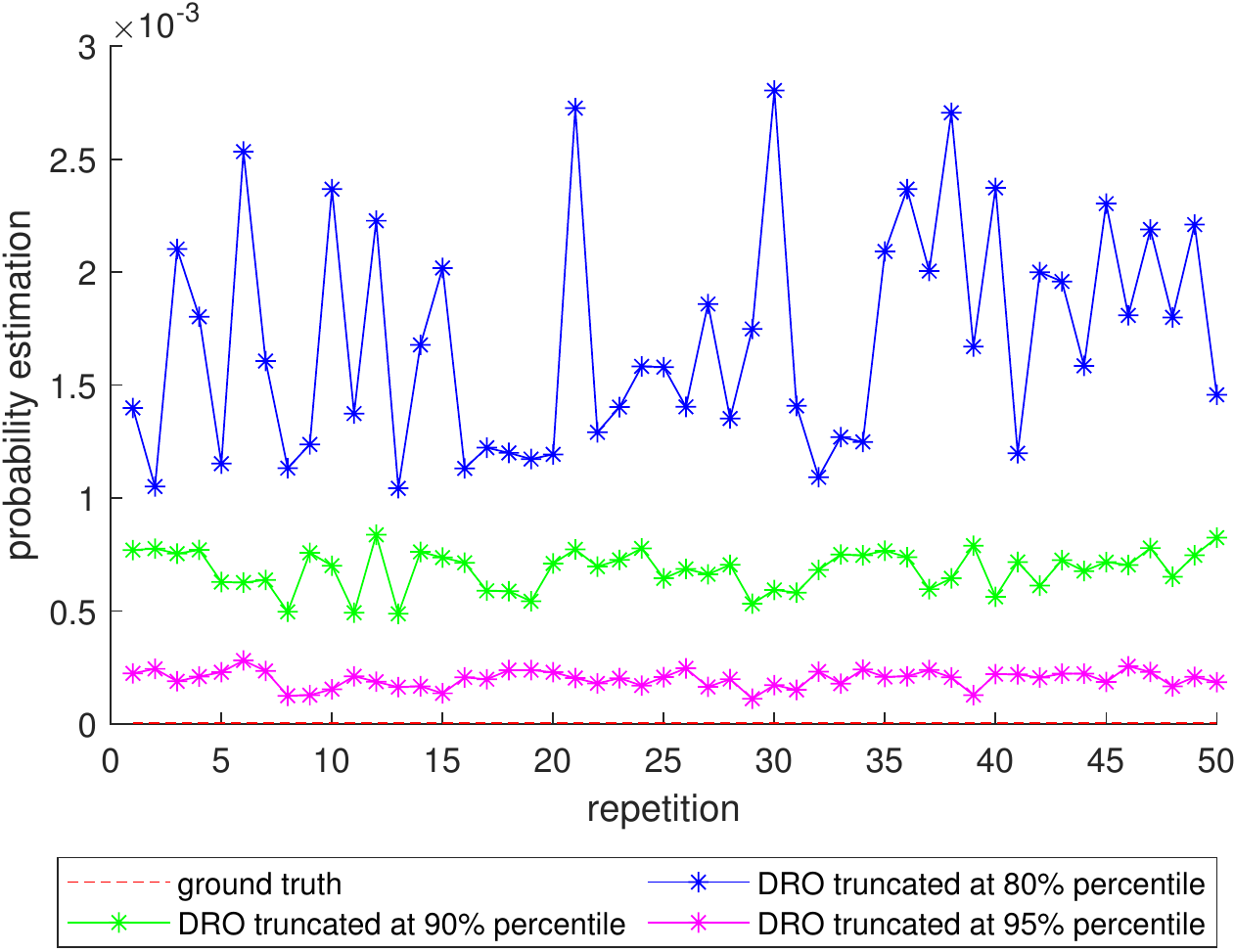}}
\caption{Data-driven DRO problems with $10^6$ samples from a bivariate normal distribution. Each problem is solved $50$ times.}
\label{num_exp_dd_10^6}
\end{figure}

\section{Conclusion}

This paper studied OU-DRO as a nonparametric alternative to existing methods for multivariate extreme event analysis. Our approach bypassed the bias-variance tradeoff and other technical complications faced by conventional multivariate extreme value theory, via the computation of worst-case upper bounds subject to shape constraints. We explained why OU is a suitable and natural shape constraint choice compared to other well-known multivariate unimodality notions such as star unimodality, block unimodality and $\alpha$-unimodality. We formulated the OU-DRO problem and presented how it can be reduced to a moment problem in the bivariate case by solving a specially designed variational problem to rule out suboptimal solutions. To extend our analysis to rare-event sets where only some of the dimensions are in their tails, we also proposed the use of POU-DRO and investigated its reduction to a tractable moment problem when only one dimension is in the tail. We demonstrated numerical results to show the statistical correctness and performance of our approach.

We suggest two directions for future work. One is the reduction of OU-DRO for general dimensions. This may require a deeper understanding on the geometry of OU sets in order to design a suitable variational problem to eliminate suboptimal solutions. Related to this is the reduction of POU-DRO for any number of dimensions being in the tail. Another direction is the alleviation of conservativeness in using DRO for extremal estimation, which involves further study on finding appropriate auxiliary constraints that can be calibrated from data while informative on tail behaviors.
% are imposed in the corresponding DRO.
% , which is more general and may build on the work for the first direction.

% Acknowledgments here
\section*{Acknowledgments}
We gratefully acknowledge support from the National Science Foundation under grants CAREER CMMI-1834710 and IIS-1849280.

\bibliographystyle{elsevier}
\bibliography{arxiv_OU_DRO}

\newpage

\begin{appendix}

\section*{Appendix}

The appendix is organized as follows. Appendix \ref{sec:calibration} details the calibration methods for all parameters using data in OU-DRO. Appendix \ref{sec:general OU} discusses OU distributions defined on the whole space $\mathbb{R}^d$ as the generalization of OU distributions defined on $\mathcal{D}_0$. Appendix \ref{sec:POU appendix} provides the details of POU distributions and POU-DRO discussed in Section \ref{sec:POU-DRO problem}. Appendix \ref{sec:proofs} contains the proofs of all our results.

% Appendix \ref{sec:variation} discusses the handling of unconditional moment constraints in OU-DRO.

\section{Statistical Calibration of Constraints}\label{sec:calibration}
% We discuss how to calibrate the parameters in the constraints of \eqref{DRO_problem_formulation1} using data.
To illustrate how to calibrate the constraints of formulation \eqref{DRO_problem_formulation1} with data, we consider the following example used in Section \ref{sec:data driven num exp}:
\begin{align}
\max\text{ }  &  P((X,Y)\in S)\nonumber\\
\text{subject to }  &  l_{\bar{F}}\leq\bar{F}(x_{0},y_{0})\leq u_{\bar{F}}\nonumber\\
&  f_{X}(x_{0})\leq u_{X}\nonumber\\
&  f_{Y}(y_{0})\leq u_{Y}\label{data-driven_DRO_formulation}\\
&  a_{X,i}\leq P(x_{0}\leq X\leq x_{i},Y\geq y_{0}|X\geq x_{0},Y\geq y_{0})\leq b_{X,i},i=1,\ldots,n_{X},\nonumber\\
&  a_{Y,j}\leq P(X\geq x_{0},y_{0}\leq Y\leq y_{j}|X\geq x_{0},Y\geq y_{0})\leq b_{Y,j},j=1,\ldots,n_{Y},\nonumber\\
&  f(x^{\prime},y^{\prime})\geq f(x,y)\text{ if }x_{0}\leq x^{\prime}\leq x\text{ and }y_{0}\leq y^{\prime}\leq y\nonumber
\end{align}
Here we do not include the lower bound density constraints since we see in Section \ref{sec:general reduction} that they are redundant. Suppose we have $m$ samples $(X_i,Y_i),i=1,\ldots,m$ and the confidence level is set to be $\alpha$. Suppose also that $x_{0}$ and $y_{0}$ are set as some coordinate-wise top percentiles of the data (e.g., top 10 percentiles) in order to locate the tail region. Then the constants in problem (\ref{data-driven_DRO_formulation}) can be calibrated as
\begin{align*}
& l_{\bar{F}}=\hat{\bar{F}}(x_{0},y_{0})-\frac{\Phi^{-1}(1-\alpha/14)}{\sqrt{m}}\hat{\sigma}_{\bar{F}(x_{0},y_{0})},\quad u_{\bar{F}}=\hat{\bar{F}}(x_{0},y_{0})+\frac{\Phi^{-1}(1-\alpha/14)}{\sqrt{m}}\hat{\sigma}_{\bar{F}(x_{0},y_{0})},\\
& u_{X} =\hat{f}_{1-\alpha/7,X|Y\geq y_{0}}(x_{0})\left(  \hat{P}(Y\geq y_{0})+\frac{\Phi^{-1}(1-\alpha/7)}{\sqrt{m}}\hat{\sigma}_{P(Y\geq y_{0})}\right)  ,\\
& u_{Y} =\hat{f}_{1-\alpha/7,Y|X\geq x_{0}}(y_{0})\left(  \hat{P}(X\geq x_{0})+\frac{\Phi^{-1}(1-\alpha/7)}{\sqrt{m}}\hat{\sigma}_{P(X\geq x_{0})}\right)  ,\\
& a_{X,i} =\hat{F}_{X}(x_{i})-\frac{q}{\sqrt{m_{1}}},\quad b_{X,i}=\hat{F}_{X}(x_{i}-)+\frac{q}{\sqrt{m_{1}}},\\
& a_{Y,j} =\hat{F}_{Y}(y_{j})-\frac{q}{\sqrt{m_{1}}},\quad b_{Y,j}=\hat{F}_{Y}(y_{j}-)+\frac{q}{\sqrt{m_{1}}}.
\end{align*}
We will explain each notation and provide justifications. First, we have
\[
\hat{\bar{F}}(x_{0},y_{0})=\frac{1}{m}\sum_{i=1}^{m}I(X_{i}\geq x_{0},Y_{i}\geq y_{0}),\quad\hat{\sigma}_{\bar{F}(x_{0},y_{0})}=\sqrt{\frac{1}{m-1}\sum_{i=1}^{m}(I(X_{i}\geq x_{0},Y_{i}\geq y_{0})-\hat{\bar{F}}(x_{0},y_{0}))^{2}}
\]
as the point estimate of $\bar{F}(x_{0},y_{0})$ given by the sample mean of $I(X_i\geq x_0,Y_i\geq y_0)$'s, and the sample standard deviation of $I(X_i\geq x_0,Y_i\geq y_0)$'s respectively. We also denote $\Phi$ as the cumulative distribution function of the standard normal distribution. Thus, the constants $l_{\bar{F}}$ and $u_{\bar{F}}$ are set as the asymptotically exact lower and upper confidence bounds of $\bar F(x_0,y_0)$ induced by the standard central limit theorem (CLT) at the confidence level $1-\alpha/7$. Here, the confidence level $1-\alpha/7$ is due to a Bonferroni correction which we will explain later.

Next, the constant $u_{X}$ is calibrated as the $(1-\alpha/7)$-level upper confidence bound on $f_X(x_0)$ by the bootstrap as follows. Note that the truncated marginal density $f_{X}(x_{0})$ can be written as $f_{X}(x_{0})=P(Y\geq y_{0})f_{X|Y\geq y_{0}}(x_{0})$, where $f_{X|Y\geq y_{0}}$ is the density of $X$ conditional on $\{Y\geq y_{0}\}$. We need to find the asymptotically exact ($1-\alpha/7$)-level upper bound for both factors. For $f_{X|Y\geq y_{0}}(x_{0})$, we select the data $(X_{i},Y_{i})$ with $Y_{i}\geq y_{0}$. Suppose the number of such data is $m_{0}$. We resample the $x$-coordinates of the selected data with replacement to the size $m_{0}$, and use the resample to construct a kernel density estimate for $X$ at $x_{0}$. Repeat the procedure many times and the $(1-\alpha/7)$-th quantile $\hat{f}_{1-\alpha/7,X|Y\geq y_{0}}(x_0)$ of these estimates is an asymptotically exact ($1-\alpha/7$)-level upper bound of $f_{X|Y\geq y_{0}}(x_0)$ (\cite{chen2017tutorial}). On the other hand, an asymptotically exact ($1-\alpha/7$)-level upper confidence bound of $P(Y\geq y_{0})$ is obtained by the CLT like for $\bar F(x_0,y_0)$ discussed earlier, where
\[
\hat{P}(Y\geq y_{0})=\frac{1}{m}\sum_{i=1}^{m}I(Y_{i}\geq y_{0}),\quad\hat{\sigma}_{P(Y\geq y_{0})}=\sqrt{\frac{1}{m-1}\sum_{i=1}^{m}(I(Y_{i}\geq y_{0})-\hat{P}(Y\geq y_{0}))^{2}}.
\]
The calibration of $u_{Y}$ is justified similarly.

Furthermore, to calibrate the constants $a_{X,i},b_{X,i}$, we use the Kolmogorov–Smirnov statistic. We collect all data such that $X\geq x_{0}$ and $Y\geq y_{0}$. Suppose these data have size $m_{1}$. The Kolmogorov–Smirnov test gives us
\begin{equation}
\sqrt{m_1}\sup_{x\ge x_{0}}|F_{X}(x)-\hat{F}_{X}(x)|\leq q\label{KS_test}
\end{equation}
an asymptotically exact $(1-\alpha/7)$-level confidence region on $F_X(x)$, where $F_{X}(x)=P(X\leq x|X\geq x_{0},Y\geq y_{0})$ and $\hat{F}_{X}$ is the empirical distribution of $X$ restricted to $X\geq x_{0},Y\geq y_{0}$, and $q$ is the $(1-\alpha/7)$-th quantile of $\sup_{t\in\lbrack0,1]}|BB(t)|$ where $BB$ is the standard Brownian bridge. By considering the left and right limits at $x=x_i$ in (\ref{KS_test}), we can rewrite \eqref{KS_test} to obtain the lower and upper bounds of $F_{X}(x_i)$ as $a_{x,i}$ and $b_{X,i}$ defined earlier, where $\hat{F}_{X}(x_{i}-)$ is the left limit of $\hat{F}_{X}$ at $x_{i}$. The constants $a_{Y,j},b_{Y,j}$ can be calibrated similarly.

Finally, the Bonferroni correction is applied since there are in total 7 sets of parameters that we need to calibrate, and we want to control the familywise Type I error to be at most $\alpha$, thus giving each individual confidence level discussed above to be $1-\alpha/7$. If the constraints are calibrated via the methods above and the true density is OU about $(x_0,y_0)$ on $[x_0,\infty)\times[y_0,\infty)$, then the constraints in (\ref{data-driven_DRO_formulation}) will hold simultaneously with confidence level $1-\alpha$ asymptotically. By Corollary \ref{OU-DRO_validity}, the optimal value will then be an asymptotically valid upper confidence bound on the true probability with level $1-\alpha$.

\section{OU Distribution on $\mathbb{R}^d$}\label{sec:general OU}
%In this section, we consider the general setting of the OU distribution. There are mainly two ways to generalize the current Definition \ref{OU mix} of OU. One is to enlarge the support from $\mathcal{D}_0$ to a larger domain such as $\mathbb{R}^d$ and keep the monotonicity property for each variable. The other one is to require the density to possess the monotonicity property only in some of the variables , which we call partially orthounimodal (POU) distribution. We will talk about the first generalization in Section \ref{sec:OU on R^d} and the second generalization in Section \ref{sec:POU}.

%\subsubsection{OU Distribution on $\mathbb{R}^d$}\label{sec:OU on R^d}

In this section, we consider a generalization of the OU distribution defined in Definition \ref{OU mix}. We will enlarge the support from $\mathcal{D}_0$ to $\mathbb{R}^d$ and keep the monotonicity property for each variable. Like the OU distribution on $\mathcal{D}_0$, we will provide two definitions for the OU distribution on $\mathbb{R}^d$. One is defined in terms of the density and the other is defined via mixture representation.
\begin{definition}[Orthounimodal density on $\mathbb{R}^d$]
A probability distribution on $\mathbb{R}^d$ with density $f$ (with respect to the Lebesgue measure) is OU about mode $x_0=(x_{10},\ldots,x_{d0})$ if for each $i=1,\ldots,d$ and any fixed $x_j\in\mathbb{R},j\neq i$, the function $x_i\mapsto f(x)$ is non-decreasing on $(-\infty,x_{i0}]$ and non-increasing on $[x_{i0},\infty)$.
\label{general OU def}
\end{definition}
\begin{definition}[Mixture representation of OU distribution on $\mathbb{R}^d$]
We have:
\begin{description}
\item[(1)] A set $K\subset\mathbb{R}^d$ is said to be orthounimodal about $x_{0}$ if for every $x\in K$, the closed rectangle with edges parallel to the axes and with opposite vertices $x_0$ and $x$ is in $K$, i.e., $(\eta_1 x_1+(1-\eta_1)x_{10},\ldots,\eta_d x_d+(1-\eta_d)x_{d0})\in K$ for any $\eta_i\in[0,1],i=1,\ldots,d$.

\item[(2)] A distribution on $\mathbb{R}^d$ is called OU about $x_{0}$ if it belongs to the closed convex hull of the set of all uniform distributions on subsets of $\mathbb{R}^d$ which are OU about $x_{0}$.
\end{description}
\label{general OU mix}
\end{definition}
From Definition \ref{general OU def} and Definition \ref{general OU mix}, we can see OU distribution has a linear transformation property: if $X$ is OU about $x_0$, then $X+x_0^{\prime}$ is OU about $x_0+x_0^{\prime}$. Besides, analogies of Lemma \ref{property_of_OU_set}, Theorem \ref{Choquet_OU} and Lemma \ref{extreme_point_OU} hold for OU distributions on $\mathbb{R}^d$.
\begin{lemma}%[Properties of OU sets on $\mathbb{R}^d$]
\label{property_of_general_OU_set}Suppose that $K\subset\mathbb{R}^d$ is an OU set about $x_{0}$. Then $K$ is Lebesgue measurable. Besides, $\bar{K}$ and $\bar{K^{\circ}}$ (closure of $K^{\circ}$) are also OU sets about $x_{0}$ satisfying $\lambda(K^{\circ})=\lambda(K)=\lambda(\bar{K})$.
\end{lemma}

The following theorem justifies the equivalence of Definition \ref{general OU mix} and Definition \ref{general OU def} and establishes the Choquet representation for OU distributions on $\mathbb{R}^d$ in the presence of density. To simplify the notations, we assume without loss of generality that the mode is the origin. If $X$ is OU about $x_0$, we can see $X-x_0$ is OU about the origin and Theorem \ref{Choquet_general_OU} applies. Recall that $W_K$ denotes the uniform distribution on $K$.

\begin{theorem}
\label{Choquet_general_OU}Suppose a distribution $P$ on $\mathbb{R}^d$ is absolutely continuous with respect to Lebesgue measure. Then $P$ is OU about the origin if and only if there is a density $f(x)$ of $P$ such that for every $s>0$, the set%
\[
C_{s}=\{x\in\mathcal{D}_{0}:f(x)\geq s\}
\]
is OU about the origin, or equivalently, if and only if $f$ is an OU density about the origin on $\mathbb{R}^d$. Besides, let $O_1,\ldots,O_{2^d}$ be the $2^d$ (closed) orthants of $\mathbb{R}^d$. Then $P$ has the following Choquet representation:%
\begin{equation}
P(B)=\sum_{1\leq i\leq2^{d},P(O_{i})>0}P(O_{i})\int_{0}^{\infty}W_{\bar{C}_{s}^{i}}(B)g_{i}(s)ds \label{equ_Choquet_general_OU}%
\end{equation}
for any Lebesgue measurable set $B$, where $C_{s}^{i}=\{x\in O_i:f_i(x)\geq s\}$, $f_i(x)=f(x)I(x\in O_i)/P(O_i)$ is the density of the conditional distribution $P(\cdot|O_i)$, $\bar{C}_{s}^{i}$ is the closure of $C_{s}^{i}$ and $g_i(s)=\lambda(\bar{C}_{s}^{i})$ is a probability density on $(0,\infty)$.
\end{theorem}

The following Lemma tells us about the extreme point in the class of OU distributions about the origin on $\mathbb{R}^d$.

\begin{lemma}
If $K\subset O_i$ is an OU set about the origin with $\lambda(K)>0$ for some orthant $O_i$, then $W_{K}$ is an extreme point in the class of OU distributions about the origin on $\mathbb{R}^d$.
\label{extreme_point_general_OU}
\end{lemma}

The condition $K\subset O_i$ is essential for Lemma \ref{extreme_point_general_OU}. Consider an example in the bivariate case. Define $K=\{(x,y):-1\le x\le1,-1\le y\le1\}$, $K_1=\{(x,y):0\le x\le1,0\le y\le1\}$, $K_2=\{(x,y):-1\le x\le0,0\le y\le1\}$, $K_3=\{(x,y):-1\le x\le0,-1\le y\le0\}$ and $K_4=\{(x,y):0\le x\le1,-1\le y\le0\}$. We can see all the sets are OU about the origin. However, $W_K$ is not an extreme point in the class of OU distributions about the origin on $\mathbb{R}^2$ since it can be written as
\[
W_K=\frac{1}{4}W_{K_1}+\frac{1}{4}W_{K_2}+\frac{1}{4}W_{K_3}+\frac{1}{4}W_{K_4}.
\]
This explains why the Choquet representation in Theorem \ref{Choquet_general_OU} is represented as (\ref{equ_Choquet_general_OU}) instead of
\begin{equation}
P(B)=\int_{0}^{\infty}W_{\bar{C}_{s}}(B)g(s)ds \label{wrong_Choquet_general_OU}
\end{equation}
although (\ref{wrong_Choquet_general_OU}) still holds in the setting of Theorem \ref{Choquet_general_OU}.

\section{Theory of Partial Orthounimodality and Partially Orthounimodal Distributionally Robust Optimization}\label{sec:POU appendix}
In this section, we provide the details on the theory of POU distributions and the POU-DRO problem discussed in Section \ref{sec:POU-DRO problem}.

\subsection{POU Distribution}\label{sec:POU}

In this section, we introduce POU and establish its Choquet representation by following the same flow as in the OU case. As we have discussed, POU models a density that is non-increasing in only some of the dimensions. For simplicity, we suppose the density is non-increasing in only the first $d^{\prime}$ components. In other words, only the first $d^{\prime}$ components of the random vector are in their tails. For other cases, we can simply permute the components to transform them into this case. Let $\mathcal{D}_0=\{x\ge x_0\}$ where $x_0=(x_{10},\ldots,x_{d0})\in \mathbb{R}^{d^{\prime}}\times[-\infty,\infty)^{d-d^{\prime}}$ for some $d^{\prime}<d$. When $x_{j0}=-\infty$, $x_j\ge x_{j0}$ is interpreted as $x_j\in\mathbb{R}$. Similar to OU distributions, POU distributions can be defined in terms of monotonicity of the probability density or mixture representation.

\begin{definition}[$d^{\prime}$-partially orthounimodal density]
A probability distribution on $\mathcal{D}_{0}$ with density $f$ (with respect to the Lebesgue measure) is $d^{\prime}$-partially orthounimodal ($d^{\prime}$-POU) about mode $x_0$ if $f$ is non-increasing with respect to $x_1,\ldots,x_{d^{\prime}}$ on $\mathcal{D}_{0}$ for any fixed $(x_{d^{\prime}+1},\ldots,x_d)$ with $(x_{d^{\prime}+1},\ldots,x_d)\ge(x_{d^{\prime}+1,0},\ldots,x_{d0})$.
\label{POU def}
\end{definition}

\begin{definition}[Mixture representation of $d^{\prime}$-POU distribution]
We have:
\begin{description}
\item[(1)] A measurable set $K\subset\mathcal{D}_{0}$ is said to be $d^{\prime}$-POU about $x_{0}$ if for every $x\in K$, we have $x^{\prime}\in K$ if $x_i\ge x_i^{\prime}\ge x_{i0}$ for $i=1,\ldots,d^{\prime}$ and $x_i=x_i^{\prime}\ge x_{i0}$ for $i=d^{\prime}+1,\ldots,d$.

\item[(2)] A distribution on $\mathcal{D}_{0}$ is called $d^{\prime}$-POU about $x_{0}$ if it belongs to the closed convex hull of the set of all uniform distributions on subsets of $\mathcal{D}_{0}$ which are $d^{\prime}$-POU about $x_{0}$.
\end{description}
\label{POU mix}
\end{definition}

The following theorem justifies the equivalence of Definitions \ref{POU def} and \ref{POU mix} in the presence of a density.

\begin{theorem}
\label{equivalence_POU}Suppose a distribution $P$ on $\mathcal{D}_{0}$ is absolutely continuous with respect to Lebesgue measure. Then $P$ is $d^{\prime}$-POU about $x_{0}$ if and only if there is a density $f(x)$ of $P$ such that for every $s>0$, the set%
\[
C_{s}=\{x\in\mathcal{D}_{0}:f(x)\geq s\}
\]
is $d^{\prime}$-POU about $x_{0}$, or equivalently, if and only if $f$ is a $d^{\prime}$-POU density about $x_{0}$ on $\mathcal{D}_{0}$.
\end{theorem}

The Choquet representation of $d^{\prime}$-POU distributions is slightly more complicated than that of the OU distributions. We first introduce the notations. For a $d^{\prime}$-POU set $K$ about $x_0$ on $\mathcal{D}_0$, we define its slice with respect to the subspace $\{y\in\mathbb{R}^d:y_{d^{\prime}+1}=x_{d^{\prime}+1},\ldots,y_d=x_d\}$ as $K_{x_{d^{\prime}+1},\ldots,x_d}=\{(y_1,\ldots,y_{d^{\prime}}):(y_1,\ldots,y_{d^{\prime}},x_{d^{\prime}+1},\ldots,x_d)\in K\}\subset \mathbb{R}^{d^{\prime}}$. In other words, $K\cap\{y\in\mathbb{R}^d:y_{d^{\prime}+1}=x_{d^{\prime}+1},\ldots,y_d=x_d\}=K_{x_{d^{\prime}+1},\ldots,x_d}\times\{(x_{d^{\prime}+1},\ldots,x_d)\}$. Since $K$ is $d^{\prime}$-POU, $K_{x_{d^{\prime}+1},\ldots,x_d}$ is either empty or a nonempty OU set about $(x_{10},\ldots,x_{d^{\prime}0})$ on $\{y\in\mathbb{R}^{d^{\prime}}:y_i\ge x_{i0},i=1,\ldots,d^{\prime}\}$. Recall that $W_K$ denotes the uniform distribution on $K$ and $K$ could be in a lower-dimensional subspace of $\mathbb{R}^d$. In the presence of a density, we have the following Choquet representation theorem:
\begin{theorem}\label{Choquet_POU}
Suppose a distribution $P$ on $\mathcal{D}_0$ is absolutely continuous with respect to Lebesgue measure and $P$ is $d^{\prime}$-POU about $x_0$. Let $f$ be a $d^{\prime}$-POU density of $P$ and $C_{s}=\{x\in\mathcal{D}_{0}:f(x)\geq s\}$. Then the Choquet representation of $P$ is given by:
\begin{equation*}
P(B)=\int_{0}^{\infty}\int_{x_{d^{\prime}+1,0}}^{\infty}\cdots\int_{x_{d0}}^{\infty}W_{C_{s}\cap\{y\in \mathbb{R}^{d}:y_{d^{\prime}+1}=x_{d^{\prime}+1},\ldots,y_{d}=x_{d}\}}(B)g(s,x_{d^{\prime}+1},\ldots,x_{d})dx_{d^{\prime}+1}\cdots dx_{d}ds
\end{equation*}
for any Lebesgue measurable set $B$, where $g(s,x_{d^{\prime}+1},\ldots,x_{d})=\lambda_{d^{\prime}}(C_{s,x_{d^{\prime}+1},\ldots,x_{d}})$ is a probability density on $\{(s,x_{d^{\prime}+1},\ldots,x_{d}):s>0,x_i\ge x_{i0},i=d^{\prime}+1,\ldots,d\}$, $C_{s,x_{d^{\prime}+1},\ldots,x_{d}}$ is the slice of $C_s$ with respect to the subspace $\{y\in\mathbb{R}^d:y_{d^{\prime}+1}=x_{d^{\prime}+1},\ldots,y_d=x_d$\} and $\lambda_{d^{\prime}}(\cdot)$ is the Lebesgue measure on $\mathbb{R}^{d^{\prime}}$.
\end{theorem}

In the proof of Theorem \ref{Choquet_POU}, we need the following lemma to ensure the representation of $P$ is indeed a Choquet representation.

\begin{lemma}
\label{extreme_point_POU}Suppose $K\subset\mathcal{D}_{0}$ is a $d^{\prime}$-POU set about $x_{0}$. If $K$ is fully contained in the subspace $\{y\in\mathbb{R}^{d}:y_{d^{\prime}+1}=x_{d^{\prime}+1},\ldots,y_{d}=x_{d}\}$ for $x_{i}\geq x_{i0},i=d^{\prime}+1,\ldots,d$ with $\lambda_{d^{\prime}}(K_{x_{d^{\prime}+1},\ldots,x_{d}})>0$, then $W_{K}$ is an extreme point in the class of $d^{\prime}$-POU distributions about $x_{0}$ on $\mathcal{D}_{0}$.
\end{lemma}

When we consider $d^{\prime}=1$, i.e., 1-POU distributions, we can see the slice $K_{x_{2},\ldots,x_d}$ is either $[x_{10},x_1)$ or $[x_{10},x_1]$ for some $x_1\ge x_{10}$. Such a simple form of $K_{x_{2},\ldots,x_d}$ gives us a cleaner Choquet representation even without the presence of a density.

\begin{theorem}\label{Choquet_1POU}
Suppose a random vector $X=(X_1,\ldots,X_d)$ takes values on $\mathcal{D}_{0}$. Then the probability measure $P$ of $X$ is 1-POU about $x_0$ if and only if $X$ is distributed as $(x_{10}+U(Z_1-x_{10}),Z_2,\ldots,Z_d)$ where $U$ is the uniform distribution on $(0,1)$ and independent of $(Z_1,\ldots,Z_d)\in\mathcal{D}_{0}$. Consequently, $P$ is 1-POU about $x_0$ if and only if it has the following Choquet representation
\begin{equation}
P=\int_{\mathcal{D}_{0}}W_{\text{1-POU}}(z)dQ(z),\label{equation_Choquet_1POU}
\end{equation}
where $W_{\text{1-POU}}(z)$ is the distribution of $(x_{10}+U(z_{1}-x_{10}),z_{2},\ldots,z_{d})$, $U$ is the uniform distribution on $(0,1)$ and $Q$ is the distribution of $(Z_1,\ldots,Z_d)$ uniquely determined by $P$. Moreover, if $P(X_1=x_{10})=0$, then $X_1$ is an absolutely continuous random variable with the probability density $E_Q[I(Z_1\ge x)/(Z_1-x_{10})],x\ge x_{10}$ which is continuous at $x=x_{10}$.
\end{theorem}

\subsection{POU-DRO Problem: Formulation and Reduction}

In this section, we discuss the formulation of the $d^{\prime}$-POU-DRO problem and the reduction of the 1-POU-DRO problem. With these, for the bivariate case, theories in this paper would cover all the extreme cases: if two dimensions are both in the tail, we can handle it by the OU-DRO problem; if only one dimension is in the tail, we can handle it by the 1-POU-DRO problem. Since the general $d^{\prime}$-POU-DRO problem involves the OU-DRO problem (when $d^{\prime}=d$) and the latter hasn't been solved completely for any dimension, we will leave the reduction of the general $d^{\prime}$-POU-DRO problem as the future work.

Recall that for $d^{\prime}$-POU, only the first $d^{\prime}$ dimensions are in the tail and the tail region is given by $\mathcal{D}_0=\{x\ge x_0\}$ with $x_0=(x_{10},\ldots,x_{d0})\in \mathbb{R}^{d^{\prime}}\times[-\infty,\infty)^{d-d^{\prime}}$. Imitating the OU-DRO problem, we formulate the $d^{\prime}$-POU as follows:
\begin{align}
\max\text{ }  &  P((X_{1},\ldots,X_{d})\in S)\nonumber\\
\text{subject to }  &  l_{\bar{F}}\leq\bar{F}(x_{0})\leq u_{\bar{F}}\nonumber\\
&f_{X_{i}}(x_{i0})\leq u_{X_{i}},i=1,\ldots,d^{\prime}\label{general_POU_DRO}\\
&  a_{i}\bar{F}(x_{0})\leq P(\underline{x}_{ji}\leq X_{j}\leq\bar{x}_{ji},j=1,\ldots,d)\leq b_{i}\bar{F}(x_{0}),i=1,\ldots,n\nonumber\\
&  f(x^{\prime})\ge f(x) \text{ if } x_i\ge x_i^{\prime}\ge x_{i0}, i\le d^{\prime} \text{ and } x_i=x_i^{\prime}\ge x_{i0}, i>d^{\prime}\nonumber
\end{align}
The interpretation and calibration of the constraints in (\ref{general_POU_DRO}) are the same as the OU-DRO problem (\ref{DRO_problem_formulation1}). However, we do not restrict the truncated marginal densities for the last $d-d^{\prime}$ components because these components are not in the tail in the motivating problems.
% and therefore there is no point to restrict their truncated marginal densities.

Now we focus on the reduction of the 1-POU-DRO problem. Here we make a small modification that we will use Definition \ref{POU mix} instead of Definition \ref{POU def} as the 1-POU shape constraint. In other words, we will not assume the existence of the density of $X$. As we will see later in Theorem \ref{reduction_1POU-DRO} and Corollary \ref{cor_reduction_1POU-DRO}, this is because we may not be able to recover an absolutely continuous $X$ for the 1-POU-DRO problem from the optimal solution to the reduced moment problem. As in the OU-DRO problem, we will replace $l_{\bar{F}}\leq\bar{F}(x_{0})\leq u_{\bar{F}}$ by $\bar{F}(x_{0})=c$ in the reduction procedure. For reference, the 1-POU-DRO problem is formulated as
\begin{align}
\max\text{ }  &  P((X_{1},\ldots,X_{d})\in S)\nonumber\\
%\text{subject to }  &  l_{\bar{F}}\leq\bar{F}(x_{0})\leq u_{\bar{F}}\nonumber\\
\text{subject to }  &  \bar{F}(x_{0})=c \nonumber\\
&f_{X_{1}}(x_{10})\leq u_{X_{1}}\label{1POU_DRO}\\
&  a_{i}\bar{F}(x_{0})\leq P(\underline{x}_{ji}\leq X_{j}\leq\bar{x}_{ji},j=1,\ldots,d)\leq b_{i}\bar{F}(x_{0}),i=1,\ldots,n\nonumber\\
%&  f(x^{\prime})\ge f(x) \text{ if } x_1\ge x_1^{\prime}\ge x_{10},\text{ and } x_i=x_i^{\prime}\ge x_{i0}, i\ge2\nonumber
& (X_{1},\ldots,X_{d})\text{ is 1-POU about }x_{0}\text{ when restricted to }\mathcal{D}_{0}\text{ and }P(X_{1}=x_{10})=0\nonumber
\end{align}
Here $f_{X_{1}}(x_1)$ denotes the marginal density of $X_1$ within $\mathcal{D}_{0}$, i.e., for any measurable $A\subset[x_{10},\infty)$, $\int_{A}f_{X_{1}}(x_{1})dx_{1}=P(X_{1}\in A,X_{i}\geq x_{i0},i\geq2)$. The existence of $f_{X_{1}}(x_1)$ is ensured by Theorem \ref{Choquet_1POU}. We require $f_{X_{1}}(x_1)$ to be continuous at $x_1=x_{10}$ in order to avoid the arbitrariness of the density at this point. When $X$ has a continuous density $f$ satisfying Definition \ref{POU def} on $\mathcal{D}_{0}$, $f_{X_{1}}$ is just given by
\[
f_{X_{1}}(x_{1})=\int_{x_{20}}^{\infty}\cdots\int_{x_{d0}}^{\infty}f(x_{1},\ldots,x_{d})dx_{2}\cdots dx_{d}
\]
as in the formulation of OU-DRO problem. For the rare-event set $S$, we assume it is in the form
\begin{equation}
S=\{x\in\mathcal{D}_{0}:g_{1}(x_{2},\ldots,x_{d})\leq x_{1}\leq g_{2}(x_{2},\ldots,x_{d})\}\label{form_S_1POU}%
\end{equation}
where $g_{i}:[x_{20},\infty)\times\cdots\times\lbrack x_{d0},\infty)\mapsto\lbrack x_{10},\infty]$ is a measurable function for $i=1,2$ and $g_{1}(x_{2},\ldots,x_{d})\leq g_{2}(x_{2},\ldots,x_{d})$ for any $(x_{2},\ldots,x_{d})$. We can see the formulation (\ref{1POU_DRO}) only depends on the behavior in $\mathcal{D}_{0}$ so we can restrict our attention to the truncated distribution on $\mathcal{D}_{0}$. The idea for solving (\ref{1POU_DRO}) is to use the Choquet representation (\ref{equation_Choquet_1POU}) to rewrite it as a moment problem with respect to the probability measure $Q$.

\begin{theorem}\label{reduction_1POU-DRO}
The 1-POU-DRO problem (\ref{1POU_DRO}) is equivalent to the following moment problem:%
\begin{align}
\max\text{ } &  cE_{Q}\left[  \frac{\min(g_{2}(Z_{2},\ldots,Z_{d}),Z_{1})-\min(g_{1}(Z_{2},\ldots,Z_{d}),Z_{1})}{Z_{1}-x_{10}}\right]  \nonumber\\
\text{subject to } &  E_{Q}\left[  \frac{1}{Z_{1}-x_{10}}\right]  \leq\frac{u_{X_{1}}}{c}\label{moment_problem_1POU-DRO}\\
&  a_{i}\leq E_{Q}\left[  \frac{\min(Z_{1},\bar{x}_{1i})-\min(Z_{1},\underline{x}_{1i})}{Z_{1}-x_{10}}I(\underline{x}_{ji}\leq Z_{j}\leq\bar{x}_{ji},j\geq2)\right]  \leq b_{i},i=1,\ldots,n\nonumber
\end{align}
where the decision variable $Q$ is the probability distribution of $(Z_1,\ldots,Z_d)\in\mathcal{D}_{0}$. If $(Z_{1}^{\ast},\ldots,Z_{d}^{\ast})\sim Q^{\ast}$ is the optimal solution to the moment problem (\ref{moment_problem_1POU-DRO}), then $cP^{\ast}$ is the optimal solution to the 1-POU-DRO problem (\ref{1POU_DRO}) where $P^{\ast}$ is the distribution of $(x_{10}+U(Z_{1}^{\ast}-x_{10}),Z_{2}^{\ast},\ldots,Z_{d}^{\ast})$ with $U$ being the uniform distribution on $(0,1)$ and independent of $(Z_{1}^{\ast},\ldots,Z_{d}^{\ast})$.
\end{theorem}

According to Theorem 3.2 in \cite{winkler1988extreme}, to solve the moment problem (\ref{moment_problem_1POU-DRO}), it suffices to consider discrete probability measures with at most $n+2$ points in the support. So (\ref{moment_problem_1POU-DRO}) is equivalent to the following non-linear optimization:
\begin{align}
\max\text{ } & c\sum_{l=1}^{n+2}p_{l}\frac{\min(g_{2}(z_{l2},\ldots,z_{ld}),z_{l1})-\min(g_{1}(z_{l2},\ldots,z_{ld}),z_{l1})}{z_{l1}-x_{10}}\nonumber \\
\text{subject to } & \sum_{l=1}^{n+2}\frac{p_{l}}{z_{l1}-x_{10}} \leq\frac{u_{X_{1}}}{c} \nonumber\\
 & a_{i} \leq\sum_{l=1}^{n+2}p_{l}\frac{\min(z_{l1},\bar{x}_{1i})-\min(z_{l1},\underline{x}_{1i})}{z_{l1}-x_{10}}I(\underline{x}_{ji}\leq z_{lj}\leq\bar{x}_{ji},j\geq2)\leq b_{i},i=1,\ldots,n \label{nonlinear_problem_1POU}\\
 & \sum_{l=1}^{n+2}p_{l} =1,p_{l}\geq0,l=1,\ldots,n+2 \nonumber\\
 & z_{li} \geq x_{i0},,l=1,\ldots,n+2,i=1,\ldots,d \nonumber
\end{align}
where $p_l,z_{li}$ are the decision variables. The non-linear optimization (\ref{nonlinear_problem_1POU}) is capable to handle the 1-POU-DRO problem (\ref{1POU_DRO}) when the constraint $\bar{F}(x_{0})=c$ is replaced by $l_{\bar{F}}\leq\bar{F}(x_{0})\leq u_{\bar{F}}$. In this case, we simply make $c$ an additional decision variable varying in $[l_{\bar{F}},u_{\bar{F}}]$. This is stated in the following corollary.

\begin{corollary}\label{cor_reduction_1POU-DRO}
We have the following:
\begin{description}
\item[(1)] The 1-POU-DRO problem (\ref{1POU_DRO}), the moment problem (\ref{moment_problem_1POU-DRO}) and the non-linear optimization problem (\ref{nonlinear_problem_1POU}) are equivalent.

\item[(2)] Consider the 1-POU-DRO problem (\ref{1POU_DRO}) with $\bar{F}(x_{0})=c$ replaced by $l_{\bar{F}}\leq\bar{F}(x_{0})\leq u_{\bar{F}}$. Then this 1-POU-DRO problem is equivalent to the non-linear optimization problem (\ref{nonlinear_problem_1POU}) with an additional decision variable $c$ and an additional constraint $l_{\bar{F}}\leq c\leq u_{\bar{F}}$. If $c^{\ast},p_{l}^{\ast},z_{li}^{\ast}$ is the optimal solution to this non-linear optimization problem and we write $(Z_{1}^{\ast},\ldots,Z_{d}^{\ast})\sim Q^{\ast}$ as the discrete distribution on $(z_{l1}^{\ast},\ldots,z_{ld}^{\ast}),l=1,\ldots,n+2$ with corresponding probability $p_{l}^{\ast}$, then $c^{\ast}P^{\ast}$ is the optimal solution to this 1-POU-DRO problem where $P^{\ast}$ is the distribution of $(x_{10}+U(Z_{1}^{\ast}-x_{10}),Z_{2}^{\ast},\ldots,Z_{d}^{\ast})$ with $U$ being the uniform distribution on $(0,1)$ and independent of $(Z_{1}^{\ast},\ldots,Z_{d}^{\ast})$.
\end{description}
\end{corollary}

Corollary \ref{cor_reduction_1POU-DRO} reveals why we do not assume the existence of the density in the formulation of the 1-POU-DRO problem (\ref{1POU_DRO}). Suppose $p_{l}^{\ast},z_{li}^{\ast}$ is the optimal solution to the non-linear optimization problem (\ref{nonlinear_problem_1POU}), i.e., the optimal solution to the moment problem (\ref{moment_problem_1POU-DRO}) can be taken as the discrete distribution $(Z_{1}^{\ast},\ldots,Z_{d}^{\ast})\sim Q^{\ast}$ on $(z_{l1}^{\ast},\ldots,z_{ld}^{\ast}),l=1,\ldots,n+2$ with corresponding probability $p_{l}^{\ast}$. Then if we recover the optimal solution $cP^*$ in the way described in Theorem \ref{reduction_1POU-DRO}, we will find $cP^*$ is not absolutely continuous on $\mathcal{D}_0$ as its support has Lebesgue measure 0. Due to this reason, we do not assume the existence of the density in (\ref{1POU_DRO}).

We close this section by discussing a variant of formulation (\ref{1POU_DRO}). As in the discussion of OU-DRO problem, we can replace the moment constraint in (\ref{1POU_DRO}) by the unconditional version:
\begin{equation}
a_{i}\leq P(\underline{x}_{ji}\leq X_{j}\leq\bar{x}_{ji},j=1,\ldots,d)\leq b_{i}.\label{uncon_moment_constraint_1POU}
\end{equation}
In this case, we just need to replace $a_i$ and $b_i$ in the problem (\ref{moment_problem_1POU-DRO}) or (\ref{nonlinear_problem_1POU}) by $a_i/c$ or $b_i/c$ and Theorem \ref{reduction_1POU-DRO} and Corollary \ref{cor_reduction_1POU-DRO} still hold since (\ref{uncon_moment_constraint_1POU}) is equivalent to the conditional version constraint:
\begin{equation*}
\frac{a_{i}}{c}\bar{F}(x_{0})\leq P(\underline{x}_{ji}\leq X_{j}\leq\bar{x}_{ji},j=1,\ldots,d)\leq \frac{b_{i}}{c}\bar{F}(x_{0})
\end{equation*}
given $\bar{F}(x_0)=c$.

\section{Proofs}\label{sec:proofs}

\begin{proof}[Proof of Lemma \ref{lemma:basic}.]
The conclusion follows since
$$\left\{\begin{array}{c}
\text{geometric property of the distribution $F$ holds for $X\ge u$},\\
\text{auxiliary constraints on $F$ holds}
\end{array}\right\}$$
implies $\{Z^*\geq Z\}$.
\end{proof}

\begin{proof}[Proof of Lemma \ref{property_of_OU_set}.]
Let us first show $\bar{K}$ and $\bar{K^{\circ}}$ are OU sets about $x_{0}$. In order to show $\bar{K}$ is OU, it suffices to show that if $x$ is a limit point of $K$, then $x^{\prime}\in\bar{K}$ if $x_{0}\leq x^{\prime}\leq x$. Suppose that the sequence $\{x_j,j\geq1\}\subset K$ converges to $x$. Consider the sequence $\{\min(x^{\prime},x_{j}),j\geq1\}$ where the minimum is taken component-wise. By the definition of OU sets, this sequence must be in $K$. Besides, we notice that the limit point of this sequence is $\min(x^{\prime},x)=x^{\prime}$. Thus $x^{\prime}\in\bar{K}$, i.e., $\bar{K}$ is an OU set about $x_{0}$. Now let's show $\bar{K^{\circ}}$ is also OU. Suppose $x\in\bar{K^{\circ}}$, i.e., there exists a sequence $\{x_j\}\subset K^{\circ}$ s.t. $x_j\rightarrow x$. We need to show $x^{\prime}\in \bar{K^{\circ}}$ if $x_0 \le x^{\prime}\le x$. By the definition of OU, $x_j\in K^{\circ}$ implies $x_0<x_j$ and $x^{\prime\prime}\in K^{\circ}$ for $x_0<x^{\prime\prime}\le x_j$. Therefore, if $x_0<x^{\prime}\le x$, we have $\min(x_j,x^{\prime})\in K^{\circ}$. Since $\min(x_j,x^{\prime})\rightarrow x^{\prime}$, we know $x^{\prime}\in \bar{K^{\circ}}$ for all $x^{\prime}$ satisfying $x_0<x^{\prime}\le x$. Moreover, since $\bar{K^{\circ}}$ is a closed set, we can get $x^{\prime}\in \bar{K^{\circ}}$ for $x_0 \le x^{\prime}\le x$, which means $\bar{K^{\circ}}$ is OU about $x_0$.

Then let us consider other claims in the lemma. Since Lebesgue measure is complete, in order to show $K$ is Lebesgue measurable and $\lambda(K^{\circ})=\lambda(K)=\lambda(\bar{K})$, it suffices to show that $\lambda(\partial K)=0$. This is already proven in \cite{lavrivc1993continuity}, which concludes our proof.
\end{proof}

\begin{proof}[Proof of Theorem \ref{Choquet_OU}.]
Suppose there is a density $f(x)$ of $P$ such that for every $s>0$, the set%
\[
C_{s}=\{x\in\mathcal{D}_{0}:f(x)\geq s\}
\]
is OU about $x_{0}$. Write $W_{\bar{C}_{s}}$ as the uniform distribution on $\bar{C}_{s}$ and let $g(s)=\lambda(\bar{C}_{s})\geq 0$. Notice that
\[
f(x)=\int_{0}^{f(x)}1ds=\int_{0}^{\infty}I(x\in C_{s})ds.
\]
Thus for any measurable set $B$,
\[
P(B)=\int_{B}f(x)dx=\int_{B}\int_{0}^{\infty}I(x\in C_{s})dsdx=\int_{0}^{\infty}\lambda(B\cap C_{s})ds.
\]
By Lemma \ref{property_of_OU_set}, we know that $\lambda(\partial C_{s})=0$. Thus, $\lambda(B\cap C_{s})=\lambda(B\cap\bar{C}_{s})$, which means%
\[
P(B)=\int_{0}^{\infty}\lambda(B\cap\bar{C}_{s})ds=\int_{0}^{\infty}W_{\bar{C}_{s}}(B)g(s)ds.
\]
Setting $B=\mathcal{D}_{0}$, we can see $P(B)=1=\int_{0}^{\infty}g(s)ds$ and thus $g(s)$ is a probability density on $(0,\infty)$. This proves the correctness of (\ref{equ_Choquet_OU}). By Lemma \ref{property_of_OU_set}, we know that $\bar{C}_{s}$ is an OU set. So $W_{\bar{C}_{s}}$ is the uniform distribution on the OU set $\bar{C}_{s}$. Since the probability measure with the density $g(s)$ must be the limit of a sequence of discrete probability measures with finite support, we can see $P$ is in the closed convex hull of the set of all uniform distributions on OU subsets of $\mathcal{D}_{0}$, i.e., $P$ is an OU distribution.

To see the converse, we assume $P$ is OU about $x_0$. Suppose that $Q$ is the uniform distribution on an OU set $K\subset\mathcal{D}_{0}$ and $0<\lambda(K)<\infty$. For a point $x\in\mathbb{R}^{d}$ and $\delta>0$, we define the neighborhood $N_{\delta}(x)=\{y\in\mathbb{R}^{d}:\max_{1\leq i\leq d}|x_{i}-y_{i}|<\delta\}$. For ease of illustration, we assume $x_{0}$ is the origin. Otherwise, we can perform a translation to achieve it. Then for $x,x^{\prime}\in\mathcal{D}_{0}^{\circ}$ with $x\ge x^{\prime}$ and $0<\delta<\min_{1\leq i\leq d}x_{i}^{\prime}$, we have%
\[
y\in K\cap N_{\delta}(x)\Rightarrow y+x^{\prime}-x\in K\cap N_{\delta}(x^{\prime}).
\]
Since Lebesgue measure is translation invariant, we can get%
\[
\lambda(K\cap N_{\delta}(x^{\prime}))\geq\lambda(K\cap N_{\delta}(x)).
\]
Dividing by $\lambda(K)$, we have%
\begin{equation}
Q(N_{\delta}(x^{\prime}))\geq Q(N_{\delta}(x)).\label{inequality_of_Q}%
\end{equation}
Clearly, this relation also holds under the convex combinations of the uniform distributions on OU sets. Since $P$ is OU about $x_{0}$, by definition, there is a sequence $\{Q_{m},m\geq1\}$ such that $Q_{m}$ converges weakly to $P$, where $Q_{m}$'s are the convex combinations of the uniform distributions on OU sets. Therefore, these $Q_{m}$'s satisfy (\ref{inequality_of_Q}). Since $P$ has a density, say $f_{0}$, we have
\[
P(\partial N_{\delta}(x^{\prime}))=P(\partial N_{\delta}(x))=0.
\]
Weak convergence and (\ref{inequality_of_Q}) imply that
\begin{equation}
P(N_{\delta}(x^{\prime}))\geq P(N_{\delta}(x)).\label{inequality_of_P}%
\end{equation}
For $x\in\mathcal{D}_{0}^{\circ}$, we define
\[
f(x)=\limsup_{\delta\downarrow0}\frac{P(N_{\delta}(x))}{\lambda(N_{\delta}(x))}=\limsup_{\delta\downarrow0}\frac{P(N_{\delta}(x))}{(2\delta)^{d}}.
\]
By the Lebesgue differentiation theorem, $f(x)=f_{0}(x)$ a.e., which means $f(x)$ is also a density for $P$. Besides, (\ref{inequality_of_P}) implies $f(x^{\prime})\geq f(x)$ for $x,x^{\prime}\in\mathcal{D}_{0}^{\circ}$ with $x\geq x^{\prime}$. For $x\in\partial\mathcal{D}_{0}$, we simply define%
\[
f(x)=\sup_{y\in\mathcal{D}_{0}^{\circ}}f(y).
\]
Then the density $f(x)$ is what we want.

Finally, note that only $W_{\bar{C}_s}$ with $\lambda(\bar{C}_s)>0$ contributes to (\ref{equ_Choquet_OU}). Thus, to show (\ref{equ_Choquet_OU}) is indeed a Choquet representation, it suffices to show $W_{K}$ is an extreme point in the class of OU distributions on $\mathcal{D}_0$ if $K$ is an OU set about $x_0$ with $\lambda(K)>0$. This is proved in Lemma \ref{extreme_point_OU}.
\end{proof}

\begin{proof}[Proof of Lemma \ref{extreme_point_OU}.]
Suppose $W_{K}$ is not an extreme point, i.e., there exist two different OU distributions $P_{1},P_{2}$ and a number $\eta\in(0,1)$ s.t.%
\begin{equation}
W_{K}=\eta P_{1}+(1-\eta)P_{2}.\label{convex_comb_dist}%
\end{equation}
By Lemma \ref{property_of_OU_set}, we know $W_{K}$ and $W_{K^{\circ}}$ are the same distribution. Therefore, the support of $W_{K}$ is the OU set $\bar{K}^{\circ}$ (closure of $K^{\circ}$). Let $K_{1}$ and $K_{2}$ be the support of $P_{1}$ and $P_{2}$ respectively. We have $\bar{K}^{\circ}=K_{1}\cup K_{2}$. Since $W_{K}$ is absolutely continuous with respect to Lebesgue measure, $P_{1}$ and $P_{2}$ are also absolutely continuous with respect to Lebesgue measure. Let $f_{W_{K}}$, $f_{1}$ and $f_{2}$ be densities of $W_{K}$, $P_{1}$ and $P_{2}$ respectively. By Theorem \ref{Choquet_OU}, we can without loss of generality assume that $f_{1}(x)$ and $f_{2}(x)$ are non-increasing in each component of $x$ on $\mathcal{D}_{0}$, which also implies $K_{1}$ and $K_{2}$ are OU sets about $x_{0}$. Then (\ref{convex_comb_dist}) can be written as%
\[
f_{W_{K}}=\eta f_{1}+(1-\eta)f_{2},\text{ a.e..}%
\]
However, since $f_{W_{K}}$ is constant on the OU set $\bar{K}^{\circ}$ and $f_{1}(x)$, $f_{2}(x)$ are non-increasing in each component of $x$, $f_{1}(x)$ and $f_{2}(x)$ must also be constant a.e. on $\bar{K}^{\circ}$. Thus, $P_{1}$ and $P_{2}$ are uniform distributions on $K_{1}$ and $K_{2}$ respectively. Since $\bar{K}^{\circ}$, $K_{1}$ and $K_{2}$ are all OU sets with positive Lebesgue measure, there exists an open set (close to $x_0$) $K^{\prime}\subset\bar{K}^{\circ}\cap K_{1}\cap K_{2}$. By (\ref{convex_comb_dist}), we have%
\[
\frac{\lambda(K^{\prime})}{\lambda(\bar{K}^{\circ})}=\eta\frac{\lambda(K^{\prime})}{\lambda(K_{1})}+(1-\eta)\frac{\lambda(K^{\prime})}{\lambda(K_{2})},
\]
which implies $\lambda(\bar{K}^{\circ})=\lambda(K_{1})=\lambda(K_{2})$. So we know $W_{K}$, $P_{1}$ and $P_{2}$ are the same distribution, which contradicts our assumption that $P_{1}$ and $P_{2}$ are different distributions. Therefore, $W_{K}$ is an extreme point in the class of OU distributions about $x_{0}$.
\end{proof}

\begin{proof}[Proof of Proposition \ref{inclusion_unimodalities}.]
Note that a closed rectangle in $\mathcal{D}_{0}$ which contains $x_{0}$ and have edges parallel to the coordinate axes is OU about $x_{0}$. Moreover, an OU set about $x_{0}$ is also star-shaped about $x_{0}$. Therefore, by the definitions of these three unimodalities, we have%
\begin{align*}
&  \{\text{distributions on }\mathcal{D}_0\text{ which is block unimodal about }x_{0}\}\\
&  \subset\{\text{distributions on }\mathcal{D}_0\text{ which is OU about }x_{0}\}\\
&  \subset\{\text{distributions on }\mathcal{D}_0\text{ which is star unimodal about }x_{0}\}.
\end{align*}
It remains to prove that the inclusions above are proper. For simplicity, we only show it in the bivariate case. Higher dimensional cases give us no additional insight but only the complication of notations. Without loss of generality, we assume $x_{0}$ is the origin.

To show the first inclusion is proper, consider a probability distribution $P$ with the following density:%
\[
f(x,y)=\left\{
\begin{array}
[c]{c}%
1/7\\
2/7\\
0
\end{array}
\left.
\begin{array}
[c]{l}%
\text{if }(x,y)\in(1,2]\times(1,2]\\
\text{if }(x,y)\in\{[0,2]\times\lbrack0,2]\}\backslash\{(1,2]\times(1,2]\}\\
\text{otherwise.}%
\end{array}
\right.  \right.
\]
Let $P_{1}$ be the uniform distribution on $[0,2]\times\lbrack0,2]$ and $P_{2}$ be the uniform distribution on $\{[0,2]\times\lbrack0,2]\}\backslash\{(1,2]\times(1,2]\}$. Then $P$ is a convex combination of $P_{1}$ and $P_{2}$ ($P=4P_{1}/7+3P_{2}/7$). Note that both sets are OU about the origin. Therefore, $P$ is OU about the origin by definition. However, $P$ is not block unimodal about the origin. To see this, consider four points $(1/2,1/2)$, $(3/2,3/2)$, $(3/2,1/2)$ and $(1/2,3/2)$ at which $f(x,y)$ is continuous. If $P$ is block unimodal, we must have%
\[
f\left(  \frac{1}{2},\frac{1}{2}\right)  +f\left(  \frac{3}{2},\frac{3}{2}\right)  \geq f\left(  \frac{3}{2},\frac{1}{2}\right)  +f\left(  \frac{1}{2},\frac{3}{2}\right)
\]
according to Section 2.2 in \cite{dharmadhikari1988unimodality} (or we can show this by the similar proof of the ``only if'' part of Theorem \ref{Choquet_OU}). However,
\[
f\left(  \frac{1}{2},\frac{1}{2}\right)  +f\left(  \frac{3}{2},\frac{3}{2}\right)  =\frac{3}{7}<\frac{4}{7}=f\left(  \frac{3}{2},\frac{1}{2}\right)+f\left(  \frac{1}{2},\frac{3}{2}\right)  ,
\]
which means $P$ is not block unimodal about the origin. This proves the first inclusion is proper.

To show the second inclusion is also proper, we consider the density mentioned in Section \ref{sec:challenges existing MU}:
\[
f(x,y)=C\exp\left(  -\max\left(  \arctan\left(  \frac{y}{x}\right),\arctan\left(  \frac{x}{y}\right)  \right)  (x+y)\right)  ,x\geq0,y\geq0,
\]
where $C$ is the normalizing constant. For any ray $\{(r\cos\theta,r\sin\theta):r\geq0\}$, $\theta\in\lbrack0,\pi/2]$, we can see $f(x,y)$ is non-increasing along it. So according to Section 2.2 in \cite{dharmadhikari1988unimodality}, $f(x,y)$ is star unimodal about the origin. However, it's not OU about the origin because by direct calculation we have
\[
f(1,2)\approx0.036C<0.043C\approx f(2,2),
\]
which violates Theorem \ref{Choquet_OU}. Thus the second inclusion is also proper.
\end{proof}

\begin{proof}[Proof of Corollary \ref{OU-DRO_validity}.]
This follows directly from Lemma \ref{lemma:basic}.
\end{proof}

\begin{proof}[Proof of Proposition \ref{triviality_without_moment_constraints}.]
The proof is constructive. Suppose we remove the moment constraint (\ref{moment_constraints}) from (\ref{DRO_problem_formulation2}) and get the following problem%
\begin{align}
\max\text{ }  &  P((X_{1},\ldots,X_{d})\in S)\nonumber\\
\text{subject to }  &  l_{\bar{F}}\leq\bar{F}(x_{0})\leq u_{\bar{F}}\nonumber\\
&  f_{X_{i}}(x_{i0})\leq u_{X_{i}},i=1,\ldots,d\label{trivial_DRO_problem}\\
&  f(x^{\prime})\ge f(x) \text{ for } x\ge x^{\prime}\ge x_0\nonumber
\end{align}
We can see there is a trivial bound for the optimal value of the problem (\ref{trivial_DRO_problem}), i.e.,%
\begin{equation}
P((X_{1},\ldots,X_{d})\in S)\leq\bar{F}(x_{0})\leq u_{\bar{F}}.
\label{trival_bound}%
\end{equation}
We will construct a sequence of feasible distributions of $(X_{1},\ldots,X_{d})$ s.t.
\begin{equation}
P((X_{1},\ldots,X_{d})\in S)\rightarrow u_{\bar{F}},
\label{target_in_trivial_problem}%
\end{equation}
which implies the optimal value of (\ref{trivial_DRO_problem}) is exactly $u_{\bar{F}}$.

Consider a sequence of rectangles $\{R_{i},i\geq1\}$ in $\mathcal{D}_{0}$ defined as%
\[
R_{i}=\{(x_{1},\ldots,x_{d}):x_{j0}\leq x_{j}\leq x_{j0}+m,j=1,\ldots,d-1\text{ and }x_{d0}\leq x_{d}\leq x_{d0}+i\},
\]
where $1/m<\min(u_{X_{1}},\ldots,u_{X_{d}})$ is a fixed number. Correspondingly, we define $F_{i}$ as the uniform distribution on $R_{i}$ with total mass $u_{\bar{F}}$. The density $f_{i}$ of the distribution $F_{i}$ is given by%
\[
f_{i}(x_{1},\ldots,x_{d})=\frac{u_{\bar{F}}}{m^{d-1}i}I(x_{j0}\leq x_{j}\leq x_{j0}+m,j=1,\ldots,d-1)I(x_{d0}\leq x_{d}\leq x_{d0}+i).
\]
Thus, the constraint that $f(x^{\prime})\ge f(x) \text{ for } x\ge x^{\prime}\ge x_0$ is satisfied. Besides, the truncated marginal density $f_{X_{j}}$ at point $x_{j0}$ satisfies%
\[
f_{X_{j}}(x_{j0})=\frac{u_{\bar{F}}}{m}\leq\frac{1}{m}<u_{X_{j}}\text{ for }j=1,\ldots,d-1,
\]
and%
\[
f_{X_{d}}(x_{d0})=\frac{u_{\bar{F}}}{i}\leq\frac{1}{i}<u_{X_{d}}\text{ when }i\text{ is large enough.}%
\]
So when $i$ is large enough, the distribution $F_{i}$ is feasible to the problem (\ref{trivial_DRO_problem}).

Now let us consider the objective value of $F_{i}$. Recall that $S=\{(x_{1},\ldots,x_{d})\in\mathcal{D}_{0}:x_{d}\geq g(x_{1},\ldots,x_{d-1})\}$ and $g(x_{1},\ldots,x_{d-1})\geq x_{d0}$.
\begin{align*}
&  P_{F_{i}}((X_{1},\ldots,X_{d})\in S)\\
&  =\int\frac{u_{\bar{F}}}{m^{d-1}i}I(x_{j}\in\lbrack x_{j0},x_{j0}+m],j=1,\ldots,d-1,x_{d}\in\lbrack g(x_{1},\ldots,x_{d-1}),x_{d0}+i])dx_{1}\cdots dx_{d}.
\end{align*}
According to our assumption, $g$ is bounded on compact sets. So $\exists$ $M\geq x_{d0}$ s.t. $g(x_{1},\ldots,x_{d-1})\leq M$ when $x_{j0}\leq x_{j}\leq x_{j0}+m,j=1,\ldots,d-1$. Then when $i>M-x_{d0}$, we have%
\begin{align*}
&  P_{F_{i}}((X_{1},\ldots,X_{d})\in S)\\
&  \geq\int\frac{u_{\bar{F}}}{m^{d-1}i}I(x_{j0}\leq x_{j}\leq x_{j0}+m,j=1,\ldots,d-1)I(M\leq x_{d}\leq x_{d0}+i)dx_{1}\cdots dx_{d}\\
&  =u_{\bar{F}}\frac{i+x_{d0}-M}{i}\rightarrow u_{\bar{F}},
\end{align*}
i.e., a sequence of feasible distributions satisfying (\ref{target_in_trivial_problem}) exists.

Finally, it follows by the trivial bound (\ref{trival_bound}) that the optimal value of the problem (\ref{trivial_DRO_problem}) is exactly $u_{\bar{F}}$.
\end{proof}

\begin{proof}[Proof of Lemma \ref{conversion_to_OU_set_optimization}.]
Suppose $f$ is a feasible solution to the problem (\ref{DRO_subproblem_formulation1}) and $C_{s}=\{x\in\mathcal{D}_{0}:f(x)\geq s\}$. We will show the objective function and the constraints in the problem (\ref{DRO_OU_optimization_MD}) are equivalent to those in the problem (\ref{DRO_subproblem_formulation1}).

Notice that $\{C_{s},s>0\}$ is a non-increasing sequence of OU sets by Theorem \ref{Choquet_OU}. First, we show two properties of $\{C_{s},s>0\}$: $\lambda_{d}(C_{s})<\infty$ for any $s>0$, and $\lambda_{d}(C_{s})=0$ implies $C_{t}=\emptyset$ for $t>s$. The first one is obviously true otherwise $\bar{F}(x_{0})\geq s\lambda_{d}(C_{s})=\infty$. Now let us show the second property. Suppose that $\lambda_{d}(C_{s})=0$ and $t>s$. If $C_{t}\neq\emptyset$, then $x_{0}\in C_{t}$ since $x_{0}$ is in any nonempty OU set about $x_{0}$. By the definition $C_{t}=\{x\in\mathcal{D}_{0}:f(x)\geq t\}$, we know that $f(x_{0})\geq t>s$. Since the regularity constraint on $f$ says%
\[
f(x_{0})=\limsup_{y\downarrow x_{0},y\in\mathcal{D}_{0}^{\circ}}f(y),
\]
there exists $y_{0}\in\mathcal{D}_{0}^{\circ}$ s.t. $f(y_{0})>s$. It follows from the OU property that $f(x)>s$ for $x_{0}\leq x\leq y_{0}$. Therefore, $\lambda_{d}(C_{s})>0$, which contradicts our assumption. So we must have $C_{t}=\emptyset$.

Define $s_{0}=\sup\{s>0:\lambda(C_{s})>0\}$. From the last paragraph, we know that $0<\lambda_{d}(C_{s})<\infty$ for $0<s<s_{0}$ and $C_{s}=\emptyset$ for $s>s_{0}$. For the probability measure $P$ induced by the density $f$, from the proof of Theorem \ref{Choquet_OU}, we know that for any measurable set $B$, the following representation holds
\[
P(B)=\int_{0}^{\infty}\lambda_{d}(B\cap\bar{C}_{s})ds=\int_{0}^{s_{0}}\lambda_{d}(B\cap\bar{C}_{s})ds=\int_{0}^{s_{0}}\frac{\lambda_{d}(B\cap\bar{C}_{s})}{\lambda_{d}(\bar{C}_{s})}\lambda_{d}(\bar{C}_{s})ds,
\]
where the last equality holds because $0<\lambda_{d}(C_{s})<\infty$ for $0<s<s_{0}$ and $\lambda_{d}(C_{s})=\lambda_{d}(\bar{C}_{s})$ by Lemma \ref{property_of_OU_set}. When $s>s_{0}$, $C_{s}=\bar{C}_{s}=\emptyset$ and thus $\lambda_{d}(B\cap\bar{C}_{s})/\lambda_{d}(\bar{C}_{s})=0$ by our definition. Then we can extend the integral interval from $(0,s_{0})$ to $(0,\infty)$ and get%
\begin{equation}
P(B)=c\int_{0}^{\infty}\frac{\lambda_{d}(B\cap\bar{C}_{s})}{\lambda_{d}(\bar{C}_{s})}g(s)ds, \label{representation_of_P}%
\end{equation}
where $g(s)=\lambda_{d}(\bar{C}_{s})/c$. In fact, $g(s)$ is a probability density on $(0,\infty)$, which can be verified by setting $B=\mathcal{D}_{0}$ in (\ref{representation_of_P}) and $P(\mathcal{D}_{0})=c$. The representation (\ref{representation_of_P}) ensures the correctness of the representations of the objective function and the last constraint in the problem (\ref{DRO_OU_optimization_MD}).

Now let us show the correctness of the representation of the first constraint in the problem (\ref{DRO_OU_optimization_MD}). Take $i=1$ as an example. Notice that
\[
f(x_{10},x_{2},\ldots,x_{d})=\int_{0}^{f(x_{10},x_{2},\ldots,x_{d})}1ds=\int_{0}^{\infty}I((x_{10},x_{2},\ldots,x_{d})\in C_{s})ds.
\]
Since $C_{s}=\emptyset$ for $s>s_{0}$, it can also be written as%
\[
f(x_{10},x_{2},\ldots,x_{d})=\int_{0}^{s_{0}}I((x_{10},x_{2},\ldots,x_{d})\in C_{s})ds.
\]
Thus, the marginal density $f_{X_{1}}(x_{10})$ is given by%
\begin{align*}
f_{X_{1}}(x_{10})  &  =\int_{x_{20}}^{\infty}\cdots\int_{x_{d0}}^{\infty}f(x_{10},x_{2},\ldots,x_{d})dx_{2}\cdots dx_{d}\\
&  =\int_{x_{20}}^{\infty}\cdots\int_{x_{d0}}^{\infty}\int_{0}^{s_{0}}I((x_{10},x_{2},\ldots,x_{d})\in C_{s})dsdx_{2}\cdots dx_{d}\\
&  =\int_{0}^{s_{0}}\int_{x_{20}}^{\infty}\cdots\int_{x_{d0}}^{\infty}I((x_{10},x_{2},\ldots,x_{d})\in C_{s})dx_{2}\cdots dx_{d}ds\\
&  =\int_{0}^{s_{0}}\lambda_{d-1}(C_{s,1})ds,
\end{align*}
where $C_{s,1}=\{(x_{2},\ldots,x_{d}):(x_{10},x_{2},\ldots,x_{d})\in C_{s}\}$ is slice of $C_{s}$ on the plane $x=x_{10}$. Next, let's show $\lambda_{d-1}(C_{s,1})=\lambda_{d-1}(\bar{C}_{s,1})$, where $\bar{C}_{s,1}$ is the slice of $\bar{C}_{s}$ on the plane $x=x_{10}$. We will show $\bar{C}_{s,1}\subset$ $\overline{C_{s,1}}$, where $\overline{C_{s,1}}$ is the closure of $C_{s,1}$. Suppose $(x_{2},\ldots,x_{d})\in\bar{C}_{s,1}$, i.e., $(x_{10},x_{2},\ldots,x_{d})\in\bar{C}_{s}$. Then there is a sequence $\{(x_{1k},x_{2k},\ldots,x_{dk}),k\geq1\}\subset C_{s}$ such that
\[
(x_{1k},x_{2k},\ldots,x_{dk})\rightarrow(x_{10},x_{2},\ldots,x_{d}).
\]
By the OU property, $\{(x_{10},x_{2k},\ldots,x_{dk}),k\geq1\}$ are also in $C_{s}$, which implies $\{(x_{2k},\ldots,x_{dk}),k\geq1\}\subset C_{s,1}$. Thus, its limiting point $(x_{2},\ldots,x_{d})$ is in $\overline{C_{s,1}}$. This means $\bar{C}_{s,1}\subset$ $\overline{C_{s,1}}$. Clearly, $C_{s,1}\subset\bar{C}_{s,1}$ and $C_{s,1}$ is OU about $(x_{20},\ldots,x_{d0})$ in $\mathbb{R}^{d-1}$. By Lemma \ref{property_of_OU_set}, we can get
\[
\lambda_{d-1}(C_{s,1})\leq\lambda_{d-1}(\bar{C}_{s,1})\leq\lambda_{d-1}(\overline{C_{s,1}})=\lambda_{d-1}(C_{s,1})\Rightarrow\lambda_{d-1}(C_{s,1})=\lambda_{d-1}(\bar{C}_{s,1}).
\]
Thus, we can rewrite $f_{X_{1}}(x_{10})$ as%
\[
f_{X_{1}}(x_{10})=\int_{0}^{s_{0}}\lambda_{d-1}(\bar{C}_{s,1})ds=\int_{0}^{s_{0}}\frac{\lambda_{d-1}(\bar{C}_{s,1})}{\lambda_{d}(\bar{C}_{s})}\lambda_{d}(\bar{C}_{s})ds=c\int_{0}^{s_{0}}\frac{\lambda_{d-1}(\bar{C}_{s,1})}{\lambda_{d}(\bar{C}_{s})}g(s)ds,
\]
where $g(s)=\lambda_{d}(\bar{C}_{s})/c$ is defined above. When $s>s_{0}$, $C_{s}=\bar{C}_{s}=\emptyset$ and thus $\lambda_{d-1}(\bar{C}_{s,1})/\lambda_{d}(\bar{C}_{s})=0$ by our definition. So we can extend the integral interval from $(0,s_{0})$ to $(0,\infty)$ and get%
\[
f_{X_{1}}(x_{10})=c\int_{0}^{\infty}\frac{\lambda_{d-1}(\bar{C}_{s,1})}{\lambda_{d}(\bar{C}_{s})}g(s)ds.
\]
Thus, the representation of the first constraint of the problem (\ref{DRO_OU_optimization_MD}) is correct. This concludes our proof of this lemma.
\end{proof}

\begin{proof}[Proof of Lemma \ref{lemma_general_OU_optimization}.]
It suffices to show for any feasible solution to the problem (\ref{DRO_OU_optimization_MD}), there exists a feasible solution to the problem (\ref{general_OU_optimization_MD}) with the same objective value. Consider a density $f(x)$ on $\mathcal{D}_{0}$ with total mass $c$ and satisfying the last two constraints in (\ref{DRO_subproblem_formulation1}). Let $\{C_{s},s>0\}$ and $g(s)$ be the sequence of OU sets and the probability density defined in Lemma \ref{conversion_to_OU_set_optimization}. In the proof of Lemma \ref{conversion_to_OU_set_optimization}, we know there exists $s_{0}\in(0,\infty)$ s.t. $0<\lambda_{d}(C_{s})=\lambda_{d}(\bar{C}_{s})<\infty$ for $0<s<s_{0}$ and $C_{s}=\emptyset$ for $s>s_{0}$. Therefore, $g(s)=0$ for $s>s_{0}$ by its definition. Consider the slice of $\bar{C}_{s}$, e.g., $\bar{C}_{s,1}$. By the property of OU sets, we know $\bar{C}_{s}\subset\mathbb{R}\times\bar{C}_{s,1}$. Therefore, $\lambda_{d}(\bar{C}_{s})>0$ implies $\lambda_{d-1}(\bar{C}_{s,1})>0$ for $0<s<s_{0}$. Similarly, we can also get $\lambda_{d-1}(\bar{C}_{s,i})>0$ for $0<s<s_{0}$ and $i=1,\ldots,d$. Moreover, we must have $\lambda_{d-1}(\bar{C}_{s,i})<\infty$ for any $s>0$ and $i=1,\ldots,d$. If not, suppose without loss of generality that $\lambda_{d-1}(\bar{C}_{s_{1},1})=\infty$ for some $s_{1}>0$. Since $\bar{C}_{s,1}$ is non-increasing in $s$, we have $\lambda_{d-1}(\bar{C}_{s,1})=\infty$ for $0<s\leq s_{1}$. Now we consider the first constraint in (\ref{DRO_OU_optimization_MD}) with $i=1$:%
\[
\frac{u_{X_{1}}}{c}\geq\int_{0}^{\infty}\frac{\lambda_{d-1}(\bar{C}_{s,1})}{\lambda_{d}(\bar{C}_{s})}g(s)ds\geq\int_{0}^{\min(s_{0},s_{1})}\frac{\lambda_{d-1}(\bar{C}_{s,1})}{\lambda_{d}(\bar{C}_{s})}g(s)ds=\frac{1}{c}\int_{0}^{\min(s_{0},s_{1})}\lambda_{d-1}(\bar{C}_{s,1})ds=\infty,
\]
which contradicts to the assumption that $f$ is a feasible solution to the problem (\ref{DRO_subproblem_formulation1}). In summary, we know that $\lambda_{d}(\bar{C}_{s})\in (0,\infty),\lambda_{d-1}(\bar{C}_{s,i})\in(0,\infty)$ for any $s\in(0,s_{0})$ and $i=1,\ldots,d$, and $g(s)=0$ for $s>s_{0}$.

Now we construct a feasible solution to the problem (\ref{general_OU_optimization_MD}) with the same objective value as $f(x)$. Notice that $s\mapsto\tan(\pi s/2s_{0})$ is a bijection from $(0,s_{0})$ into $(0,\infty)$ and%
\begin{equation}
\int_{0}^{s_{0}}\varphi(s)g(s)ds=\int_{0}^{\infty}\varphi\left(  \frac{2s_{0}\arctan(u)}{\pi}\right)  g\left(  \frac{2s_{0}\arctan(u)}{\pi}\right)\frac{2s_{0}}{\pi(1+u^{2})}du\label{change_of_variable}
\end{equation}
for any non-negative measurable function $\varphi$ by the change of variable $u=\tan(\pi s/2s_{0})$. Therefore, if we define $R_{s}$ by%
\[
R_{s}=\bar{C}_{\frac{2s_{0}\arctan(u)}{\pi}},s\in(0,\infty)
\]
and define $G(s)$ as the probability distribution with density
\[
G^{\prime}(s)=g\left(  \frac{2s_{0}\arctan(s)}{\pi}\right)  \frac{2s_{0}}{\pi(1+s^{2})},s\in(0,\infty),
\]
it follows by (\ref{change_of_variable}) that the objective value and the constraints in (\ref{DRO_OU_optimization_MD}) and (\ref{general_OU_optimization_MD}) are exactly the same. Further, $R_{s}$ satisfies $\lambda_d(R_{s})\in(0,\infty),\lambda_{d-1}(R_{s,i})\in(0,\infty)$ for any $i=1,\ldots,d$ and $s>0$, which means $\{R_{s},s>0\}$ and $G(s)$ are feasible to the problem (\ref{general_OU_optimization_MD}). This concludes our proof.
\end{proof}

\begin{proof}[Proof of Corollary \ref{cor_general_OU_optimization_2D}.]
This follows directly from Lemma \ref{lemma_general_OU_optimization} when $d=2$.
\end{proof}

\begin{proof}[Proof of Lemma \ref{lemma_reduction_of_constraints}.]
By the definition of $R$ in (\ref{definition_R}) and the constraints of $h$ in (\ref{reduction_of_constraints}), we can easily see that $R$ is a closed OU set with $R^{X}\leq R_0^{X}$ and $R^{Y}=h(x_{0})\leq h_{0}(x_{0}+)\leq R_0^{Y}$. By the first constraint in (\ref{reduction_of_constraints}), we have
\begin{align*}
\lambda(R) &  =\int_{x_{0}}^{x_{0}+r_{X}}(h(x)-y_{0})dx=\sum_{i=0}^{n_{R_{0}}}\int_{x_{i}}^{x_{i+1}}(h(x)-y_{0})dx\\
&  =\sum_{i=0}^{n_{R_{0}}}\lambda(\{(x,y):x_{i}\leq x\leq x_{i+1},y\geq y_{0}\}\cap R_{0})\\
&  =\lambda(\{(x,y):x_{0}\leq x\leq x_{0}+r_{X},y\geq y_{0}\}\cap R_{0})=\lambda(R_{0}).
\end{align*}
So it remains to verify $\lambda(\{(x,y):x_{1i}\leq x\leq x_{2i},y_{1i}\leq y\leq y_{2i}\}\cap R)=\lambda(\{(x,y):x_{1i}\leq x\leq x_{2i},y_{1i}\leq y\leq y_{2i}\}\cap R_{0}),i=1,\ldots,n$. There are several cases for these equalities, considering whether $h_{0}^{-1}(y_{1i})$ and $h_{0}^{-1}(y_{2i})$ are well defined and all possible orders of $x_{1i},x_{2i},h_{0}^{-1}(y_{1i})$ and $h_{0}^{-1}(y_{2i})$. Here we only consider one case and other cases can be proved similarly. Suppose that both $h_{0}^{-1}(y_{1i})$ and $h_{0}^{-1}(y_{2i})$ are well defined and $x_{1i}<h_{0}^{-1}(y_{2i})<h_{0}^{-1}(y_{1i})<x_{2i}$. In this case, $x_{1i},h_{0}^{-1}(y_{2i}),h_{0}^{-1}(y_{1i})$ and $x_{2i}$ are distinct in the ordered sequence $x_{0}<x_{1}<\cdots<x_{n_{R_{0}}+1}$. By the definition $h_{0}^{-1}(y)=\sup\{x:h_{0}(x)\geq y\}$ and the left-continuity of $h_{0}$, we must have $h_{0}(h_{0}^{-1}(y))\geq y$ and $h_{0}(h_{0}^{-1}(y)+)\leq y$. Notice that we have
\begin{align*}
&  \lambda(\{(x,y):x_{1i}\leq x\leq x_{2i},y_{1i}\leq y\leq y_{2i}\}\cap R_{0})\\
&  =(y_{2i}-y_{1i})(h_{0}^{-1}(y_{2i})-x_{1i})+\int_{h_{0}^{-1}(y_{2i})}^{h_{0}^{-1}(y_{1i})}(h_{0}(x)-y_{1i})dx\\
&  =(y_{2i}-y_{1i})(h_{0}^{-1}(y_{2i})-x_{1i})+\lambda(\{(x,y):h_{0}^{-1}(y_{2i})\leq x\leq h_{0}^{-1}(y_{1i}),y\geq y_{0}\}\cap R_{0})\\
&  -(y_{1i}-y_{0})(h_{0}^{-1}(y_{1i})-h_{0}^{-1}(y_{2i})).
\end{align*}
Now let us calculate $\lambda(\{(x,y):x_{1i}\leq x\leq x_{2i},y_{1i}\leq y\leq y_{2i}\}\cap R)$. By the second and third constraints in (\ref{reduction_of_constraints}), we can see: $h(x)\geq h_{0}(h_{0}^{-1}(y_{2i}))\geq y_{2i}$ if $x\in(x_{1i},h_{0}^{-1}(y_{2i})]$, $y_{1i}\leq h_{0}(h_{0}^{-1}(y_{1i}))\leq h(x)\leq h_{0}(h_{0}^{-1}(y_{2i})+)\leq y_{2i}$ if $x\in(h_{0}^{-1}(y_{2i}),h_{0}^{-1}(y_{1i})]$ and $h(x)\leq h_{0}(h_{0}^{-1}(y_{1i})+)\leq y_{1i}$ if $x\in(h_{0}^{-1}(y_{1i}),x_{2i}]$. So we have%
\[
\lambda(\{(x,y):x_{1i}\leq x\leq x_{2i},y_{1i}\leq y\leq y_{2i}\}\cap R)=(y_{2i}-y_{1i})(h_{0}^{-1}(y_{2i})-x_{1i})+\int_{h_{0}^{-1}(y_{2i})}^{h_{0}^{-1}(y_{1i})}(h(x)-y_{1i})dx.
\]
Note that%
\begin{align*}
&  \int_{h_{0}^{-1}(y_{2i})}^{h_{0}^{-1}(y_{1i})}(h(x)-y_{1i})dx\\
&  =\int_{h_{0}^{-1}(y_{2i})}^{h_{0}^{-1}(y_{1i})}(h(x)-y_{0})dx-(y_{1i}-y_{0})(h_{0}^{-1}(y_{1i})-h_{0}^{-1}(y_{2i}))\\
&  =\lambda(\{(x,y):h_{0}^{-1}(y_{2i})\leq x\leq h_{0}^{-1}(y_{1i}),y\geq y_{0}\}\cap R_{0})-(y_{1i}-y_{0})(h_{0}^{-1}(y_{1i})-h_{0}^{-1}(y_{2i})),
\end{align*}
where the second equality follows from the first constraint in (\ref{reduction_of_constraints}). It implies $\lambda(\{(x,y):x_{1i}\leq x\leq x_{2i},y_{1i}\leq y\leq y_{2i}\}\cap R)=\lambda(\{(x,y):x_{1i}\leq x\leq x_{2i},y_{1i}\leq y\leq y_{2i}\}\cap R_{0})$ holds in this case. For other cases, the proofs are similar and thus are omitted. This concludes our proof.
\end{proof}

\begin{proof}[Proof of Lemma \ref{find_dom_OU}.]
Since the function $h_0$ corresponding to $R_0$ is a feasible solution to (\ref{function_optimization_problem}), by (\ref{representation_obj}) and the optimality of $h^*$, we have $\lambda(S\cap \tilde{R}) \ge \lambda(S\cap R_0)$. Moreover, by Lemma \ref{lemma_reduction_of_constraints}, $\tilde{R}$ satisfies all the remaining requirements of the dominating OU set of $R_0$.
\end{proof}

To prove Lemma \ref{opt_geometric_problem}, we first prove the following simple lemma.

\begin{lemma}
\label{convexity_of_auxiliary_function}Suppose $a_{3}>a_{2}>a_{1}$ and $c$ are some constants in $\mathbb{R}$, and $g:(a_{1},a_{3}]\mapsto(-\infty,\infty]$. Consider the following function%
\[
H\left(  y\right)  =\int_{a_{1}}^{a_{2}}\left(  y-g\left(  x\right)  \right)_{+}dx+\int_{a_{2}}^{a_{3}}\left(  \frac{c-\left(  a_{2}-a_{1}\right) y}{a_{3}-a_{2}}-g\left(  x\right)  \right)  _{+}dx
\]
whose domain is $y\in\lbrack\underline{y},\bar{y}].$ Then the maximum value of $H\left(  y\right)  $ is attained at the endpoints of the domain.
\end{lemma}

\begin{proof}[Proof of Lemma \ref{convexity_of_auxiliary_function}.]
It suffices to show that $H\left(  y\right)  $ is a convex function. For $y_{1},y_{2}\in\lbrack\underline{y},\bar{y}]$ and $\eta\in\left[  0,1\right]$, we have%
\begin{align*}
&  H\left(  \eta y_{1}+\left(  1-\eta\right)  y_{2}\right) \\
&  =\int_{a_{1}}^{a_{2}}\left(  \eta y_{1}+\left(  1-\eta\right) y_{2}-g\left(  x\right)  \right)  _{+}dx+\int_{a_{2}}^{a_{3}}\left(\frac{c-\left(  a_{2}-a_{1}\right)  \left(  \eta y_{1}+\left(  1-\eta\right) y_{2}\right)  }{a_{3}-a_{2}}-g\left(  x\right)  \right)  _{+}dx\\
&  =\int_{a_{1}}^{a_{2}}\left(  \eta\left(  y_{1}-g\left(  x\right)  \right)+\left(  1-\eta\right)  \left(  y_{2}-g\left(  x\right)  \right)  \right)_{+}dx\\
&  +\int_{a_{2}}^{a_{3}}\left(  \eta\left(  \frac{c-\left(  a_{2}-a_{1}\right)  y_{1}}{a_{3}-a_{2}}-g\left(  x\right)  \right)  +\left(1-\eta\right)  \left(  \frac{c-\left(  a_{2}-a_{1}\right)  y_{2}}{a_{3}-a_{2}}-g\left(  x\right)  \right)  \right)  _{+}dx\\
&  \leq\int_{a_{1}}^{a_{2}}\left[  \eta\left(  y_{1}-g\left(  x\right)\right)  _{+}+\left(  1-\eta\right)  \left(  y_{2}-g\left(  x\right)  \right)_{+}\right]  dx\\
&  +\int_{a_{2}}^{a_{3}}\left[  \eta\left(  \frac{c-\left(  a_{2}-a_{1}\right)  y_{1}}{a_{3}-a_{2}}-g\left(  x\right)  \right)  _{+}+\left(1-\eta\right)  \left(  \frac{c-\left(  a_{2}-a_{1}\right)  y_{2}}{a_{3}-a_{2}}-g\left(  x\right)  \right)  _{+}\right]  dx\\
&  =\eta H\left(  y_{1}\right)  +\left(  1-\eta\right)  H\left(  y_{2}\right).
\end{align*}
Therefore, $H(y)$ is a convex function.
\end{proof}

Now, let us show Lemma \ref{opt_geometric_problem}.
\begin{proof}[Proof of Lemma \ref{opt_geometric_problem}.]
We define a subproblem of the problem (\ref{geometric_problem}) by adding an additional constraint:%
\begin{align}
\max\text{ } &  \int_{\underline{x}}^{\bar{x}}(h(x)-g(x))_{+}dx\nonumber\\
\text{subject to } &  \int_{\underline{x}}^{\bar{x}}h(x)dx=C\label{subproblem1_of_geometric_problem}\\
&  \underline{b}\leq h(x)\leq\bar{b},x\in(\underline{x},\bar{x}]\nonumber\\
&  h\text{ is a non-increasing left-continuous step function}\nonumber
\end{align}
where the step function is in the form of%
\[
h\left(  x\right)  =\left\{
\begin{array}
[c]{c}%
y_{1},\\
y_{2},\\
\cdots\\
y_{n},
\end{array}
\left.
\begin{array}
[c]{l}%
a_{0}<x\leq a_{1}\\
a_{1}<x\leq a_{2}\\
\\
a_{n-1}<x\leq a_{n}%
\end{array}
\right.  \right.  ,
\]
where $n\in\mathbb{N}$, $\underline{x}=a_{0}<a_{1}<\cdots<a_{n}=\bar{x}$ and $y_{i}$'s are some constants. To clarify, here we abuse the notations $a_{i}$ and $n$. They are not the constants in the formulation of\ our DRO problem (\ref{DRO_problem_formulation1}). We aim at showing the optimal value of the problem (\ref{subproblem1_of_geometric_problem}) is the same as that of the problem (\ref{geometric_problem}). It suffices to show that for any feasible solution $h$ to the problem (\ref{geometric_problem}), there exists a sequence of feasible solutions $\left\{  h_{n},n\geq1\right\}  $ to the problem (\ref{subproblem1_of_geometric_problem}), s.t.%
\begin{equation}
\int_{\underline{x}}^{\bar{x}}\left(  h_{n}\left(  x\right)  -g\left(x\right)  \right)  _{+}dx\rightarrow\int_{\underline{x}}^{\bar{x}}\left(h\left(  x\right)  -g\left(  x\right)  \right)  _{+}dx.\label{target_to_justify_geometric_subproblem1}%
\end{equation}

Now fix any feasible solution $h$ to the problem (\ref{geometric_problem}). For $n\geq1$, let's construct $h_{n}$. For simplicity, we extend the definition of $h$ to $[\underline{x},\bar{x}]$ by defining $h(\underline{x})=h(\underline{x}+)$. For $i\in\left\{  0,1,\ldots,n-1\right\}  $, since $h$ is non-increasing, we have%
\[
h\left(  \underline{x}+\frac{(i+1)(\bar{x}-\underline{x})}{n}\right)\frac{(\bar{x}-\underline{x})}{n}\leq\int_{\underline{x}+\frac{i(\bar{x}-\underline{x})}{n}}^{\underline{x}+\frac{(i+1)(\bar{x}-\underline{x})}{n}}h\left(  x\right)  dx\leq h\left(  \underline{x}+\frac{i(\bar{x}-\underline{x})}{n}\right)  \frac{(\bar{x}-\underline{x})}{n}.
\]
Therefore, there is a real number $h_{i,n}$ satisfying
\[
h_{i,n}\in\left[  h\left(  \underline{x}+\frac{(i+1)(\bar{x}-\underline{x})}{n}\right)  ,h\left(  \underline{x}+\frac{i(\bar{x}-\underline{x})}{n}\right)  \right]
\]
s.t.%
\[
\int_{\underline{x}+\frac{i(\bar{x}-\underline{x})}{n}}^{\underline{x}+\frac{(i+1)(\bar{x}-\underline{x})}{n}}h\left(  x\right)  dx=h_{i,n}\frac{(\bar{x}-\underline{x})}{n}.
\]
Then our $h_{n}$ is defined as
\[
h_{n}\left(  x\right)  =h_{i,n},\forall x\in\left(  \underline{x}+\frac{i(\bar{x}-\underline{x})}{n},\underline{x}+\frac{(i+1)(\bar{x}-\underline{x})}{n}\right]  .
\]
We can see $h_{n}\ $is a feasible solution to the problem (\ref{subproblem1_of_geometric_problem}) by our construction. Let us verify (\ref{target_to_justify_geometric_subproblem1}).%
\begin{align*}
&  \left\vert \int_{\underline{x}}^{\bar{x}}\left(  h_{n}\left(  x\right)-g\left(  x\right)  \right)  _{+}dx-\int_{\underline{x}}^{\bar{x}}\left(h\left(  x\right)  -g\left(  x\right)  \right)  _{+}dx\right\vert \\
&  \leq\int_{\underline{x}}^{\bar{x}}|\left(  h_{n}\left(  x\right)  -g\left(x\right)  \right)  _{+}-\left(  h\left(  x\right)  -g\left(  x\right)\right)  _{+}|dx\\
&  \leq\int_{\underline{x}}^{\bar{x}}|\left(  h_{n}\left(  x\right)  -g\left(x\right)  \right)  -\left(  h\left(  x\right)  -g\left(  x\right)  \right)|dx\\
&  =\sum_{i=0}^{n-1}\int_{\underline{x}+\frac{i(\bar{x}-\underline{x})}{n}}^{\underline{x}+\frac{(i+1)(\bar{x}-\underline{x})}{n}}\left\vert h_{n}\left(  x\right)  -h\left(  x\right)  \right\vert dx\\
&  \leq\sum_{i=0}^{n-1}\int_{\underline{x}+\frac{i(\bar{x}-\underline{x})}{n}}^{\underline{x}+\frac{(i+1)(\bar{x}-\underline{x})}{n}}\left(  h\left(\underline{x}+\frac{i(\bar{x}-\underline{x})}{n}\right)  -h\left(\underline{x}+\frac{(i+1)(\bar{x}-\underline{x})}{n}\right)  \right)  dx\\
&  =\frac{(\bar{x}-\underline{x})}{n}(h(\underline{x})-h(\bar{x}))\rightarrow0
\end{align*}
as $n\rightarrow\infty$. Thus, the optimal value of the problem (\ref{subproblem1_of_geometric_problem}) is the same as that of the problem (\ref{geometric_problem}).

Next, we define a subproblem of the problem (\ref{subproblem1_of_geometric_problem}) by adding an additional constraint again:%
\begin{align}
\max\text{ } &  \int_{\underline{x}}^{\bar{x}}(h(x)-g(x))_{+}dx\nonumber\\
\text{subject to } &  \int_{\underline{x}}^{\bar{x}}h(x)dx=C\label{subproblem2_of_geometric_problem}\\
&  \underline{b}\leq h(x)\leq\bar{b},x\in(\underline{x},\bar{x}]\nonumber\\
&  h\text{ is a non-increasing left-continuous step function with at most three steps}\nonumber
\end{align}
Again, we aim at showing the optimal value of the problem (\ref{subproblem2_of_geometric_problem}) is the same as that of the problem (\ref{subproblem1_of_geometric_problem}). It suffices to show that for any feasible solution $h_{1}$ to the problem (\ref{subproblem1_of_geometric_problem}), there is a feasible solution $h_{2}$ to the problem (\ref{subproblem2_of_geometric_problem}) s.t.
\begin{equation}
\int_{\underline{x}}^{\bar{x}}\left(  h_{1}\left(  x\right)  -g\left(x\right)  \right)  _{+}dx\leq\int_{\underline{x}}^{\bar{x}}\left(h_{2}\left(  x\right)  -g\left(  x\right)  \right)  _{+}dx.\label{target_to_justify_geometric_subproblem2}%
\end{equation}

Notice that for any feasible solution $h_{1}$ of the problem (\ref{subproblem1_of_geometric_problem}), it can be uniquely written as%
\[
h_{1}\left(  x\right)  =\left\{
\begin{array}
[c]{c}%
y_{1},\\
y_{2},\\
\cdots\\
y_{n},
\end{array}
\left.
\begin{array}
[c]{l}%
a_{0}<x\leq a_{1}\\
a_{1}<x\leq a_{2}\\
\\
a_{n-1}<x\leq a_{n}%
\end{array}
\right.  \right.  ,
\]
where $\underline{b}\leq y_{n}<y_{n-1}<\cdots<y_{1}\leq\bar{b}$ and $\underline{x}=a_{0}<a_{1}<\cdots<a_{n}=\bar{x}$. Now let's prove (\ref{target_to_justify_geometric_subproblem2}) by induction with respect to $n$ (hereafter we will call $n$ the number of steps of $h_{1}$). When $n\leq 3$, $h_{1}\left(  x\right)  $ is also a feasible solution to the problem (\ref{subproblem2_of_geometric_problem}). So we can choose $h_{2}=h_{1}$ and then (\ref{target_to_justify_geometric_subproblem2}) holds. Assume that when $n\leq k$ ($k\geq3$), the conclusion is true. Now suppose $n=k+1\geq4$.

Let's consider the steps $y_{k-1}$ and $y_{k}$, i.e., consider the interval $(a_{k-2},a_{k}]$. So the objective value in this interval is given by%
\[
\int_{a_{k-2}}^{a_{k-1}}\left(  y_{k-1}-g\left(  x\right)  \right)_{+}dx+\int_{a_{k-1}}^{a_{k}}\left(  y_{k}-g\left(  x\right)  \right)  _{+}dx.
\]
We want to change the function values in $(a_{k-2},a_{k-1}]$ and $(a_{k-1},a_{k}]$ to get a larger objective value while keeping the function feasible for the problem (\ref{subproblem2_of_geometric_problem}). Suppose $y_{k-1}$ is changed into $y$. To make $\int_{\underline{x}}^{\bar{x}}h(x)dx$ fixed, we must change $y_{k}$ into%
\[
\frac{y_{k-1}\left(  a_{k-1}-a_{k-2}\right)  +y_{k}\left(  a_{k}-a_{k-1}\right)  -\left(  a_{k-1}-a_{k-2}\right)  y}{\left(  a_{k}-a_{k-1}\right)  }.
\]
To make the function still non-increasing, $y$ should satisfy%
\[
y_{k+1}\leq\frac{y_{k-1}\left(  a_{k-1}-a_{k-2}\right)  +y_{k}\left(a_{k}-a_{k-1}\right)  -\left(  a_{k-1}-a_{k-2}\right)  y}{\left(a_{k}-a_{k-1}\right)  }\leq y\leq y_{k-2},
\]
i.e.,%
\[
y\in\left[  y_{k-1}\frac{a_{k-1}-a_{k-2}}{a_{k}-a_{k-2}}+y_{k}\frac{a_{k}-a_{k-1}}{a_{k}-a_{k-2}},\min\left(  y_{k-2},y_{k-1}+(y_{k}-y_{k+1})\frac{a_{k}-a_{k-1}}{a_{k-1}-a_{k-2}}\right)  \right]  .
\]
Now the objective value in $(a_{k-2},a_{k}]$ becomes%
\begin{align*}
&  H(y)=\int_{a_{k-2}}^{a_{k-1}}\left(  y-g\left(  x\right)  \right)  _{+}dx\\
&  +\int_{a_{k-1}}^{a_{k}}\left(  \frac{y_{k-1}\left(  a_{k-1}-a_{k-2}\right)+y_{k}\left(  a_{k}-a_{k-1}\right)  -\left(  a_{k-1}-a_{k-2}\right) y}{\left(  a_{k}-a_{k-1}\right)  }-g\left(  x\right)  \right)  _{+}dx.
\end{align*}
Therefore, by Lemma \ref{convexity_of_auxiliary_function}, we have that%
\begin{align*}
&  \int_{a_{k-2}}^{a_{k-1}}\left(  y_{k-1}-g\left(  x\right)  \right)_{+}dx+\int_{a_{k-1}}^{a_{k}}\left(  y_{k}-g\left(  x\right)  \right)_{+}dx=H\left(  y_{k-1}\right)  \\
&  \leq\max\left(  H\left(  y_{k-1}\frac{a_{k-1}-a_{k-2}}{a_{k}-a_{k-2}}+y_{k}\frac{a_{k}-a_{k-1}}{a_{k}-a_{k-2}}\right)  ,H\left(  \min\left(y_{k-2},y_{k-1}+(y_{k}-y_{k+1})\frac{a_{k}-a_{k-1}}{a_{k-1}-a_{k-2}}\right)\right)  \right)  .
\end{align*}
By considering which one attains the maximal value, there are three cases.

Case 1:
\[
H\left(  y_{k-1}\frac{a_{k-1}-a_{k-2}}{a_{k}-a_{k-2}}+y_{k}\frac{a_{k}-a_{k-1}}{a_{k}-a_{k-2}}\right)  >H\left(  \min\left(  y_{k-2},y_{k-1}+(y_{k}-y_{k+1})\frac{a_{k}-a_{k-1}}{a_{k-1}-a_{k-2}}\right)  \right)  .
\]
In this case,%
\[
H\left(  y_{k-1}\right)  \leq H\left(  y_{k-1}\frac{a_{k-1}-a_{k-2}}{a_{k}-a_{k-2}}+y_{k}\frac{a_{k}-a_{k-1}}{a_{k}-a_{k-2}}\right)  ,
\]
which means we can get a better objective value by changing both $y_{k-1}$ and $y_{k}$ into
\[
y_{k-1}\frac{a_{k-1}-a_{k-2}}{a_{k}-a_{k-2}}+y_{k}\frac{a_{k}-a_{k-1}}{a_{k}-a_{k-2}}.
\]
In other words, the new function%
\[
h\left(  x\right)  =\left\{
\begin{array}
[c]{l}%
y_{1},\\
\cdots\\
y_{k-2},\\
y_{k-1}\frac{a_{k-1}-a_{k-2}}{a_{k}-a_{k-2}}+y_{k}\frac{a_{k}-a_{k-1}}{a_{k}-a_{k-2}},\\
y_{k+1},
\end{array}
\left.
\begin{array}
[c]{l}%
a_{0}<x\leq a_{1}\\
\\
a_{k-3}<x\leq a_{k-2}\\
a_{k-2}<x\leq a_{k}\\
a_{k}<x\leq a_{k+1}%
\end{array}
\right.  \right.
\]
has a better objective value than the original function $h_{1}(x)$. Since the number of steps of $h\left(  x\right)  $ is $k$, by the induction hypothesis, there is a feasible solution $h_{2}$ to the problem (\ref{subproblem2_of_geometric_problem}) s.t.%
\[
\int_{\underline{x}}^{\bar{x}}\left(  h\left(  x\right)  -g\left(  x\right)\right)  _{+}dx\leq\int_{\underline{x}}^{\bar{x}}\left(  h_{2}\left(x\right)  -g\left(  x\right)  \right)  _{+}dx,
\]
which implies%
\[
\int_{\underline{x}}^{\bar{x}}\left(  h_{1}\left(  x\right)  -g\left(x\right)  \right)  _{+}dx\leq\int_{\underline{x}}^{\bar{x}}\left(h_{2}\left(  x\right)  -g\left(  x\right)  \right)  _{+}dx.
\]
So (\ref{target_to_justify_geometric_subproblem2}) holds in this case.

Case 2:%
\[
H\left(  y_{k-1}\frac{a_{k-1}-a_{k-2}}{a_{k}-a_{k-2}}+y_{k}\frac{a_{k}-a_{k-1}}{a_{k}-a_{k-2}}\right)  \leq H\left(  \min\left(  y_{k-2},y_{k-1}+(y_{k}-y_{k+1})\frac{a_{k}-a_{k-1}}{a_{k-1}-a_{k-2}}\right)  \right),
\]
and%
\[
\min\left(  y_{k-2},y_{k-1}+(y_{k}-y_{k+1})\frac{a_{k}-a_{k-1}}{a_{k-1}-a_{k-2}}\right)  =y_{k-2}.
\]
This is to say%
\[
H\left(  y_{k-1}\right)  \leq H\left(  y_{k-2}\right)  ,
\]
and%
\[
y_{k}+(y_{k-1}-y_{k-2})\frac{a_{k-1}-a_{k-2}}{a_{k}-a_{k-1}}\geq y_{k+1},
\]
which means we can get a better objective value by changing $y_{k-1}$ into $y_{k-2}$. In other words, the new function%
\[
h\left(  x\right)  =\left\{
\begin{array}
[c]{l}%
y_{1},\\
\cdots\\
y_{k-2},\\
y_{k}+(y_{k-1}-y_{k-2})\frac{a_{k-1}-a_{k-2}}{a_{k}-a_{k-1}},\\
y_{k+1},
\end{array}
\left.
\begin{array}
[c]{l}%
a_{0}<x\leq a_{1}\\
\\
a_{k-3}<x\leq a_{k-1}\\
a_{k-1}<x\leq a_{k}\\
a_{k}<x\leq a_{k+1}%
\end{array}
\right.  \right.
\]
has a better objective value than the original function $h_{1}(x)$. We can see the number of steps of $h\left(  x\right)  $ is $k$. By the same argument in case 1, (\ref{target_to_justify_geometric_subproblem2}) also holds in this case.

Case 3:%
\[
H\left(  y_{k-1}\frac{a_{k-1}-a_{k-2}}{a_{k}-a_{k-2}}+y_{k}\frac{a_{k}-a_{k-1}}{a_{k}-a_{k-2}}\right)  \leq H\left(  \min\left(  y_{k-2},y_{k-1}+(y_{k}-y_{k+1})\frac{a_{k}-a_{k-1}}{a_{k-1}-a_{k-2}}\right)  \right),
\]
and%
\[
\min\left(  y_{k-2},y_{k-1}+(y_{k}-y_{k+1})\frac{a_{k}-a_{k-1}}{a_{k-1}-a_{k-2}}\right)  =y_{k-1}+(y_{k}-y_{k+1})\frac{a_{k}-a_{k-1}}{a_{k-1}-a_{k-2}}.
\]
This is to say%
\[
H\left(  y_{k-1}\right)  \leq H\left(  y_{k-1}+(y_{k}-y_{k+1})\frac{a_{k}-a_{k-1}}{a_{k-1}-a_{k-2}}\right)  ,
\]
and%
\[
y_{k-2}\geq y_{k-1}+(y_{k}-y_{k+1})\frac{a_{k}-a_{k-1}}{a_{k-1}-a_{k-2}},
\]
which means we can get a better objective value by changing $y_{k-1}$ into $y_{k-1}+(y_{k}-y_{k+1})(a_{k}-a_{k-1})/(a_{k-1}-a_{k-2})$ and changing $y_{k}$ into $y_{k+1}$. In other words, the new function%
\[
h\left(  x\right)  =\left\{
\begin{array}
[c]{l}%
y_{1},\\
\cdots\\
y_{k-2},\\
y_{k-1}+(y_{k}-y_{k+1})\frac{a_{k}-a_{k-1}}{a_{k-1}-a_{k-2}},\\
y_{k+1},
\end{array}
\left.
\begin{array}
[c]{l}%
a_{0}<x\leq a_{1}\\
\\
a_{k-3}<x\leq a_{k-2}\\
a_{k-2}<x\leq a_{k-1}\\
a_{k-1}<x\leq a_{k+1}%
\end{array}
\right.  \right.
\]
has a better objective value than the original function $h_{1}(x)$. We can see the number of steps of $h\left(  x\right)  $ is $k$. By the same argument in case 1, (\ref{target_to_justify_geometric_subproblem2}) also holds in this case.

Combining all the cases, we can see (\ref{target_to_justify_geometric_subproblem2}) holds when $n=k+1$. By induction, we can see (\ref{target_to_justify_geometric_subproblem2}) holds for any feasible solution $h_{1}$ to the problem (\ref{subproblem1_of_geometric_problem}), which implies the optimal value of the problem (\ref{subproblem2_of_geometric_problem}) is the same as that of the problem (\ref{subproblem1_of_geometric_problem}).

The last thing we need to show is the optimal solution to the problem (\ref{subproblem2_of_geometric_problem}) exists. Suppose we have a sequence of feasible solutions $\{h_{n},n\geq1\}$ to the problem (\ref{subproblem2_of_geometric_problem}) whose objective values converge to the optimal value. Since each $h_{n}$ is characterized by five bounded real numbers, we can find a subsequence along which the five sequences of real numbers converge. This means we can find an a.e. convergent subsequence of $\{h_{n},n\geq1\}$. By the dominated converge theorem, we can see the limit function is also feasible to the problem (\ref{subproblem2_of_geometric_problem}) (we may need to change its values at the endpoints of the steps to make it left-continuous) and its objective value is exactly the optimal value. Thus, the limit function is the optimal solution to the problem (\ref{subproblem2_of_geometric_problem}).
\end{proof}

\begin{proof}[Proof of Corollary \ref{characterization_dom_OU}.]
By Lemma \ref{opt_geometric_problem}, we know that each $h_{i}^{\ast}$ is a step function with at most three steps. In total we have $n_{R_{0}}+1\leq4n+1$ subproblems and $h^{\ast}$ is just a ``combination" of these $h_i^*$'s. Therefore, $h^{\ast}$ is a non-increasing left-continuous step function with at most $3(4n+1)$ steps and can be represented by (\ref{12n+3_step_function}) for some $(z,w)$. Consequently, $\tilde{R}$ defined by
\begin{equation}
\tilde{R}=\{(x,y)\in\mathcal{D}_{0}:y_{0}\leq y\leq h^*(x;z,w)\} \label{equ_characterization_dom_OU}
\end{equation}
is a dominating OU set of $R_0$ by Lemma \ref{find_dom_OU}. Here the requirement $(z,w)\in(0,\infty)^{12n+4}\times[0,\infty)^{12n+2}$ is due to $\lambda(\tilde{R})=\lambda(R_0)>0$. Note that the number of steps $12n+3$ doesn't rely on the choice of $R_0$ so the dominating OU set of any closed OU set $R_0$ with $\lambda(R_0)\in (0,\infty),R_0^{X}\in (0,\infty),R_0^{Y}\in (0,\infty)$ can be represented by (\ref{equ_characterization_dom_OU}) with possibly different $(z,w)$. In other words, all dominating OU sets are contained in the class $\mathcal{R}^{\ast}=\{R_{z,w}:R_{z,w}=\{(x,y)\in\mathcal{D}_{0}:y_{0}\leq y\leq h^*(x;z,w)\}$ with $h^*$ defined in (\ref{12n+3_step_function})$\}$.
\end{proof}

\begin{proof}[Proof of Proposition \ref{equiv_moment_problem}.]
By the construction of dominating OU sets, in problem (\ref{general_OU_optimization_2D}), we can restrict the choice of $R_{s}$ to $\mathcal{R}^{\ast}$ which is fully parameterized by $(z,w)$. Further, the cardinality of $(z,w)\in(0,\infty)^{12n+4}\times[0,\infty)^{12n+2}$ and the cardinality of $s\in(0,\infty)$ are both continuum. So the distribution $G(s)$ on the index $s$ in problem (\ref{general_OU_optimization_2D}) is equivalent to a distribution $Q$ on $(z,w)$. Thus, problem (\ref{general_OU_optimization_2D}) can be rewritten as the following problem:
\begin{align}
\max\text{ } &  c\int\frac{\lambda(S\cap R_{z,w})}{\lambda(R_{z,w})}dQ(z,w)\nonumber\\
\text{subject to } &  \int\frac{R_{z,w}^{Y}}{\lambda(R_{z,w})}dQ(z,w)\leq\frac{u_{X}}{c}\nonumber\\
&  \int\frac{R_{z,w}^{X}}{\lambda(R_{z,w})}dQ(z,w)\leq\frac{u_{Y}}{c}\label{change_index}\\
&  a_{i}\leq\int\frac{\lambda(\{(x,y):x_{1i}\leq x\leq x_{2i},y_{1i}\leq y\leq y_{2i}\}\cap R_{z,w})}{\lambda(R_{z,w})}dQ(z,w)\leq b_{i},i=1,\ldots,n\nonumber\\
&  \text{all the integrands are measurable}\nonumber
\end{align}
where $R_{z,w}$ is the closed OU set represented by $R_{z,w}=\{(x,y)\in\mathcal{D}_{0}:y_{0}\leq y\leq h^*(x;z,w)\}$ with $h^*$ defined in (\ref{12n+3_step_function}). In view of the form of $h^*$ and the representation of $S=\{(x,y)\in\mathcal{D}_{0}:y\geq g(x)\}$ for some known function $g:[x_{0},\infty)\mapsto\lbrack y_{0},\infty]$, we can obtain
\[
\lambda(R_{z,w})=\sum_{i=1}^{12n+3}\sum_{j=1}^{12n+4-i}z_{i}w_{j},\quad R_{z,w}^{Y}=\sum_{i=1}^{12n+3}w_{i},\quad R_{z,w}^{X}=\sum_{i=1}^{12n+3}z_{i},
\]
\[
\lambda(S\cap R_{z,w})=\sum_{i=1}^{12n+3}\int\left(  y_{0}+\sum_{j=1}^{12n+4-i}w_{j}-g(x)\right)  _{+}I\left(  x_{0}+\sum_{j=1}^{i-1}z_{j}<x\leq x_{0}+\sum_{j=1}^{i}z_{j}\right)  dx,
\]
and
\begin{align*}
& \lambda(\{(x,y):x_{1i}\leq x\leq x_{2i},y_{1i}\leq y\leq y_{2i}\}\cap R_{z,w})\\
& =\sum_{i=1}^{12n+3}\int I\left(  \sum_{j=1}^{i-1}z_{j}<x-x_{0}\leq\sum_{j=1}^{i}z_{j},x_{1k}\leq x\leq x_{2k}\right)  \\
& \times\left(  \min\left(  y_{0}+\sum_{j=1}^{12n+4-i}w_{j},y_{2k}\right)-\min\left(  y_{0}+\sum_{j=1}^{12n+4-i}w_{j},y_{1k}\right)  \right)  dx.
\end{align*}
Plugging them into problem (\ref{change_index}) and noticing that ``all the integrands are measurable" automatically holds, we get the moment problem (\ref{finite_dimensional_moment_problem}).
\end{proof}

\begin{proof}[Proof of Theorem \ref{main_theorem}.]
(1): By Corollary \ref{cor_general_OU_optimization_2D} and Proposition \ref{equiv_moment_problem}, we know the optimal value of the OU-DRO problem (\ref{DRO_problem_2D}) is not greater than the optimal value of the moment problem (\ref{finite_dimensional_moment_problem}). On the other hand, any feasible solution to the moment problem (\ref{finite_dimensional_moment_problem}) can be transformed into a feasible solution to the OU-DRO problem (\ref{DRO_problem_2D}) with the same objective value by means of the OU distribution $Q$ defined in (\ref{construct_feasible_solution}). Combining two directions together, we know that optimal values of the OU-DRO problem (\ref{DRO_problem_2D}) and the moment problem (\ref{finite_dimensional_moment_problem}) are the same. The equivalence of the moment problem (\ref{finite_dimensional_moment_problem}) and the non-linear optimization (\ref{nonlinear_problem}) is ensured by Theorem 3.2 in \cite{winkler1988extreme} and the conditions in Theorem 3.2 can be verified by Theorem 2.1 and Proposition 3.1 in the same paper. The optimality of $P^*$ is clear since it has the same objective value as the optimal solution $Q^*$.

(2): By part (1), the OU-DRO problem (\ref{DRO_problem_2D}) has the same optimal value as the non-linear optimization (\ref{nonlinear_problem}). Thus, when the constraint $\bar{F}(x_{0},y_{0})=c$ is replaced by $l_{\bar{F}}\leq\bar{F}(x_{0},y_{0})\leq u_{\bar{F}}$, it suffices to make $c$ vary in the interval $[l_{\bar{F}},u_{\bar{F}}]$ when solving the non-linear optimization (\ref{nonlinear_problem}). Thus, these two problems have the same optimal value. Finally, $P^*$ is optimal since its objective value is exactly the optimal value of both the non-linear optimization and the DRO problem.
\end{proof}

\begin{proof}[Proof of Lemma \ref{property_of_general_OU_set}.]
The proof of $\bar{K}$ and $\bar{K^{\circ}}$ being OU is similar to the proof of Lemma \ref{property_of_OU_set} and thus is omitted. To show $K$ is Lebesgue measurable and satisfies $\lambda(K^{\circ})=\lambda(K)=\lambda(\bar{K})$, it suffices to show $\lambda(\partial K)=0$. Without loss of generality, assume $x_0$ is the origin. Let $O_1,\ldots,O_{2^d}$ be the $2^d$ (closed) orthants of $\mathbb{R}^d$. We have
\begin{equation*}
K=\bigcup_{i=1}^{2^{d}}(K\cap O_{i})
\end{equation*}
and thus
\begin{equation*}
\partial K\subset\bigcup_{i=1}^{2^{d}}\partial(K\cap O_{i}).
\end{equation*}
Note that by suitable reflection, each $K\cap O_{i}$ can be transformed into an OU set about the origin on the first orthant $[0,\infty)^d$. Then by \cite{lavrivc1993continuity}, we have $\lambda(\partial(K\cap O_{i}))=0$ for each $i=1,\ldots,2^d$, which implies $\lambda(\partial K)=0$.
\end{proof}

\begin{proof}[Proof of Theorem \ref{Choquet_general_OU}.]
The first ``if and only if" claim can be proved by the similar arguments in the proof of Theorem \ref{Choquet_OU}. The second ``if and only if" claim is obvious. To prove the Choquet representation, we first note that
\[
P(B)=\sum_{1\leq i\leq2^{d},P(O_{i})>0}P(O_{i})P(B|O_{i})
\]
by the law of total probability. By Theorem \ref{Choquet_OU}, $P(B|O_{i})$ can be written as
\[
P(B|O_{i})=\int_{0}^{\infty}W_{\bar{C}_{s}^{i}}(B)g_{i}(s)ds,
\]
which justifies (\ref{equ_Choquet_general_OU}). Note that $\bar{C}_s^i$ is an OU set which is fully contained in $O_i$. To show (\ref{equ_Choquet_general_OU}) is indeed a Choquet representation, it suffices to show $W_{\bar{C}_{s}^{i}}$ is an extreme point in the class of OU distributions about the origin if $\lambda(\bar{C}_{s}^{i})>0$. This is proved in Lemma \ref{extreme_point_general_OU}.
\end{proof}

\begin{proof}[Proof of Lemma \ref{extreme_point_general_OU}.]
The proof is similar to the proof of Lemma \ref{extreme_point_OU} and thus is omitted.
\end{proof}

\begin{proof}[Proof of Theorem \ref{equivalence_POU}.]
Suppose there is a density $f(x)$ of $P$ such that for every $s>0$, the set%
\[
C_{s}=\{x\in\mathcal{D}_{0}:f(x)\geq s\}
\]
is $d^{\prime}$-POU about $x_{0}$. Write $W_{C_{s}}$ as the uniform distribution on $C_{s}$ and let $g(s)=\lambda(C_{s})\geq 0$. Notice that
\[
f(x)=\int_{0}^{f(x)}1ds=\int_{0}^{\infty}I(x\in C_{s})ds.
\]
Thus for any measurable set $B$,
\[
P(B)=\int_{B}f(x)dx=\int_{B}\int_{0}^{\infty}I(x\in C_{s})dsdx=\int_{0}^{\infty}\lambda(B\cap C_{s})ds=\int_{0}^{\infty}W_{C_{s}}(B)g(s)ds.
\]
Setting $B=\mathcal{D}_{0}$, we can see $P(B)=1=\int_{0}^{\infty}g(s)ds$ and thus $g(s)$ is a probability density on $(0,\infty)$. Since the probability measure with the density $g(s)$ must be the limit of a sequence of discrete probability measures with finite support, we can see $P$ is in the closed convex hull of the set of all uniform distributions on $d^{\prime}$-POU subsets of $\mathcal{D}_{0}$, i.e., $P$ is a $d^{\prime}$-POU distribution.

To prove the ``only if" part, we assume $P$ is $d^{\prime}$-POU about $x_0$. Suppose that $Q$ is the uniform distribution on a $d^{\prime}$-POU set $K\subset\mathcal{D}_{0}$ and $0<\lambda(K)<\infty$. For a point $x\in\mathbb{R}^{d}$ and $\delta>0$, we define the neighborhood $N_{\delta}(x)=\{y\in\mathbb{R}^{d}:\max_{1\leq i\leq d}|x_{i}-y_{i}|<\delta\}$ as in the proof of Theorem \ref{Choquet_OU}. For $x,x^{\prime}\in\mathcal{D}_{0}^{\circ}$ with $x_i\ge x_i^{\prime}$ for $i=1,\ldots,d^{\prime}$, $x_i=x_i^{\prime}$ for $i=d^{\prime}+1,\ldots,d$ and $0<\delta<\min_{1\leq i\leq d^{\prime}}(x_{i}^{\prime}-x_{i0})$, we have%
\[
y\in K\cap N_{\delta}(x)\Rightarrow y+x^{\prime}-x\in K\cap N_{\delta}(x^{\prime}).
\]
Since Lebesgue measure is translation invariant, we can get%
\[
\lambda(K\cap N_{\delta}(x^{\prime}))\geq\lambda(K\cap N_{\delta}(x)).
\]
Dividing by $\lambda(K)$, we have%
\begin{equation}
Q(N_{\delta}(x^{\prime}))\geq Q(N_{\delta}(x)).\label{inequality_of_Q_POU}%
\end{equation}
Clearly, this relation also holds under the convex combinations of the uniform distributions on $d^{\prime}$-POU sets. Since $P$ is $d^{\prime}$-POU about $x_{0}$, by definition, there is a sequence $\{Q_{m},m\geq1\}$ such that $Q_{m}$ converges weakly to $P$, where $Q_{m}$'s are the convex combinations of the uniform distributions on $d^{\prime}$-POU sets. Therefore, these $Q_{m}$'s satisfy (\ref{inequality_of_Q_POU}). Since $P$ has a density, say $f_{0}$, we have
\[
P(\partial N_{\delta}(x^{\prime}))=P(\partial N_{\delta}(x))=0.
\]
Weak convergence and (\ref{inequality_of_Q_POU}) imply that
\begin{equation}
P(N_{\delta}(x^{\prime}))\geq P(N_{\delta}(x)).\label{inequality_of_P_POU}%
\end{equation}
For $x\in\mathcal{D}_{0}^{\circ}$, we define
\[
f(x)=\limsup_{\delta\downarrow0}\frac{P(N_{\delta}(x))}{\lambda(N_{\delta}(x))}=\limsup_{\delta\downarrow0}\frac{P(N_{\delta}(x))}{(2\delta)^{d}}.
\]
By the Lebesgue differentiation theorem, $f(x)=f_{0}(x)$ a.e., which means $f(x)$ is also a density for $P$. Besides, (\ref{inequality_of_P_POU}) implies $f(x^{\prime})\geq f(x)$ for $x,x^{\prime}\in\mathcal{D}_{0}^{\circ}$ with $x_i\ge x_i^{\prime}$ for $i=1,\ldots,d^{\prime}$, $x_i=x_i^{\prime}$ for $i=d^{\prime}+1,\ldots,d$. For $x\in\partial\mathcal{D}_{0}$, we simply define%
\[
f(x)=\sup_{y\in\mathcal{D}_{0}^{\circ}}f(y).
\]
Then the density $f(x)$ is $d^{\prime}$-POU about $x_{0}$ on $\mathcal{D}_{0}$.

The equivalence of $C_{s}$ being $d^{\prime}$-POU about $x_{0}$ and $f$ being a $d^{\prime}$-POU density about $x_{0}$ on $\mathcal{D}_{0}$ is easy to see by definition.
\end{proof}

\begin{proof}[Proof of Theorem \ref{Choquet_POU}.]
From the proof of Theorem \ref{equivalence_POU}, we have
\begin{equation}
P(B)=\int_{0}^{\infty}\lambda(C_{s}\cap B)ds.\label{POU_representation}
\end{equation}
By Fubini's theorem, we have
\begin{align*}
\lambda(C_{s}\cap B)  & =\int_{\mathcal{D}_{0}}I(x\in C_{s}\cap B)dx_{1}\cdots dx_{d}\\
& =\int_{x_{d^{\prime}+1,0}}^{\infty}\cdots\int_{x_{d0}}^{\infty}\left(\int_{x_{10}}^{\infty}\cdots\int_{x_{d^{\prime}0}}^{\infty}I(x\in C_{s}\cap B)dx_{1}\cdots dx_{d^{\prime}}\right)  dx_{d^{\prime}+1}\cdots dx_{d}\\
& =\int_{x_{d^{\prime}+1,0}}^{\infty}\cdots\int_{x_{d0}}^{\infty}\lambda_{d^{\prime}}((C_{s}\cap B)_{x_{d^{\prime}+1},\ldots,x_{d}})dx_{d^{\prime}+1}\cdots dx_{d}\\
& =\int_{x_{d^{\prime}+1,0}}^{\infty}\cdots\int_{x_{d0}}^{\infty}\lambda_{d^{\prime}}(C_{s,x_{d^{\prime}+1},\ldots,x_{d}}\cap B_{x_{d^{\prime}+1},\ldots,x_{d}})dx_{d^{\prime}+1}\cdots dx_{d}\\
& =\int_{x_{d^{\prime}+1,0}}^{\infty}\cdots\int_{x_{d0}}^{\infty}W_{C_{s,x_{d^{\prime}+1},\ldots,x_{d}}}(B_{x_{d^{\prime}+1},\ldots,x_{d}})\lambda_{d^{\prime}}(C_{s,x_{d^{\prime}+1},\ldots,x_{d}})dx_{d^{\prime}+1}\cdots dx_{d}.
\end{align*}
Notice that for any set $K$ we have%
\[
K_{x_{d^{\prime}+1},\ldots,x_{d}}\times\{(x_{d^{\prime}+1},\ldots,x_{d})\}=K\cap\{y\in\mathbb{R}^{d}:y_{d^{\prime}+1}=x_{d^{\prime}+1},\ldots,y_{d}=x_{d}\}.
\]
Therefore%
\begin{align*}
W_{C_{s,x_{d^{\prime}+1},\ldots,x_{d}}}(B_{x_{d^{\prime}+1},\ldots,x_{d}})  & =W_{C_{s}\cap\{y\in \mathbb{R}^{d}:y_{d^{\prime}+1}=x_{d^{\prime}+1},\ldots,y_{d}=x_{d}\}}(B\cap\{y\in\mathbb{R}^{d}:y_{d^{\prime}+1}=x_{d^{\prime}+1},\ldots,y_{d}=x_{d}\})\\
& =W_{C_{s}\cap\{y\in \mathbb{R}^{d}:y_{d^{\prime}+1}=x_{d^{\prime}+1},\ldots,y_{d}=x_{d}\}}(B).
\end{align*}
It follows that $\lambda(C_{s}\cap B)$ can be represented as%
\begin{equation}
\lambda(C_{s}\cap B)=\int_{x_{d^{\prime}+1,0}}^{\infty}\cdots\int_{x_{d0}}^{\infty}W_{C_{s}\cap\{y\in \mathbb{R}^{d}:y_{d^{\prime}+1}=x_{d^{\prime}+1},\ldots,y_{d}=x_{d}\}}(B)\lambda_{d^{\prime}}(C_{s,x_{d^{\prime}+1},\ldots,x_{d}})dx_{d^{\prime}+1}\cdots dx_{d}.\label{representation_Cs_B}
\end{equation}
Plugging (\ref{representation_Cs_B}) into (\ref{POU_representation}), we get
\begin{equation}
P(B)=\int_{0}^{\infty}\int_{x_{d^{\prime}+1,0}}^{\infty}\cdots\int_{x_{d0}}^{\infty}W_{C_{s}\cap\{y\in \mathbb{R}^{d}:y_{d^{\prime}+1}=x_{d^{\prime}+1},\ldots,y_{d}=x_{d}\}}(B)g(s,x_{d^{\prime}+1},\ldots,x_{d})dx_{d^{\prime}+1}\cdots dx_{d}ds,\label{equ_Choquet_POU}%
\end{equation}
where $g(s,x_{d^{\prime}+1},\ldots,x_{d})=\lambda_{d^{\prime}}(C_{s,x_{d^{\prime}+1},\ldots,x_{d}})$. Letting $B=\mathcal{D}_{0}$, we can see $g(s,x_{d^{\prime}+1},\ldots,x_{d})$ is a probability density. Since only $W_{C_{s}\cap\{y\in \mathbb{R}^{d}:y_{d^{\prime}+1}=x_{d^{\prime}+1},\ldots,y_{d}=x_{d}\}}$ with $\lambda_{d^{\prime}}(C_{s,x_{d^{\prime}+1},\ldots,x_{d}})>0$ contributes to (\ref{equ_Choquet_POU}), in order to prove (\ref{equ_Choquet_POU}) is indeed a Choquet representation, it suffices to show $W_{C_{s}\cap\{y\in \mathbb{R}^{d}:y_{d^{\prime}+1}=x_{d^{\prime}+1},\ldots,y_{d}=x_{d}\}}$ with $\lambda_{d^{\prime}}(C_{s,x_{d^{\prime}+1},\ldots,x_{d}})>0$ is an extreme point in the class of $d^{\prime}$-POU distributions. This is proved in Lemma \ref{extreme_point_POU}.
\end{proof}

\begin{proof}[Proof of Lemma \ref{extreme_point_POU}.]
Suppose there exist two $d^{\prime}$-POU distributions $P_1$, $P_{2}$ and $\eta\in(0,1)$ s.t.%
\[
W_{K}=\eta P_{1}+(1-\eta)P_{2}.
\]
Since the support of $W_{K}$ is contained in $\{y\in\mathbb{R}^{d}:y_{d^{\prime}+1}=x_{d^{\prime}+1},\ldots,y_{d}=x_{d}\}$, the supports of $P_{1}$ and $P_{2}$ are also contained in $\{y\in\mathbb{R}^{d}:y_{d^{\prime}+1}=x_{d^{\prime}+1},\ldots,y_{d}=x_{d}\}$. Then we can reduce $W_{K},P_{1}$ and $P_{2}$ to OU distributions on the subspace $\{y\in\mathbb{R}^{d}:y_{d^{\prime}+1}=x_{d^{\prime}+1},\ldots,y_{d}=x_{d}\}$. Note that $W_{K}$ is reduce to $W_{K_{x_{d^{\prime}+1},\ldots,x_{d}}}$ on $\{x\in\mathbb{R}^{d^{\prime}}:x_{i}\geq x_{i0},i=1,\ldots,d^{\prime}\}$\ and $K_{x_{d^{\prime}+1},\ldots,x_{d}}$ is an OU set about $(x_{10},\ldots,x_{d^{\prime}0})$ on $\{x\in\mathbb{R}^{d^{\prime}}:x_{i}\geq x_{i0},i=1,\ldots,d^{\prime}\}$ with $\lambda_{d^{\prime}}(K_{x_{d^{\prime}+1},\ldots,x_{d}})>0$. By Lemma \ref{extreme_point_OU}, $W_{K_{x_{d^{\prime}+1},\ldots,x_{d}}}$ is an extreme point in the class of OU distributions about $(x_{10},\ldots,x_{d^{\prime}0})$ on $\{x\in\mathbb{R}^{d^{\prime}}:x_{i}\geq x_{i0},i=1,\ldots,d^{\prime}\}$. Therefore, the distributions of $P_{1}$ and $P_{2}$ on $\{y\in\mathbb{R}^{d}:y_{d^{\prime}+1}=x_{d^{\prime}+1},\ldots,y_{d}=x_{d}\}$ are the same as $W_{K_{x_{d^{\prime}+1},\ldots,x_{d}}}$, which means $P_{1}$ and $P_{2}$ are equal to $W_{k}$. Therefore, $W_{K}$ is an extreme point in the class of $d^{\prime}$-POU distributions about $x_{0}$ on $\mathcal{D}_{0}$.
\end{proof}

\begin{proof}[Proof of Theorem \ref{Choquet_1POU}.]
We use the similar arguments for star unimodality in \cite{dharmadhikari1988unimodality}. We first prove the ``if'' part. Suppose $X\overset{d}{=}(x_{10}+U(Z_{1}-x_{10}),Z_{2},\ldots,Z_{d})$. We need to prove $X$ is 1-POU. In fact, we only need to prove it when $(Z_{1},\ldots,Z_{d})$ is degenerate, i.e., $(Z_{1},\ldots,Z_{d})$ put mass 1 at a point $(z_{1},\ldots,z_{d})\in\mathcal{D}_{0}$. Then by taking the convex mixture and weak limit and noticing that 1-POU distribution is closed under these operations, we can see it holds for any $(Z_{1},\ldots,Z_{d})$. When $(Z_{1},\ldots,Z_{d})$ is degenerate at $(z_{1},\ldots,z_{d})\in\mathcal{D}_{0}$, $X\overset{d}{=}(x_{10}+U(z_{1}-x_{10}),z_{2},\ldots,z_{d})$ whose distribution is uniform on $[x_{10},z_{1}]\times\{(z_{2},\ldots,z_{d})\}$. Since $[x_{10},z_{1}]\times\{(z_{2},\ldots,z_{d})\}$ is a 1-POU set on $\mathcal{D}_{0}$, $X$ is clearly 1-POU. This proves the ``if'' part.

Now we consider the ``only if'' part. Suppose $X\equiv(X_{1},\ldots,X_{d})$ is the uniform distribution on a 1-POU set $K$ which can be written as
\[
K=\{x\equiv(x_{1},\ldots,x_{d})\in\mathcal{D}_{0}:x_{10}\leq x_{1}\leq g(x_{2},\ldots,x_{d})\},
\]
where $g(x_{2},\ldots,x_{d})>x_{10}$ is continuous with domain $Dom(g)=[x_{20},\infty)\times\cdots\times\lbrack x_{d0},\infty)$ and $\lambda(K)\in(0,\infty)$. Then the density of $X$ is%
\[
f_{X}(x)=\frac{1}{\int_{Dom(g)}(g(x_{2},\ldots,x_{d})-x_{10})dx_{2}\cdots dx_{d}}I(x_{10}\leq x_{1}\leq g(x_{2},\ldots,x_{d})).
\]
Therefore, the conditional density of $X_{1}$ given $(X_{2},\ldots,X_{d})=(x_{2},\ldots,x_{d})$ is%
\[
f_{X_{1}|(X_{2},\ldots,X_{d})}(x_{1}|x_{2},\ldots,x_{d})=\frac{1}{g(x_{2},\ldots,x_{d})-x_{10}}I(x_{10}\leq x_{1}\leq g(x_{2},\ldots,x_{d})).
\]
Let $U=(X_{1}-x_{10})/(g(X_{2},\ldots,X_{d})-x_{10})$. Then we can see the conditional distribution of $U$ given $(X_{2},\ldots,X_{d})=(x_{2},\ldots,x_{d})$ is the uniform distribution on $(0,1)$ and thus is independent of $(X_{2},\ldots,X_{d})$. So $X$ can be represented by
\begin{align*}
X  &  \equiv(X_{1},\ldots,X_{d})=(x_{10}+U(g(X_{2},\ldots,X_{d})-x_{10}),X_{2},\ldots,X_{d})\\
&  =(x_{10}+U(Z_{1}-x_{10}),Z_{2},\ldots,Z_{d}),
\end{align*}
where $(Z_{1},Z_{2},\ldots,Z_{d})=(g(X_{2},\ldots,X_{d}),X_{2},\ldots,X_{d})$ and $U$ is independent of $(Z_{1},Z_{2},\ldots,Z_{d})$. By taking the convex mixture and weak limit, we can see ``only if'' part holds for any $X$.

The representation (\ref{equation_Choquet_1POU}) is just a reformulation of $X=(x_{10}+U(Z_{1}-x_{10}),Z_{2},\ldots,Z_{d})$ by conditioning on the values of $(Z_{1},\ldots,Z_{d})$. By Lemma \ref{extreme_point_POU}, we know $W_{\text{1-POU}}(z)$ is an extreme point in the class of 1-POU distributions about $x_{0}$ on $\mathcal{D}_{0}$. So (\ref{equation_Choquet_1POU}) is indeed a Choquet representation. Next we will show $Q$ is uniquely determined by $P$, i.e., the Choquet representation (\ref{equation_Choquet_1POU}) is unique. We write $\varphi_{X}$ and $\varphi_{Z}$ as the characteristic function of $X$ and $Z$ respectively. For $t\in\mathbb{R}^{n}$, we have%
\begin{align*}
\varphi_{X}(t) &  =E[e^{it^{\top}X}]=E[e^{it^{\top}(x_{10}+U(Z_{1}-x_{10}),Z_{2},\ldots,Z_{d})}]=\int_{0}^{1}E[e^{it^{\top}(x_{10}+u(Z_{1}-x_{10}),Z_{2},\ldots,Z_{d})}]du\\
&  =\int_{0}^{1}e^{it_{1}x_{10}(1-u)}\varphi_{Z}(ut_{1},t_{2},\ldots,t_{d})du=e^{it_{1}x_{10}}\int_{0}^{1}e^{-it_{1}ux_{10}}\varphi_{Z}(ut_{1},t_{2},\ldots,t_{d})du.
\end{align*}
So for $y>0$, we have%
\[
\varphi_{X}(yt_{1},t_{2},\ldots,t_{d})=e^{iyt_{1}x_{10}}\int_{0}^{1}e^{-it_{1}yux_{10}}\varphi_{Z}(yut_{1},t_{2},\ldots,t_{d})du
\]
By change of variable $v=uy$, we have%
\begin{align}
&  \varphi_{X}(yt_{1},t_{2},\ldots,t_{d})=\frac{e^{iyt_{1}x_{10}}}{y}\int_{0}^{y}e^{-it_{1}vx_{10}}\varphi_{Z}(vt_{1},t_{2},\ldots,t_{d})dv\nonumber\\
&  \Leftrightarrow\frac{y}{e^{iyt_{1}x_{10}}}\varphi_{X}(yt_{1},t_{2},\ldots,t_{d})=\int_{0}^{y}e^{-it_{1}vx_{10}}\varphi_{Z}(vt_{1},t_{2},\ldots,t_{d})dv.\label{relation_CF}%
\end{align}
Since the right hand side of (\ref{relation_CF}) is differentiable with respect to $y$, so is the left hand side. Then we have%
\begin{align*}
& \frac{d}{dy}\left[  \frac{y}{e^{iyt_{1}x_{10}}}\varphi_{X}(yt_{1},t_{2},\ldots,t_{d})\right]  =e^{-it_{1}yx_{10}}\varphi_{Z}(yt_{1},t_{2},\ldots,t_{d})\\
& \Leftrightarrow e^{it_{1}yx_{10}}\frac{d}{dy}\left[  \frac{y}{e^{iyt_{1}x_{10}}}\varphi_{X}(yt_{1},t_{2},\ldots,t_{d})\right] =\varphi_{Z}(yt_{1},t_{2},\ldots,t_{d}).
\end{align*}
Letting $y=1$, we can see $\varphi_{X}$ determines $\varphi_{Z}$ and hence $P$ determines $Q$.

Finally, suppose $P(X_{1}=x_{10})=0$, we need to prove $X_{1}$ is an absolutely continuous random variable. Note that $X_{1}\overset{d}{=}x_{10}+U(Z_{1}-x_{10})$ so we have $P(Z_{1}=x_{10})=0$. Consider any measurable set $A$ with $\lambda(A)=0$ and write $Q_{Z_{1}}$ as the probability distribution of $Z_{1}$. We have%
\begin{align*}
P(X_{1}  & \in A)=P(x_{10}+U(Z_{1}-x_{10})\in A)=\int_{(x_{10},\infty)}P(x_{10}+U(z_{1}-x_{10})\in A)dQ_{Z_{1}}(z_{1})\\
& =\int_{(x_{10},\infty)}0dQ_{Z_{1}}(z_{1})=0,
\end{align*}
which means $X_{1}$ is absolutely continuous. Now we prove $E_Q[I(Z_{1}\geq x)/(Z_{1}-x_{10})],x\geq x_{10}$ is a probability density of $X_{1}$. For any $x_{1}>x_{10}$, we have%
\begin{align*}
& P(X_{1}\geq x_{1})=P(x_{10}+U(Z_{1}-x_{10})\geq x_{1})=\int_{[x_{1},\infty)}P(x_{10}+U(z_{1}-x_{10})\geq x_{1})dQ_{Z_{1}}(z_{1})\\
& =\int_{[x_{1},\infty)}\frac{z_{1}-x_{1}}{z_{1}-x_{10}}dQ_{Z_{1}}(z_{1})=\int_{[x_{1},\infty)}\int_{[x_{1},z_{1}]}\frac{1}{z_{1}-x_{10}}dxdQ_{Z_{1}}(z_{1})\\
& =\int_{[x_{1},\infty)}\int_{[x,\infty)}\frac{1}{z_{1}-x_{10}}dQ_{Z_{1}}(z_{1})dx=\int_{[x_{1},\infty)}E_Q\left[  \frac{I(Z_{1}\geq x)}{Z_{1}-x_{10}}\right]  dx.
\end{align*}
Since this holds for any $x_{1}>x_{10}$, we know $E_Q[I(Z_{1}\geq x)/(Z_{1}-x_{10})],x\geq x_{10}$ is indeed a probability density of $X_{1}$. Moreover, by Monotone Convergence Theorem and $P(Z_{1}=x_{10})=0$,%
\[
\lim_{x\downarrow x_{10}}E_Q\left[  \frac{I(Z_{1}\geq x)}{Z_{1}-x_{10}}\right]=E_Q\left[  \lim_{x\downarrow x_{10}}\frac{I(Z_{1}\geq x)}{Z_{1}-x_{10}}\right]=E_Q\left[  \frac{I(Z_{1}>x_{10})}{Z_{1}-x_{10}}\right]  =E_Q\left[  \frac{I(Z_{1}\geq x_{10})}{Z_{1}-x_{10}}\right]  ,
\]
which means $E_Q[I(Z_{1}\geq x)/(Z_{1}-x_{10})]$ is continuous at $x=x_{10}$.
\end{proof}

\begin{proof}[Proof of Theorem \ref{reduction_1POU-DRO}.]
Without loss of generality suppose $P$ is the probability measure induced by $X$. Notice that we have $\bar{F}(x_{0})=c$ in (\ref{1POU_DRO}). Therefore, $P/c$ is a 1-POU probability measure on $\mathcal{D}_{0}$. By Theorem \ref{Choquet_1POU}, we have the representation
\begin{equation}
P\equiv c\frac{P}{c}=c\int_{\mathcal{D}_{0}}W_{\text{1-POU}}(z)dQ(z) \label{scaled_Choquet_1POU}%
\end{equation}
for some probability measure $Q$ on $\mathcal{D}_{0}$. We write $Z=(Z_{1},\ldots,Z_{d})$ as a random vector with the distribution $Q$ and write $U$ as the uniform distribution on $(0,1)$ independent of $Z$. Then $(x_{10}+U(Z_{1}-x_{10}),Z_{2},\ldots,Z_{d})$ has the probability distribution $P/c$ on $\mathcal{D}_{0}$.

We first consider the reduction of the objective function. Since $S$ is in the form of (\ref{form_S_1POU}), by (\ref{scaled_Choquet_1POU}), we have%
\begin{align*}
&  P((X_{1},\ldots,X_{d})\in S)\equiv P(S)=c\int_{\mathcal{D}_{0}}\frac{\min(g_{2}(z_{2},\ldots,z_{d}),z_{1})-\min(g_{1}(z_{2},\ldots,z_{d}),z_{1})}{z_{1}-x_{10}}dQ(z)\\
&  =cE_{Q}\left[  \frac{\min(g_{2}(Z_{2},\ldots,Z_{d}),Z_{1})-\min(g_{1}(Z_{2},\ldots,Z_{d}),Z_{1})}{Z_{1}-x_{10}}\right]  .
\end{align*}
So the objective function is expressed by the expectation under $Q$.

Next we consider the density constraint $f_{X_{1}}(x_{10})\leq u_{X_{1}}$. Since $f_{X_{1}}$ is the marginal density of $X_{1}$ within $\mathcal{D}_{0}$, we can see $f_{X_{1}}/c$ is the probability density of $x_{10}+U(Z_{1}-x_{10})$. Recall that we require $f_{X_{1}}$ to be continuous at $x=x_{10}$. So the value $f_{X_{1}}(x_{10})$ is uniquely determined and is actually given by Theorem \ref{Choquet_1POU}, i.e.,%
\[
f_{X_{1}}(x_{10})=c\frac{f_{X_{1}}(x_{10})}{c}=cE_{Q}\left[  \frac{I(Z_{1}\geq x_{10})}{Z_{1}-x_{10}}\right]  =cE_{Q}\left[  \frac{1}{Z_{1}-x_{10}}\right]  ,
\]
where the last inequality follows from the fact that $Z_{1}$ takes values in $[x_{10},\infty)$. Therefore, the density constraint can be rewritten as%
\[
E_{Q}\left[  \frac{1}{Z_{1}-x_{10}}\right]  \leq\frac{u_{X_{1}}}{c}.
\]
Moreover, this constraint also implies that $Q(Z_{1}=x_{10})=0$ for any feasible probability measure $Q$. Therefore, $P(X_{1}=x_{10})=P(x_{10}+U(Z_{1}-x_{10})=x_{10})=0$. In other words, the constraint that $P(X_{1}=x_{10})=0$ is already involved in the density constraint.

Then we consider the moment constraint%
\[
a_{i}\bar{F}(x_{0})\leq P(\underline{x}_{ji}\leq X_{j}\leq\bar{x}_{ji},j=1,\ldots,d)\leq b_{i}\bar{F}(x_{0}).
\]
By (\ref{scaled_Choquet_1POU}) and $\bar{F}(x_{0})=c$, it can be represented by%
\[
a_{i}\leq\int_{\mathcal{D}_{0}}\frac{\min(z_{1},\bar{x}_{1i})-\min(z_{1},\underline{x}_{1i})}{z_{1}-x_{10}}I(\underline{x}_{ji}\leq z_{j}\leq\bar{x}_{ji},j\geq2)dQ(z)\leq b_{i},
\]
i.e.,%
\[
a_{i}\leq E_{Q}\left[  \frac{\min(Z_{1},\bar{x}_{1i})-\min(Z_{1},\underline{x}_{1i})}{Z_{1}-x_{10}}I(\underline{x}_{ji}\leq Z_{j}\leq\bar{x}_{ji},j\geq2)\right]  \leq b_{i}.
\]

Finally, by Theorem \ref{Choquet_1POU}, the representation (\ref{scaled_Choquet_1POU}) is an ``if and only if'' condition for $P$ being 1-POU about $x_{0}$ on $\mathcal{D}_{0}$. Hence, the moment problem (\ref{moment_problem_1POU-DRO}) is equivalent to the 1-POU-DRO problem (\ref{1POU_DRO}). Moreover, the optimal solution of the two problems are also related by the representation given in Theorem \ref{Choquet_1POU}. This completes our proof.
\end{proof}

\begin{proof}[Proof of Corollary \ref{cor_reduction_1POU-DRO}.]
This follows directly from Theorem 3.2 in \cite{winkler1988extreme}.
\end{proof}
\end{appendix}
\end{document}